\documentclass[journal]{IEEEtran}
\ifCLASSINFOpdf
\else
\fi
\usepackage{graphicx}
\usepackage{amsmath}
\usepackage{graphics}
\usepackage{amsfonts}
\usepackage{amssymb}%
\usepackage{amsthm}
\usepackage{hyperref}
\usepackage{epsfig}
\usepackage{psfrag}
\usepackage{pgfpages}
\usepackage{subfigure}
\usepackage{color}
\usepackage{listings}
\usepackage{booktabs}
\usepackage{tabu}
\usepackage{multirow}

\newtheorem{Thm}{Theorem}
\newtheorem{Lemm}{Lemma}

\newtheorem{Propos}{Proposition}
\newtheorem{Coro}{Corollary}
\newtheorem{Assump}{Assumption}
\newtheorem{Def}{Definition}
\newtheorem{Remark}{Remark}

\newcommand{\Min}{{\mathop{\mathrm{minimize}}}}

\newcommand{\tabincell}[2]{\begin{tabular}{@{}#1@{}}#2\end{tabular}}

\begin{document}

%
\title{On Nonconvex Decentralized Gradient Descent}
%
%
%

\author{Jinshan~Zeng
        and~Wotao~Yin
\thanks{J. Zeng is with the College of Computer Information Engineering, Jiangxi Normal University, Nanchang, Jiangxi 330022, China (email: jsh.zeng@gmail.com)}
\thanks{W. Yin is with the Department of Mathematics, University of California, Los Angeles, CA 90095, USA (email: wotaoyin@ucla.edu).}
}

\maketitle

\begin{abstract}
Consensus optimization has received considerable attention in recent years. A number of decentralized algorithms have been proposed for {convex} consensus optimization. However, to the behaviors or consensus \emph{nonconvex} optimization, our understanding is more limited.

When we lose convexity, we cannot hope our algorithms always return global solutions though they sometimes still do.
Somewhat surprisingly, the decentralized consensus algorithms, DGD and Prox-DGD, retain most other properties that are known in the convex setting. In particular, when diminishing (or constant) step sizes are used, we can prove convergence to a (or a neighborhood of) consensus stationary solution under some regular assumptions.
It is worth noting that Prox-DGD can handle nonconvex nonsmooth functions if their proximal operators can be computed. Such functions include SCAD, MCP and $\ell_q$ quasi-norms, $q\in[0,1)$. Similarly, Prox-DGD can take the constraint to a nonconvex set with an easy projection.

To establish these properties, we have to introduce a completely different line of analysis, as well as modify existing proofs that were used in the convex setting.

%
\end{abstract}


%
\IEEEpeerreviewmaketitle

\section{Introduction}
%
%
%
%
\IEEEPARstart{W}{e} consider an undirected, connected network of $n$ agents and the following consensus optimization problem defined on the network:
\begin{equation}
\label{Eq:multi-agentOPT}
\mathop{\Min}_{x\in \mathbb{R}^p} f(x) \triangleq \sum_{i=1}^n f_i(x),
\end{equation}
where $f_i$ is a differentiable function only known to the agent $i$. We also consider the consensus optimization problem in the following differentiable+proximable\footnote{{We call a function proximable if its \emph{proximal operator} $\mathrm{prox}_{\alpha f}(y)\triangleq \mathop{\mathrm{argmin}}_{x} \left\{\alpha f(x)+\tfrac{1}{2}\|x-y\|^2 \right\}$ is easy to compute.}}  form:
\begin{align}
\label{Eq:multi-agentCompOPT}
\mathop{\Min}_{x\in \mathbb{R}^p}\  s(x) \triangleq \sum_{i=1}^n (f_i(x)+r_i(x)),
\end{align}
where $f_i, r_i$ are differentiable and proximable functions, respectively,  only known to the agent $i$. Each function $r_i$ is  possibly non-differentiable or nonconvex, or both.

The models (\ref{Eq:multi-agentOPT}) and \eqref{Eq:multi-agentCompOPT} find applications in decentralized averaging, learning, estimation, and control. Some typical applications include:
(i) the distributed compressed sensing problems \cite{Duarte-DCS2005,Ling-DCS2010,Mateos-DistSparLinearReg2010,Eldar2014,Ravazzi-DIHT2015};
(ii) distributed consensus \cite{Zhu2013}, \cite{Chang-InexactADMM2015}, \cite{Lee-DistRandProj2013,Leung2013,Scutari2014,Wai-Regression2015};
(iii) distributed and parallel machine learning \cite{Scutari2014,Scutari-Bigdata2015,Hong-Prox-PDA2017,Liu-D-PSGD2017,Omidshafiei-RL2017}.
More specifically, in these applications,
each $f_i$ can be: 1) the data-fidelity term (possibly nonconvex) in statistical learning and machine learning \cite{Scutari-Bigdata2015,Wai2016a};
2) nonconvex utility functions used in applications such as resource allocation \cite{Bjornson-RS2012,Hong-RS2013};
3) empirical risk of deep neural networks with nonlinear activation functions \cite{Hinton-DL2015}.
The proximal function $r_i$ can be taken as: 1) convex penalties such as nonsmooth $\ell_1$-norm or smooth $\ell_2$-norm;
2) the indicator function for a closed convex set (or a nonconvex set with an easy projection) \cite{Bianchi2013},  that is, $r_i(x)=0$ if $x$ satisfies the constraint and $\infty$ otherwise;
3) nonconvex penalties such as $\ell_q$ quasi-norm ($0\leq q<1$) \cite{Ravazzi-DIHT2015,Chen-Lowerbound-Lq}, smoothly clipped absolute deviation (SCAD) penalty \cite{Fan-SCAD2001} and the minimax concave penalty (MCP) \cite{Zhang-MCP2010}.

When $f_i$'s are convex, the existing algorithms include the (sub)gradient methods \cite{Chen2012a,Chen2012b,Jakovetic-FastGradient2014,Matei-Subgradient2011,Nedic-Subgradient2009,Tsitsiklis1986,Yin-DGD2013,Qu-Li2016},
and the primal-dual domain methods such as the decentralized alternating direction method of multipliers (DADMM) \cite{Schizas-DADMM2008,Shi-DADMM2014,Chang-InexactADMM2015}, DLM \cite{Ling-DLM2015}, and EXTRA \cite{Yin-EXTRA2015,Shi-PGEXTRA2015}.
When $f_i$'s are nonconvex,
some existing results include \cite{Bianchi2013,Bianchi2013b,Hong2016,Lorenzo2016a,Lorenzo2016b,Tatarenko2016conf,Tatarenko2016arXiv,Wai2016b,Wai2015,Wai2016a,Zhu2013}. In spite of the algorithms and their analysis in these works, the convergence of the simple algorithm Decentralized Gradient Descent (DGD) \cite{Nedic-Subgradient2009} under nonconvex $f_i$'s is still unknown. Furthermore, although DGD is slower than DADMM, DLM and EXTRA on convex problems, DGD is simpler and thus easier to extend to a variety of settings such as \cite{Raginsky11,Yan-13,McMahan-14,Hosseini16}, where online processing and delay tolerance are considered. Therefore, we expect our results to motivate future adoptions of nonconvex DGD.

This paper studies the convergence of two algorithms: DGD for solving problem \eqref{Eq:multi-agentOPT} and Prox-DGD for problem \eqref{Eq:multi-agentCompOPT}. In each DGD iteration, every agent locally computes a gradient and then updates its variable by combining the average of its neighbors' with the negative gradient step. In each Prox-DGD iteration, every agent locally computes a gradient of $f_i$ and a proximal map of $r_i$, as well as exchanges information with its neighbors. Both algorithms can use either a fixed step size or a sequence of decreasing step sizes.

When the problem is convex and a fixed step size is used, DGD does not converge to a solution of the original problem \eqref{Eq:multi-agentOPT} but a point in its neighborhood
\cite{Yin-DGD2013}. This motivates the use of decreasing step sizes such as in \cite{Chen2012b,Jakovetic-FastGradient2014}.
Assuming $f_i$'s are convex and have Lipschitz continuous and bounded gradients, \cite{Chen2012b} shows that decreasing step sizes $\alpha_k = \frac{1}{\sqrt{k}}$ lead to a convergence rate ${\cal O}(\frac{\ln k}{k})$ of the running best of objective errors. 
\cite{Jakovetic-FastGradient2014} uses nested loops and shows an outer-loop convergence rate ${\cal O}(\frac{1}{k^2})$ of objective errors, utilizing Nesterov's acceleration, provided that the inner loop performs substantial consensus computation. Without a substantial inner loop, their single-loop algorithm using the decreasing step sizes $\alpha_k = \frac{1}{k^{1/3}}$ has a reduced rate ${\cal O}(\frac{\ln k}{k})$.

The objective of this paper is two-fold: (a) we aim to show, other than losing global optimality, most existing convergence results of DGD and Prox-DGD that are known in the convex setting remain valid in the nonconvex setting, and (b) to achieve (a), we
illustrate how to tailor nonconvex analysis tools for decentralized optimization. In particular, our asymptotic exact and inexact consensus results require new treatments because they are special to decentralized algorithms.

The analytic results of this paper can be summarized as follows.
\begin{enumerate}
\item[(a)]
When a fixed step size $\alpha$ is used and properly bounded, the DGD iterates converge to a stationary point of a Lyapunov function. The difference between each local estimate of $x$ and the global average of all local estimates is bounded, and the bound is proportional to $\alpha$.

\item[(b)]
When a decreasing step size $\alpha_k = {\cal O}(1/(k+1)^{\epsilon})$ is used, where  $0<\epsilon\le 1$ and $k$ is the iteration number, the objective sequence converges, and the iterates of DGD are asymptotically consensual (i.e., become equal one another), and they achieve this at the rate of ${\cal O}(1/(k+1)^{\epsilon})$. 
Moreover, we show the convergence of DGD to a stationary point of the original problem, and derive the convergence rates of DGD with different $\epsilon$ for objective functions that are convex.

\item[(c)]
The convergence analysis of DGD can be extended to the algorithm Prox-DGD for solving problem \eqref{Eq:multi-agentCompOPT}. However, when the proximable functions $r_i$'s are nonconvex, the mixing matrix is required to be positive definite and a smaller step size is also required. (Otherwise, the mixing matrix can be non-definite.)
\end{enumerate}
The detailed comparisons between our results and the existing results on DGD and Prox-DGD are presented in Tables \ref{Tab:comp_DGD} and \ref{Tab:comp_prox-DGD}. The global objective error rate in these two tables refers to the rate of $\{f(\bar{x}^k) - f(x_{\mathrm{opt}})\}$ or $\{s(\bar{x}^k) - s(x_{\mathrm{opt}})\}$, where $\bar{x}^k = \frac{1}{n}\sum_{i=1}^n {\bf x}_{(i)}^k$ is the average of the $k$th iterate and $x_{\mathrm{opt}}$ is a global solution. The comparisons beyond DGD and Prox-DGD are presented in Section \ref{sc:relatedwork} and Table \ref{tab:main_results}.

New proof techniques are introduced in this paper, particularly, in the analysis of convergence of DGD and Prox-DGD with decreasing step sizes. Specifically, the convergence of objective sequence and convergence to a stationary point of the original problem with decreasing step sizes are justified via taking a Lyapunov function and several new lemmas (cf. Lemmas \ref{Lemm:supermartingale}, \ref{Lemm:betak}, and the proof of Theorem \ref{Thm:Lalphak}). Moreover, we estimate the consensus rate by introducing an auxiliary sequence and then showing both sequences have the same rates (cf. the proof of Proposition \ref{Propos:asympconsensus}). All these proof techniques are new and distinguish our paper from the existing works such as \cite{Chen2012b,Jakovetic-FastGradient2014,Nedic-Subgradient2009,Bianchi2013,Hong2016,Lorenzo2016a,Tatarenko2016arXiv,Wai2016a}.
It should be mentioned that during the revision of this paper, we found some recent, related but independent work on the convergence of nonconvex decentralized algorithms including \cite{Liu-D-PSGD2017,Hong-Prox-PDA2017,Hong-PP-PDA2017,Hong-zeroth-order2017}.
We will give detailed comparisons with these work latter.


The rest of this paper is organized as follows. Section \ref{sc:algorithm} describes the problem setup and reviews the algorithms. Section \ref{sc:result} presents our assumptions and main results. Section \ref{sc:relatedwork} discusses related works. Section \ref{sc:experiment} shows some numerical experiments to verify the developed theoretical results.
Section \ref{sc:proof} presents the proofs of our main results. We conclude this paper in Section \ref{sc:conclusion}. 

\textbf{Notation:} Let $I$ denote the identity matrix of the size $n\times n$, and ${\bf 1} \in \mathbb{R}^{n}$ denote the vector of all $1$'s.
For the matrix $X$, $X^T$ denotes its transpose, $X_{ij}$ denotes its $(i,j)$th component, and $\|X\| \triangleq \sqrt{\langle X, X \rangle}=\sqrt{\sum_{i,j}X_{ij}^2}$ is its Frobenius norm, which simplifies to the\ Euclidean norm when $X$ is a vector.
Given a symmetric, positive semidefinite matrix $G\in\mathbf{R}^{n\times n}$, we let $\|X\|_G^2 \triangleq \langle X, GX \rangle$ be the induced semi-norm. Given a function $h$, $\mathrm{dom}(f)$ denotes its domain.

\begin{table*}[ht!]
\tabulinesep=1.5mm
\centering
\caption{Comparisons on different algorithms for consensus smooth optimization problem \eqref{Eq:multi-agentOPT}}
\label{tab:main_results}
\begin{tabu}{|l|c|c|c|c|}\hline
& \multicolumn{2}{c|}{Fixed step size}          & \multicolumn{2}{c|}{Decreasing step sizes}                                               \\ \hline

algorithm & \multicolumn{1}{c|}{DGD \cite{Yin-DGD2013}} & DGD (this paper)  & \multicolumn{1}{c|}{D-NG \cite{Jakovetic-FastGradient2014}} & DGD (this paper)            \\ \hline

$f_i$ & convex only & (non)convex & convex only   & (non)convex
\\ \hline

$\nabla f_i$ & \multicolumn{2}{c|}{Lipschitz} & \multicolumn{2}{c|}{Lipschitz, bounded}
\\ \hline

step size & \multicolumn{2}{c|}{$ 0<\alpha< \frac{1+\lambda_n(W)}{L_f}$} & \multicolumn{1}{c|}{\tabincell{c}{${\cal O}(\frac{1}{k})$\\with Nesterov acc.}}  & \tabincell{c}{${\cal O}(\frac{1}{k^{\epsilon}})$\\[5pt] $\epsilon \in(0, 1]$
 }            \\ \hline

consensus & \multicolumn{2}{c|}{error ${\cal O}(\alpha)$}& \multicolumn{1}{c|}{${\cal O}(\frac{1}{k})$} & ${\cal O}(\frac{1}{k^{\epsilon}})$             \\ \hline

$\min_{j\le k}\|{\bf x}^{j+1}-{\bf x}^j\|^2$
& \multicolumn{2}{c|}{$o(\frac{1}{k})$}
& no rate & $o(\frac{1}{k^{1+\epsilon}})$             \\\hline

global objective error & \tabincell{l}{${\cal O}(\frac{1}{k})$ until error \\ ${\cal O}(\frac{\alpha}{1-\zeta})$} & \tabincell{l}{Convex: ${\cal O}(\frac{1}{k})$ until \\ error ${\cal O}(\frac{\alpha}{1-\zeta})$;\\ Nonconvex: no rate}   & \tabincell{l}{${\cal O}(\frac{\ln k}{k})$}  & \tabincell{l}{Convex$^\flat$: ${\cal O}(\frac{\ln k}{\sqrt{k}}) (\epsilon=1/2)$, \\ ${\cal O}(\frac{1}{\ln k}) (\epsilon=1)$,\\${\cal O} (\frac{1}{k^{\min\{\epsilon,1-\epsilon\}}})$(other $\epsilon$); \\ Nonconvex: no rate} \\\hline
\end{tabu}\\
$^\flat$The objective error rates of DGD and Prox-DGD obtained in this paper and those in convex DProx-Grad \cite{Chen2012b} are ergodic or running best rates.
\label{Tab:comp_DGD}
\end{table*}

\begin{table*}[ht!]
\tabulinesep=1.5mm
\centering
\caption{Comparisons on different algorithms for consensus composite optimization problem \eqref{Eq:multi-agentCompOPT}}
\label{tab:main_results}
\begin{tabu}{|l|c|c|c|c|c|}\hline
& \multicolumn{3}{c|}{Fixed step size}          & \multicolumn{2}{c|}{Decreasing step sizes}                                               \\ \hline
algorithm & \multicolumn{1}{c|}{AccDProx-Grad \cite{Chen2012a}} &\multicolumn{1}{c|}{DProx-Grad \cite{Chen2012b}} &Prox-DGD (this paper)  & \multicolumn{1}{c|}{DProx-Grad \cite{Chen2012b}} & Prox-DGD (this paper)             \\ \hline

$f_i, r_i$ & \multicolumn{2}{c|}{convex only} & (non)convex & convex only  & (non)convex
\\ \hline

$\nabla f_i$ & \multicolumn{2}{c|}{Lipschitz, bounded} & Lipschitz & \multicolumn{2}{c|}{Lipschitz, bounded}
\\ \hline

$\partial r_i$ & \multicolumn{2}{c|}{bounded} & -- & \multicolumn{2}{c|}{bounded}
\\ \hline

step size & \multicolumn{2}{c|}{$ 0<\alpha< \frac{1}{L_f}$} & \multicolumn{1}{c|}{\tabincell{l}{$0<\alpha < \frac{1+\lambda_n(W)}{L_f}$\\ (convex $r_i$); \\ $0<\alpha < \frac{\lambda_n(W)}{L_f}$\\ (nonconvex $r_i$, $\lambda_n(W)>0$)}} & \multicolumn{1}{c|}{${\cal O}(\frac{1}{(k+1)^{1/2}})$} & \tabincell{l}{${\cal O}(\frac{1}{(k+1)^{\epsilon}})$\\[5pt] $\epsilon\in(0,1]$}            \\ \hline

consensus & \multicolumn{1}{c|}{${\cal O}(\gamma^k k^2), 0<\gamma<1$} & \multicolumn{2}{c|}{error ${\cal O}(\alpha)$} & \multicolumn{1}{c|}{${\cal O}(\frac{1}{k^{1/2}})$} & ${\cal O}(\frac{1}{k^{\epsilon}})$             \\ \hline

$\min_{j\le k}\|{\bf x}^{j+1}-{\bf x}^j\|^2$
& \multicolumn{1}{c|}{no rate}
& no rate & $o(\frac{1}{k})$ & no rate & $o(\frac{1}{k^{1+\epsilon}})$             \\\hline

global objective error & \multicolumn{1}{c|}{${\cal O}(\frac{1}{k})$} & \multicolumn{1}{c|}{\tabincell{l}{Form \\${\frac{D_1}{\alpha}+D_2\alpha},$\\ $D_1,D_2>0$}}  & \tabincell{l}{Convex:\\ form ${\frac{D_3}{\alpha}+D_4\alpha},$\\ $D_3,D_4>0$; \\ Nonconvex: no rate} & ${\cal O}(\frac{\ln k}{k})^{\dag}$ & \tabincell{l}{Convex$^\dag$: ${\cal O}(\frac{\ln k}{\sqrt{k}}) (\epsilon=1/2)$, \\ ${\cal O}(\frac{1}{\ln k}) (\epsilon=1)$,\\${\cal O}(\frac{1}{k^{\min\{\epsilon,1-\epsilon\}}})$(other $\epsilon$), \\Nonconvex: no rate} \\\hline
\end{tabu}
\label{Tab:comp_prox-DGD}
\end{table*}

\begin{table*}[ht!]
\tabulinesep=1.5mm
\centering
\caption{Comparisons on scenarios applied for different nonconvex decentralized algorithms$^\natural$}
\label{tab:main_results}
\begin{tabu}{|l|c|c|c|c|c|c|c|c|c|c|c|}
\hline
& \multicolumn{1}{c|}{$f_i$}  & \multicolumn{2}{c|}{nonsmooth $r_i$} & \multicolumn{2}{c|}{step size} & \multicolumn{2}{c|}{network ($W$)} & \multicolumn{2}{c|}{algorithm type} & \multicolumn{2}{c|}{fusion scheme}
\\ \hline
algorithm & smooth &\multicolumn{1}{c|}{cvx} &ncvx  & \multicolumn{1}{c|}{fixed} & diminish & \multicolumn{1}{c|}{static} & dynamic & \multicolumn{1}{c|}{determin} & stochastic   & \multicolumn{1}{c|}{ATC} & CTA
\\ \hline
DGD (this paper) & $\surd$ &\multicolumn{1}{c|}{} &  & \multicolumn{1}{c|}{$\surd$} &$\surd$ & \multicolumn{1}{c|}{$\surd$ (doubly)} & $--$ & \multicolumn{1}{c|}{$\surd$} & $--$   & \multicolumn{1}{c|}{$--$} & $\surd$
\\ \hline
Perturbed Push-sum \cite{Tatarenko2016arXiv} & $\surd$ &\multicolumn{1}{c|}{} &  & \multicolumn{1}{c|}{$--$} &$\surd$ & \multicolumn{1}{c|}{$--$} & $\surd$ (column) & \multicolumn{1}{c|}{$\surd$} & $\surd$   & \multicolumn{1}{c|}{$--$} & $\surd$
\\ \hline
ZENITH \cite{Hong2016}& $\surd$ &\multicolumn{1}{c|}{} & & \multicolumn{1}{c|}{$\surd$} & $--$ & \multicolumn{1}{c|}{$\surd$ (doubly)} & $--$ & \multicolumn{1}{c|}{$\surd$} & $--$   & \multicolumn{1}{c|}{$--$} & $\surd$
\\ \hline \hline
Prox-DGD (this paper) & $\surd$ &\multicolumn{1}{c|}{$\surd$} &$\surd$  & \multicolumn{1}{c|}{$\surd$} &$\surd$ & \multicolumn{1}{c|}{$\surd$ (doubly)} &$--$    & \multicolumn{1}{c|}{$\surd$} & $--$   & \multicolumn{1}{c|}{$--$} & $\surd$
\\ \hline
NEXT \cite{Lorenzo2016a}& $\surd$ &\multicolumn{1}{c|}{$\surd$} &$--$  & \multicolumn{1}{c|}{$--$} & $\surd$ & \multicolumn{1}{c|}{$--$} & $\surd$ (doubly) & \multicolumn{1}{c|}{$\surd$} & $--$   & \multicolumn{1}{c|}{$\surd$} & $--$
\\ \hline
DeFW \cite{Wai2016a} &$\surd$ &\multicolumn{1}{c|}{$\surd$} &$--$  & \multicolumn{1}{c|}{$--$} & $\surd$ & \multicolumn{1}{c|}{$\surd$ (doubly)} & $--$ & \multicolumn{1}{c|}{$\surd$} &$--$   & \multicolumn{1}{c|}{$\surd$} &$--$
\\ \hline
Proj SGD \cite{Bianchi2013}&$\surd$ &\multicolumn{1}{c|}{$\surd$} &$--$  & \multicolumn{1}{c|}{$--$} &$\surd$ & \multicolumn{1}{c|}{$--$} &$\surd$ (row) & \multicolumn{1}{c|}{$--$} & $\surd$   & \multicolumn{1}{c|}{$\surd$} &$--$
\\ \hline
\end{tabu}
\leftline{$^\natural$ In this table, the full names of these abbreviations are list as follows: cvx (convex), ncvx (nonconvex), diminish (diminishing), determin (deterministic),}
\leftline{ATC (adaptive-then-combine), CTA (combine-then-adaptive), doubly (doubly stochastic), column (column stochastic), row (row stochastic), where vocabularies}
\leftline{in the brackets are the full names. A row, or column, or double stochastic $W$ means that: $W{\bf 1}={\bf 1}$, or $W^T{\bf 1} = {\bf 1}$, or both hold.}
\label{Tab:comp_scenario}
\end{table*}

\section{Problem setup and algorithm review}
\label{sc:algorithm}
Consider a connected \emph{undirected} network ${\cal G} = \{{\cal V},{\cal E}\}$, where ${\cal V}$ is a set of $n$ nodes and ${\cal E}$ is the edge set.
Any edge $(i,j)\in {\cal E}$ represents a communication link between nodes $i$ and  $j$. Let $x_{(i)} \in \mathbb{R}^p$ denote the \emph{local copy} of $x$ at node $i$. 
We reformulate the consensus problem \eqref{Eq:multi-agentOPT} into the \textbf{equivalent problem}: 
\begin{align}
&\mathop{\Min}_{\mathbf{x}}~\  {\bf 1}^T {\bf f(x)} \triangleq \sum_{i=1}^n f_i({\bf x}_{(i)}),\label{Eq:consensusProblem}\\
&\mathrm{subject\ to}~\ {\bf x}_{(i)} = {\bf x}_{(j)},~ \forall (i,j)\in {\cal E}, \nonumber
\end{align}
where ${\bf x}\in \mathbb{R}^{n\times p}$, ${\bf f(x)} \in \mathbb{R}^{n}$ with
$$
{\bf x} \triangleq \left(
\begin{array}{ccc}
\mbox{---} &{\bf x}^T_{(1)} & \mbox{---}\\
\mbox{---} &{\bf x}^T_{(2)} & \mbox{---}\\
  &\vdots &  \\
\mbox{---} &{\bf x}^T_{(n)} & \mbox{---}\\
\end{array}
\right),
~\
{\bf f(x)} \triangleq \left(
\begin{array}{c}
f_1({\bf x}_{(1)})\\
f_2({\bf x}_{(2)})\\
\vdots   \\
f_n({\bf x}_{(n)})\\
\end{array}
\right)
.
$$
In addition, the gradient of ${\bf f}(\bf x)$ is
\begin{align}
\label{Eq:grad-f}
{\bf \nabla f(x)} \triangleq \left(
\begin{array}{ccc}
\mbox{---} &\nabla f_1({\bf x}_{(1)})^T & \mbox{---}\\
\mbox{---} &\nabla f_2({\bf x}_{(2)})^T & \mbox{---}\\
  &\vdots &  \\
\mbox{---} &\nabla f_n({\bf x}_{(n)})^T & \mbox{---}\\
\end{array}
\right)
\in \mathbb{R}^{n\times p}.
\end{align}
The $i$th rows of the  matrices $\mathbf{x}$ and $\nabla \mathbf{f}(\mathbf{x})$,  and vector
${\bf f(x)}$, correspond to agent $i$. The analysis in this paper applies to any integer $p \ge 1$. \textbf{For simplicity, one can let $p=1$ and treat $\mathbf{x}$ and $\nabla \mathbf{f}(\mathbf{x})$ as vectors (rather than matrices).}


The \textbf{algorithm DGD} \cite{Nedic-Subgradient2009} for \eqref{Eq:consensusProblem} is described as follows:
\begin{quote}
Pick an arbitrary ${\bf x}^0$. For $k=0,1,\ldots,$ compute
\begin{align}
\label{Eq:DGD}
{\bf x}^{k+1} \gets W{\bf x}^k - \alpha_k \nabla {\bf f}({\bf x}^k),
\end{align}
where $W$ is a mixing matrix and $\alpha_k>0$ is a step-size parameter.
\end{quote}

Similarly, we can reformulate the composite problem \eqref{Eq:multi-agentCompOPT} as the following equivalent form:
\begin{align}
\label{Eq:CompositeProblem}
&\mathop{\Min}_{\mathbf{x}}\  \sum_{i=1}^n (f_i({\bf x}_{(i)})+r_i({\bf x}_{(i)})),\nonumber\\
&\mathrm{subject\  to}\ {\bf x}_{(i)} = {\bf x}_{(j)}, \ \forall (i,j)\in {\cal E}.
\end{align}
Let $r({\bf x}) \triangleq \sum_{i=1}^n r_i({\bf x}_{(i)}).$
The algorithm Prox-DGD  can be applied to the above problem \eqref{Eq:CompositeProblem}: \begin{quote}
\textbf{Prox-DGD:} Take an arbitrary ${\bf x}^0$. For $k=0,1,\ldots,$ perform
\begin{align}
\label{Eq:proximalDGD}
{\bf x}^{k+1} \gets \mathrm{prox}_{\alpha_k r}(W{\bf x}^k - \alpha_k \nabla {\bf f}({\bf x}^k)),
\end{align}
where the proximal operator is
\begin{equation}\label{eq:prox_def}
\mathrm{prox}_{\alpha_k r}({\bf x}) \triangleq \mathop{\mathrm{argmin}}_{{\bf u}\in \mathbb{R}^{n\times p}} \left\{\alpha_k r({\bf u})+\frac{\|{\bf u}-{\bf x}\|^2}{2}\right\}.
\end{equation}
\end{quote}

\section{Assumptions and main results}
\label{sc:result}
This section presents all of our main results.
\subsection{Definitions and assumptions}
\begin{Def}[Lipschitz differentiability]
\label{Def:LipDiff}
A function $h$ is called Lipschitz differentiable if $h$ is differentiable and its gradient $\nabla h$ is Lipschitz continuous, i.e., $\|\nabla h(u)-\nabla h(v)\| \leq L\|u-v\|, \forall u,v \in \mathrm{dom}(h),$ where $L>0$ is its Lipschitz constant.
\end{Def}

\begin{Def}[Coercivity]
\label{Def:coercive}
A function $h$ is called coercive if $\|u\| \rightarrow +\infty$ implies $h(u) \rightarrow +\infty$.
\end{Def}

{The next definition is a property that many functions have (see \cite[Section 2.2]{Xu-Yin2013} for examples) and can help obtain whole sequence convergence\footnote{Whole sequence convergence from any starting point is referred to as ``global convergence'' in the literature. Its limit is not necessarily a global solution.} from subsequence convergence.
}
\begin{Def}[Kurdyka-{\L}ojasiewicz (K{\L}) property \cite{lojasiewicz1993geometrie,bolte2007lojasiewicz,Attouch2013}]
\label{Def:KLProp}
A function $h:\mathbb{R}^p \rightarrow \mathbb{R}\cup \{+\infty\}$ has the K{\L} property at $x^*\in \mathrm{dom}(\partial h)$ if there exist $\eta \in (0,+\infty]$, a neighborhood $U$ of $x^*$, and a continuous concave function $\varphi:[0,\eta)\rightarrow \mathbb{R}_{+}$ such that:
    \begin{enumerate}
    \item[(i)]
    $\varphi(0) = 0$ and $\varphi$ is differentiable on $(0,\eta)$;

    \item[(ii)]
    for all $s\in (0,\eta)$, $\varphi'(s)>0$;

    \item[(iii)]
    for all $x$ in $U\cap \{x: h(x^*)<h(x)<h(x^*) + \eta\}$,
    the K{\L} inequality holds
    \begin{equation}
    \varphi'\big(h(x)-h(x^*)\big)\cdot \mathrm{dist}\big(0,\partial h(x)\big) \geq 1.
    \label{KLIneq}
    \end{equation}
    \end{enumerate}
Proper lower semi-continuous functions that satisfy the K{\L} inequality at each point of $\mathrm{dom}(\partial h)$ are called K{\L} functions.
\end{Def}


\begin{Assump}[Objective]
\label{Assump:objective}
The objective functions $f_i:\mathbb{R}^p \rightarrow \mathbb{R}\cup \{+\infty\}$, $i=1,\ldots,n$, satisfy the following:
\begin{enumerate}
\item[(1)]
$f_i$ is Lipschitz differentiable with constant $L_{f_i}>0$.

\item[(2)]
 $f_i$ is proper (i.e., not everywhere infinite) and coercive. 
\end{enumerate}
\end{Assump}
The sum $\sum_{i=1}^n f_i({\bf x}_{(i)})$ is $L_f$-Lipschitz differentiable with $L_f \triangleq \max_i L_{f_i}$
(this can be easily verified via the definition of $\nabla {\bf f}({\bf x})$ as shown in \eqref{Eq:grad-f}).
In addition, each $f_i$ is lower bounded following Part (2) of the above assumption.

\begin{Assump}[Mixing matrix]
\label{Assump:MixMat}
The mixing matrix $W = [w_{ij}] \in \mathbb{R}^{n\times n}$ has the following properties:
\begin{enumerate}
\item[(1)]
(Graph) If $i\neq j$ and $(i,j) \notin {\cal E}$, then $w_{ij} =0$, otherwise, $w_{ij} >0$.

\item[(2)]
(Symmetry) $W = W^T$.

\item[(3)]
(Null space property) $\mathrm{null} \{I-W\} = \mathrm{span}\{\bf 1\}$.

\item[(4)]
(Spectral property) $I \succeq W \succ -I.$
\end{enumerate}
\end{Assump}
By Assumption \ref{Assump:MixMat}, a  solution ${\bf x}_{\mathrm{opt}}$ to problem \eqref{Eq:consensusProblem} satisfies $(I-W){\bf x}_{\mathrm{opt}} =0.$ Due to the symmetric assumption of $W$, its eigenvalues are real and can be sorted in the nonincreasing order.
Let  $\lambda_i(W)$ denote the $i$th largest eigenvalue of $W$. Then by Assumption \ref{Assump:MixMat}, 
\[\lambda_1(W) =1> \lambda_2(W) \geq \cdots \geq \lambda_n(W)>-1.\]
Let $\zeta$ be the second largest magnitude eigenvalue of $W$. Then
\begin{align}
\label{Eq:zeta}
\zeta = \max\{|\lambda_2(W)|, |\lambda_n(W)|\}.
\end{align}

\subsection{Convergence results of DGD}
We consider the convergence of DGD with both a fixed step size and a sequence of decreasing step sizes.

\subsubsection{Convergence results of DGD with a fixed step size}
The  convergence result of DGD with a fixed step size (i.e., $\alpha_k \equiv \alpha$) is established based on the Lyapunov function \cite{Yin-DGD2013}:
\begin{align}
\label{Def:L-alpha}
{\cal L}_{\alpha}({\bf x}) \triangleq {\bf 1}^T{\bf f}({\bf x}) + \frac{1}{2\alpha} \|{\bf x}\|_{I-W}^2.
\end{align}
It is worth reminding that convexity is \emph{not} assumed.
\begin{Thm}[Global convergence]
\label{Thm:Globalconverg}
Let $\{{\bf x}^k\}$ be the sequence generated by DGD \eqref{Eq:DGD} with the step size $0<\alpha<\frac{1+\lambda_n(W)}{L_f}$.
Let Assumptions \ref{Assump:objective} and \ref{Assump:MixMat} hold. Then  $\{{\bf x}^k\}$  has at least one accumulation point ${\bf x}^*$, and any such point is a stationary point of ${\cal L}_{\alpha}({\bf x})$.
Furthermore, the running best rates\footnote{Given a nonnegative sequence ${a_k}$, its running best sequence is $b_k=\min\{a_i : i\le k\}$. We say ${a_k}$ has a running best rate of $o(1/k)$ if $b_k=o(1/k)$.} of the sequences\footnote{These quantities naturally appear in the analysis, so we keep the squares.} $\{\|{\bf x}^{k+1} - {\bf x}^k\|^2\}$, and $\{\|\nabla {\cal L}_{\alpha}({\bf x}^k)\|^2\}$, and
$\{\|\frac{1}{n}{\bf 1}^T\nabla {\bf f}({\bf x}^k)\|^2\}$ are  $o(\frac{1}{{k}})$.
The convergence rate of the sequence $\{\frac{1}{K} \sum_{k=0}^{K-1} \|\frac{1}{n}{\bf 1}^T \nabla {\bf f}({\bf x}^k)\|^2\}$ is
${\cal O}(\frac{1}{K})$.

In addition, if ${\cal L}_{\alpha}$ satisfies the K{\L} property at an accumulation point ${\bf x}^*$, then $\{{\bf x}^k\}$ globally converges to ${\bf x}^*$.
\end{Thm}

\begin{Remark}
Let ${\bf x}^*$ be a stationary point of ${\cal L}_{\alpha}({\bf x})$, and thus
\begin{align}
\label{Eq:optcond-remark}
0 = \nabla {\bf f}({\bf x}^*) + \alpha^{-1} (I-W){\bf x}^*.
\end{align}
Since ${\bf 1}^T(I-W)=0$, \eqref{Eq:optcond-remark} yields
$0 = {\bf 1}^T \nabla {\bf f}({\bf x}^*),$
indicating that ${\bf x}^*$ is also a stationary point to the separable function $\sum_{i=1}^n f_i({\bf x}_{(i)})$.
Since the rows of ${\bf x}^*$ are not necessarily identical, we \emph{cannot} say ${\bf x}^*$ is  a stationary point to Problem \eqref{Eq:consensusProblem}. However, the differences between the rows of ${\bf x}^*$  are bounded, following our next result below adapted from \cite{Yin-DGD2013}: 
\end{Remark}
\begin{Propos}[Consensual bound on ${\bf x}^*$]
\label{propos:consensualbound}
For each iteration $k$, define $\bar{x}^k\triangleq\frac{1}{n}\sum_{i=1}^n {\bf x}_{(i)}^k.$
Then, it holds  for each node $i$ that
\begin{equation}\label{eq:consensualbound}\|{\bf x}_{(i)}^k-\bar{x}^k\|\le\frac{\alpha D}{1-\zeta},
\end{equation}
where $D$ is a universal bound of $\|{\bf \nabla f(x^k)}\|$ defined in Lemma \ref{Lemm:bound_grad_f} below, $\zeta$ is the second largest magnitude eigenvalue of $W$ specified in \eqref{Eq:zeta}. As $k\to\infty$, \eqref{eq:consensualbound} yields the consensual bound
$$\|{\bf x}_{(i)}^*-\bar{x}^*\|\le\frac{\alpha D}{1-\zeta},$$
where $\bar{x}^* \triangleq \frac{1}{n}\sum_{i=1}^n {\bf x}_{(i)}^*.$
\end{Propos}
In Proposition \ref{propos:consensualbound}, the consensual bound is proportional to the step size $\alpha$ and inversely proportional to the gap between the largest and the second largest magnitude eigenvalues of $W$.

Let us compare the DGD iteration with the iteration of \emph{centralized gradient descent} \eqref{eq:graddesc} for $f(x)$. Averaging the rows of \eqref{Eq:DGD} yields the following comparison:
\begin{align}\label{eq:dgdavg}
\text{DGD averaged:}\quad \bar{x}^{k+1}&\gets \bar{x}^k -\alpha \bigg(\frac{1}{n}\sum_{i=1}^n \nabla f_{i}({\bf x}_{(i)}^k)\bigg).\\
\text{Centralized:}\quad \bar{x}^{k+1}&\gets \bar{x}^k -\alpha \bigg(\frac{1}{n}\sum_{i=1}^n \nabla f_{i}(\bar{x}^k)\bigg).\label{eq:graddesc}
\end{align}
Apparently, DGD approximates centralized gradient descent by evaluating $\nabla f_{(i)}$ at local variables ${\bf x}_{(i)}^k$ instead of the global average. We can estimate the error of this approximation as
\begin{align*}
&\|\frac{1}{n}\sum_{i=1}^n \nabla f_{i}({\bf x}_{(i)}^k)-\frac{1}{n}\sum_{i=1}^n \nabla f_{i}(\bar{x}^k)\| \\
&\le\frac{1}{n}\sum_{i=1}^n\|\nabla f_{i}({\bf x}_{(i)}^k)-\nabla f_{i}(\bar{x}^k)\|\le \frac{\alpha DL_f}{1-\zeta}.
\end{align*}
Unlike the convex analysis in \cite{Yin-DGD2013}, it is impossible to bound the difference between the sequences of \eqref{eq:dgdavg} and \eqref{eq:graddesc} without convexity because  the two sequences may converge to different stationary points of ${\cal L}_{\alpha}$.
\begin{Remark}
The K{\L} assumption on ${\cal L}_{\alpha}$ in Theorem \ref{Thm:Globalconverg} can be satisfied if each $f_i$ is a sub-analytic function. Since $\|{\bf x}\|_{I-W}^2$ is obviously sub-analytic and the sum of two sub-analytic functions remains sub-analytic,  ${\cal L}_{\alpha}$ is sub-analytic if each $f_i$ is so. See  \cite[Section 2.2]{Xu-Yin2013} for more details and examples.
\end{Remark}

\begin{Propos}[K{\L} convergence rates]
\label{Propos:conv-rate}
Let the assumptions of Theorem {\ref{Thm:Globalconverg}} hold. Suppose that ${\cal L}_{\alpha}$ satisfies the K{\L} inequality at an accumulation point ${\bf x}^*$
with $\psi(s) = c s^{1-\theta}$ for some constant $c>0$. Then, the following convergence rates hold:
\begin{enumerate}
\item[(a)]
If $\theta = 0,$ ${\bf x}^k$ converges to ${\bf x}^*$ in finitely many iterations.

\item[(b)]
If $\theta \in (0,\frac{1}{2}]$, $\|{\bf x}^k - {\bf x}^*\| \leq C_0 \tau^k$ for all $k\geq k^*$ for some $k^*>0, C_0>0, \tau \in [0,1)$.

\item[(c)]
If $\theta \in (\frac{1}{2},1)$, $\|{\bf x}^k - {\bf x}^*\| \leq C_0k^{-(1-\theta)/(2\theta -1)}$ for all $k\geq k^*$, for certain $k^*>0, C_0>0$.
\end{enumerate}
\end{Propos}
Note that the rates in parts (b) and (c) of Proposition \ref{Propos:conv-rate} are of the \emph{eventual} type.

Using fixed step sizes, our results are limited because the stationary point ${\bf x}^*$ of ${\cal L}_{\alpha}$ is not a stationary point of the original problem. We only have a consensual bound on ${\bf x}^*$. To address this issue, the next subsection uses decreasing step sizes and presents better convergence results.

\subsubsection{Convergence of DGD with decreasing step sizes}

The positive consensual error bound in Proposition \ref{propos:consensualbound}, which is proportional to the constant step size $\alpha$, motivates the  use of properly decreasing step sizes $\alpha_k = {\cal O}(\frac{1}{(k+1)^{\epsilon}})$, for some $0<\epsilon \leq 1$, to  diminish the consensual bound to $0$.
As a result, any accumulation point ${\bf x}^*$ becomes a stationary point of the original problem \eqref{Eq:consensusProblem}. To analyze DGD with decreasing step sizes, we  add the following assumption.

\begin{Assump}[Bounded gradient]
\label{Assump:boundedgradient}
For any $k$, $\nabla {\bf f}({\bf x}^k)$ is uniformly bounded by some constant $B>0$, i.e., $\|\nabla {\bf f}({\bf x}^k)\|\leq B.$
\end{Assump}

Note that the bounded gradient assumption is a regular assumption in the convergence analysis of decentralized gradient methods (see, \cite{Bianchi2013,Bianchi2013b,Hong2016,Lorenzo2016a,Lorenzo2016b,Tatarenko2016conf,Tatarenko2016arXiv,Wai2016b,Wai2016a} for example), even in the convex setting \cite{Jakovetic-FastGradient2014} and also \cite{Chen2012b}, though it is  not required for centralized gradient descent.



{We take the step size sequence:
\begin{align}\label{eq:decralpha}
\alpha_k = \frac{1}{L_f(k+1)^{\epsilon}}, \quad 0<\epsilon \le 1,
\end{align}
throughout the rest part of this section. (The numerator 1 can be replaced by any positive constant.)
By iteratively applying iteration \eqref{Eq:DGD},} we obtain the following expression
\begin{align}
\label{Eq:recursion-xk}
{\bf x}^{k} = W^{k}{\bf x}^0 - \sum_{j=0}^{k-1} \alpha_j W^{k-1-j}\nabla {\bf f}({\bf x}^j).
\end{align}

\begin{Propos}[Asymptotic consensus rate]
\label{Propos:asympconsensus}
Let Assumptions \ref{Assump:MixMat} and \ref{Assump:boundedgradient} hold. Let DGD use \eqref{eq:decralpha}. Let ${\bar{\bf x}}^{k} \triangleq \frac{1}{n}{\bf 1}{\bf 1}^T {\bf x}^{k}$.  
Then,
$\|{\bf x}^{k} - {\bar{\bf x}}^{k}\|$ converges to 0 at the rate of ${\cal O}(1/(k+1)^{\epsilon})$.
\end{Propos}

According to Proposition \ref{Propos:asympconsensus}, the iterates of DGD with decreasing step sizes can reach consensus asymptotically (compared to a nonzero bound in the fixed step size case  in Proposition \ref{propos:consensualbound}). 
Moreover, with a larger $\epsilon$, faster decaying step sizes generally imply a faster asymptotic consensus rate.
Note that $(I-W)\bar{\bf x}^k =0$ and thus $\|{\bf x}^k\|_{I-W}^2 = \|{\bf x}^k - {\bar{\bf x}}^{k}\|_{I-W}^2$. Therefore, the above proposition implies the following result.

\begin{Coro}
\label{Coro:conv-decreasingstepsize}
Apply the setting of Proposition \ref{Propos:asympconsensus}.
$\|{\bf x}^k \|_{I-W}^2$  converges to 0 at the rate of  ${\cal O}(1/(k+1)^{2\epsilon})$.
\end{Coro}
Corollary \ref{Coro:conv-decreasingstepsize} shows that the sequence $\{{\bf x}^k\}$ in the $(I-W)$ semi-norm can decay to 0 at a sublinear rate. For any \emph{global} consensual solution ${\bf x}_{\mathrm{opt}}$ to problem \eqref{Eq:consensusProblem}, we have $\|{\bf x}^k - {\bf x}_{\mathrm{opt}}\|_{I-W}^2 =\|{\bf x}^k\|_{I-W}^2$ so, if $\{{\bf x}^k\}$ does converge to ${\bf x}_{\mathrm{opt}}$, then their distance in the same semi-norm decays at ${\cal O}(1/k^{2\epsilon})$.


\begin{Thm}[Convergence]
\label{Thm:Lalphak}
Let Assumptions \ref{Assump:objective}, \ref{Assump:MixMat} and \ref{Assump:boundedgradient} 
hold. Let DGD use step sizes \eqref{eq:decralpha}. Then
\begin{enumerate}
\item[(a)]
$\{{\cal L}_{\alpha_k}({\bf x}^k)\}$ and $\{{\bf 1}^T {\bf f}({\bf x}^k)\}$ converge to the same limit;


\item[(b)]
$\mathop{\lim}_{k\rightarrow \infty} {\bf 1}^T \nabla {\bf f}({\bf x}^k)=0,$
and any limit point of $\{{\bf x}^k\}$ is a stationary point of problem \eqref{Eq:consensusProblem};

\item[(c)] In addition, if there exists an isolated accumulation point, then $\{{\bf x}^k\}$ converges.
\end{enumerate}
\end{Thm}

In the proof of Theorem \ref{Thm:Lalphak}, we will establish
\[\sum_{k=0}^{\infty} \big(\alpha_k^{-1}(1+\lambda_n(W))-L_f\big)\|{\bf x}^{k+1}-{\bf x}^k\|^2 <\infty,\]
which implies that the running best rate of the sequence $\{\|{\bf x}^{k+1}-{\bf x}^k\|^2\}$ is $o(1/k^{1+\epsilon})$.
Theorem \ref{Thm:Lalphak} shows that the objective sequence converges, and any limit point of $\{{\bf x}^k\}$  is a stationary point of the original problem. However, there is no result on the convergence rate of the objective sequence to an optimal value, and it is generally difficult to get such a rate without convexity.

Although our primary focus is nonconvexity, next we  assume convexity and present the objective convergence rate, which has an interesting relation with $\epsilon$.

For any ${\bf x} \in \mathbb{R}^{n\times p}$, let ${\bar f}({\bf x}) \triangleq \sum_{i=1}^n f_i({\bf x}_{(i)})$. Even if $f_i$'s are convex, the solution to \eqref{Eq:consensusProblem} may be non-unique. Thus, let ${\cal X}^*$ be the set of solutions to \eqref{Eq:consensusProblem}. Given ${\bf x}^k$, we pick the solution ${\bf x}_{\mathrm{opt}}= \mathrm{Proj}_{{\cal X}^*}({\bf x}^k) \in {\cal X}^*$. Also let $f_{\mathrm{opt}} = {\bar f}({\bf x}_{\mathrm{opt}})$ be the optimal value of \eqref{Eq:multi-agentOPT}. Define the ergodic objective:
\begin{align}
\label{Eq:def-bar-fk}
\bar{f}^K = \frac{\sum_{k=0}^K \alpha_k {\bar f}({\bar {\bf x}}^{k+1})}{\sum_{k=0}^K \alpha_k},
\end{align}
where $\bar{\bf x}^{k+1} = \frac{1}{n}({\bf 1}^T {\bf x}^{k+1}){\bf 1}$.
Obviously,
\begin{align}
\label{Eq:barf-K}
\bar{f}^K \geq \mathop{\min}_{k=1,\ldots,K+1} {\bar f}(\bar{\bf x}^k).
\end{align}

\begin{Propos}[Convergence rates under convexity]
\label{Propos:convergrate-dgd}
Let Assumptions \ref{Assump:objective}, \ref{Assump:MixMat} and \ref{Assump:boundedgradient} 
hold. Let DGD use step sizes \eqref{eq:decralpha}. If  $\lambda_n(W)>0$ and each $f_i$ is convex, then $\{{\bar f}^K\}$ defined in \eqref{Eq:def-bar-fk} converges to the optimal objective value $f_{\mathrm{opt}}$ at the following rates:
\begin{enumerate}
\item[(a)] if $0<\epsilon<1/2$, the rate is ${\cal O}(\frac{1}{K^{\epsilon}})$;

\item[(b)] if $\epsilon=1/2$, the rate is ${\cal O}(\frac{\ln K}{\sqrt{K}})$;

\item[(c)] if $1/2<\epsilon<1$, the rate is ${\cal O}(\frac{1}{K^{1-\epsilon}})$;

\item[(d)] if $\epsilon=1$, the rate is ${\cal O}(\frac{1}{\ln K})$.
\end{enumerate}
\end{Propos}


The convergence rates established in  Proposition \ref{Propos:convergrate-dgd} almost as good as ${\cal O}(\frac{1}{\sqrt{K}})$ when $\epsilon=\tfrac{1}{2}$. As $\epsilon$ goes to either 0 or 1, the rates become slower, and $\epsilon=1/2$ may be the optimal choice in terms of the convergence rate. However, by Proposition \ref{Propos:asympconsensus}, a larger $\epsilon$ implies a faster consensus rate. Therefore, there is a tradeoff to choose an appropriate $\epsilon$ in the practical implementation of DGD.


\begin{Remark}
\label{Remark:Push-sum}
A related algorithm is the perturbed push-sum algorithm, also called subgradient-push, which was proposed in \cite{Kempe2003} for average consensus problem over time-varying network.
Its convergence in the convex setting was developed in \cite{Nedic-Subgradientpush2015}.
Recently, its convergence (to a critical point) in the nonconvex setting was established in \cite{Tatarenko2016arXiv} under some regularity assumptions. Moreover, by utilizing perturbations on the update process and the assumption of no saddle-point existence, almost sure convergence to a local minimum of its perturbed variant was also shown in \cite{Tatarenko2016arXiv}.

\end{Remark}

\begin{Remark}
\label{Remark:D-PSGD}
Another recent algorithm is decentralized stochastic gradient descent (D-PSGD) in \cite{Liu-D-PSGD2017} with support to nonconvex large-sum objectives.
An ${\cal O}(\frac{1}{K} + \frac{1}{\sqrt{nK}})$-ergodic convergence rate was established assuming $K$ is sufficiently large and the step size $\alpha$ is sufficiently small. 
When applied to the setting of this paper, \cite[Theorem 1]{Liu-D-PSGD2017} implies that the sequence $\{\frac{1}{K} \sum_{k=0}^{K-1} \|\frac{1}{n}{\bf 1}^T \nabla {\bf f}({\bf x}^k)\|^2\}$ converges to zero at the rate ${\cal O}(\frac{1}{K})$ if the step size $0<\alpha < \frac{1-\zeta}{6L_f\sqrt{n}}$, where $\zeta$ is defined in \eqref{Eq:zeta}.
From Theorem \ref{Thm:Globalconverg}, we can also establish such an ${\cal O}(\frac{1}{K})$-ergodic convergence rate of DGD as long as $0<\alpha < \frac{1+\lambda_n(W)}{L_f}$.
Similar rates of convergence to a stationary point have also been shown for different nonconvex algorithms in \cite{Hong2016,Tatarenko2016arXiv,Lan-AccSGD2016}.
\end{Remark}

\subsection{Convergence results of Prox-DGD}
Similarly, we consider the convergence of Prox-DGD with both a fixed step size and decreasing step sizes.
The iteration \eqref{Eq:proximalDGD} can be reformulated as
\begin{align}
\label{Eq:proximalDGDequiv1}
{\bf x}^{k+1} =\mathrm{prox}_{\alpha_k r}({\bf x}^k - \alpha_k\nabla{\cal L}_{\alpha_k}({\bf x^k}))
\end{align}
based on which, we define the Lyapunov function
\begin{align*}
\hat{\cal L}_{\alpha_k}({\bf x}) \triangleq {\cal L}_{\alpha_k}({\bf x}) + r({\bf x}),
\end{align*}
where we recall ${\cal L}_{\alpha_k}({\bf x}) = \sum_{i=1}^n f_i({\bf x}_{(i)}) +  \frac{1}{2\alpha_k}\|{\bf x}\|_{I-W}^2$. Then \eqref{Eq:proximalDGDequiv1} is clearly the forward-backward splitting (a.k.a., prox-gradient) iteration for
$\mathop{\mathrm{minimize}}_{\bf x} \,\hat{\cal L}_{\alpha_k}({\bf x}).$
Specifically, \eqref{Eq:proximalDGDequiv1} first performs gradient descent to the differentiable function ${\cal L}_{\alpha_k}({\bf x})$ and then computes the proximal of  $r({\bf x})$.

To analyze  Prox-DGD, we should revise Assumption \ref{Assump:objective} 
as follows.
\begin{Assump}[Composite objective]
\label{Assump:objective-composite}
The objective function of \eqref{Eq:CompositeProblem} satisfies the following:
\begin{enumerate}
\item[(1)]
Each $f_i$ is Lipschitz differentiable with constant $L_{f_i}>0$.

\item[(2)]
Each $(f_i+r_i)$ is proper, lower semi-continuous, coercive.
\end{enumerate}
\end{Assump}
As before,  $\sum_{i=1}^n f_i({\bf x}_{(i)})$ is $L_f$-Lipschitz differentiable for $L_f \triangleq \max_i L_{f_i}.$

\subsubsection{Convergence results of Prox-DGD with a fixed step size}
Based on the above assumptions, we can get the global convergence of Prox-DGD as follows.
\begin{Thm}[Global convergence of Prox-DGD]
\label{Thm:Globalconverg-PGDGD}
Let $\{{\bf x}^k\}$ be the sequence generated by Prox-DGD \eqref{Eq:proximalDGD} where the step size $\alpha$ satisfies $0<\alpha<\frac{1+\lambda_n(W)}{L_f}$ when $r_i$'s are convex; and $0<\alpha<\frac{\lambda_n(W)}{L_f}$, when $r_i$'s are not necessarily convex (this case requires $\lambda_n(W)>0$). Let Assumptions \ref{Assump:MixMat} and \ref{Assump:objective-composite} hold.
Then  $\{{\bf x}^k\}$  has at least one accumulation point ${\bf x}^*$, and any accumulation point is a stationary point of $\hat{\cal L}_{\alpha}({\bf x})$.
Furthermore, the running best rates
of the sequences $\{\|{\bf x}^{k+1} - {\bf x}^k\|^2\}$, $\{\|{\bf g}^{k+1}\|^2\}$ and
$\{\|\frac{1}{n}{\bf 1}^T\nabla {\bf f}({\bf x}^k) + \frac{1}{n}{\bf 1}^T \xi^k\|^2\}$
(where ${\bf g}^{k+1}$ is defined in Lemma \ref{Lemm:boundedsubgradient-PGDGD}, and $\xi^k$ is defined in Lemma \ref{Lemm:recursion-xk-prox-dgd})   are $o(\frac{1}{k})$.
The convergence rate of the sequence $\{\frac{1}{K} \sum_{k=0}^{K-1} \|\frac{1}{n}{\bf 1}^T (\nabla {\bf f}({\bf x}^k)+\xi^k)\|^2\}$ is
${\cal O}(\frac{1}{K})$.

In addition, if $\hat{\cal L}_{\alpha}$ satisfies the K{\L} property at an accumulation point ${\bf x}^*$, then $\{{\bf x}^k\}$ converges to ${\bf x}^*$.
\end{Thm}
The rate of convergence of Prox-DGD can be also established by leveraging the 
K{\L} property. 
\begin{Propos}[Rate of convergence of Prox-DGD]
\label{Propos:conv-rate-prox-DGD}
Under assumptions of Theorem {\ref{Thm:Globalconverg-PGDGD}}, suppose that $\hat{\cal L}_{\alpha}$ satisfies the K{\L} inequality at an accumulation point $x^*$ with $\psi(s) = c_1 s^{1-\theta}$ for some constant $c_1>0$. Then the following hold:
\begin{enumerate}
\item[(a)]
If $\theta = 0,$ ${\bf x}^k$ converges to ${\bf x}^*$ in finitely many iterations.

\item[(b)]
If $\theta \in (0,\frac{1}{2}]$, $\|{\bf x}^k - {\bf x}^*\| \leq C_1 \tau^k$ for all $k\geq k^*$ for some $k^*>0, C_1>0, \tau \in [0,1)$.

\item[(c)]
If $\theta \in (\frac{1}{2},1)$, $\|{\bf x}^k - {\bf x}^*\| \leq C_1k^{-(1-\theta)/(2\theta -1)}$ for all $k\geq k^*$, for certain $k^*>0, C_1>0$.
\end{enumerate}
\end{Propos}

\subsubsection{Convergence of Prox-DGD with decreasing step sizes}

In Prox-DGD, we also use the decreasing step size \eqref{eq:decralpha}. To investigate its convergence, 
the bounded gradient Assumption \ref{Assump:boundedgradient} should be revised as follows.

\begin{Assump}[Bounded composite subgradient]
\label{Assump:boundedcompositesubgradient}
For each $i$, $\nabla f_i$ is uniformly bounded by some constant $B_i>0$, i.e., $\|\nabla f_i(x)\|\leq B_i$ for any $x\in \mathbb{R}^p$. Moreover, $\|\xi_i\|\leq B_{r_i}$ for any $\xi_i \in \partial r_i(x)$ and $x\in \mathbb{R}^p$, $i=1\ldots,n$.
\end{Assump}
Let $\bar{B} \triangleq \sum_{i=1}^n (B_i+B_{r_i})$. Then $\nabla {\bf f}({\bf x})+\xi$ (where $\xi \in \partial r({\bf x})$ for any ${\bf x}\in \mathbb{R}^{n\times p}$) is uniformly bounded by $\bar{B}$. Note that the same assumption is used to analyze the convergence of distributed proximal-gradient method in the convex setting \cite{Chen2012a,Chen2012b}, and also is widely used to analyze the convergence of nonconvex decentralized algorithms like in \cite{Lorenzo2016a,Lorenzo2016b}. In light of Lemma \ref{Lemm:recursion-xk-prox-dgd} below, the claims in Proposition \ref{Propos:asympconsensus} and Corollary \ref{Coro:conv-decreasingstepsize} also hold for Prox-DGD.
\begin{Propos}[Asymptotic consensus and rate]
\label{Propos:asympconsensus-prox-dgd}
Let Assumptions \ref{Assump:MixMat} and \ref{Assump:boundedcompositesubgradient} hold. In Prox-DGD, use the step sizes \eqref{eq:decralpha}.  There hold
\begin{align*}
\|{\bf x}^{k} - {\bar{\bf x}}^{k}\| \leq C \big(\|{\bf x}^0\|\zeta^{k}+\bar{B} \sum_{j=0}^{k-1}\alpha_j \zeta^{k-1-j}\big),
\end{align*}
and $\|{\bf x}^{k} - {\bar{\bf x}}^{k}\|$ converges to 0 at the rate of ${\cal O}(1/(k+1)^{\epsilon})$.
Moreover, let ${\bf x}^*$ be any \emph{global}  solution of the problem \eqref{Eq:CompositeProblem}. Then $\|{\bf x}^k - {\bf x}^*\|_{I-W}^2 =\|{\bf x}^k\|_{I-W}^2=\|{\bf x}^k - \bar{\bf x}^*\|_{I-W}^2$ converges to 0 at the rate of  ${\cal O}(1/(k+1)^{2\epsilon})$.
\end{Propos}

For any ${\bf x} \in \mathbb{R}^{n\times p}$, define ${\bar s}({\bf x}) = \sum_{i=1}^n f_i({\bf x}_{(i)}) + r_i({\bf x}_{(i)})$.
Let ${\cal X}^{\dag}$ be a set of solutions of \eqref{Eq:CompositeProblem}, ${\bf x}_{\mathrm{opt}}= \mathrm{Proj}_{{\cal X}^{\dag}}({\bf x}^k) \in {\cal X}^{\dag}$, and $s_{\mathrm{opt}} = {\bar s}({\bf x}_{\mathrm{opt}})$ be the optimal value of \eqref{Eq:CompositeProblem}.
Define
\begin{align}
\label{Eq:def-bar-sk}
\bar{s}^K = \frac{\sum_{k=0}^K \alpha_k {\bar s}({\bar {\bf x}}^{k+1})}{\sum_{k=0}^K \alpha_k}.
\end{align}

\begin{Thm}[Convergence and rate]
\label{Thm:hatLalphak}
Let Assumptions \ref{Assump:MixMat}, \ref{Assump:objective-composite} and \ref{Assump:boundedcompositesubgradient} hold. In Prox-DGD, use the step sizes \eqref{eq:decralpha}. Then
\begin{enumerate}
\item[(a)]
$\{\hat{\cal L}_{\alpha_k}({\bf x}^k)\}$ and $\{\sum_{i=1}^n f_i({\bf x}_{(i)}^k)+r_i({\bf x}_{(i)}^k)\}$ converge to the same limit;

\item[(b)]
$\sum_{k=0}^{\infty} \big(\alpha_k^{-1}(1+\lambda_n(W))-L_f\big)\|{\bf x}^{k+1}-{\bf x}^k\|^2 <\infty$ when $r_i$'s are convex; or, $\sum_{k=0}^{\infty} \big(\alpha_k^{-1}\lambda_n(W)-L_f\big)\|{\bf x}^{k+1}-{\bf x}^k\|^2 <\infty$ when $r_i$'s are not necessarily convex (this case requires $\lambda_n(W)>0$);

\item[(c)]
if $\{\xi^{k}\}$ satisfies $\|\xi^{k+1}-\xi^k\| \leq L_r \|{\bf x}^{k+1}-{\bf x}^k\|$ for each $k>k_0$, some constant $L_r>0$, and a sufficiently large integer $k_0>0$, then
\[\lim_{k\rightarrow \infty} {\bf 1}^T(\nabla {\bf f}({\bf x}^k)+{\xi}^{k+1})=0,\]
where ${\xi}^{k+1}\in \partial r({\bf x}^{k+1})$ is the one determined by the proximal operator \eqref{eq:prox_def}, and any limit point is a stationary point of problem \eqref{Eq:CompositeProblem}.

\item[(d)] in addition, if there exists an isolated accumulation point, then $\{{\bf x}^k\}$ converges.

\item[(e)] furthermore, if $f_i$ and $r_i$ are convex and $\lambda_n(W)>0$, then the claims on the rates of $\{{\bar f}^K\}$ in Proposition \ref{Propos:convergrate-dgd} hold for the sequence $\{{\bar s}^K\}$ defined in \eqref{Eq:def-bar-sk}.
\end{enumerate}
\end{Thm}

Theorem \ref{Thm:hatLalphak}(b) implies that the running best rate of $\|{\bf x}^{k+1}-{\bf x}^k\|^2$ is $o(\frac{1}{k^{1+\epsilon}})$.
The additional condition imposed on $\{\xi^k\}$ in Theorem \ref{Thm:hatLalphak}(c) is some type of restricted continuous regularity of the subgradient $\partial r$ with respect to the generated sequence.
If $\partial r$ is locally Lipschitz continuous in a neighborhood of a limit point, then such Lipschitz condition on $\{\xi^k\}$ can generally be satisfied, since $\{{\bf x}^k\}$ is asymptotic regular, and thus ${\bf x}^k$ will lies in such neighborhood of this limit point when $k$ is sufficiently large.
There are many kinds of proximal functions satisfying such assumption as studied in \cite{Zeng-IJT2016} (see, Remark \ref{Remark:Lip-subgrad} for detailed information).
Theorem \ref{Thm:hatLalphak}(e) gives the convergence rates of Prox-DGD in the convex setting.

\begin{Remark}
\label{Remark:Lip-subgrad}
A typical proximal function $r$ satisfying the assumption of Theorem \ref{Thm:hatLalphak} (c) is the $\ell_q$ quasi-norm ($0\leq q<1$) widely studied in sparse optimization, which takes the form
$r(x) = \sum_{i=1}^p |x_i|^q$\footnote{When $q=0$, we denote $0^0=0$.}.
From \cite{Chen-Lowerbound-Lq} and \cite{Zeng-IJT2016}, there is a positive lower bound for the absolute values of non-zero components of the solutions of $\ell_q$ regularized optimization problem.
Furthermore, as shown by \cite[Property 1(b)]{Zeng-IJT2016}, the sequence generated by Prox-DGD also has the similar lower bound property.
Moreover, by Theorem \ref{Thm:hatLalphak}(b), we have $\|{\bf x}^{k+1}-{\bf x}^k\|^2 \rightarrow 0$ as $k\rightarrow \infty$. Together with the lower bound property, we can easily obtain the finite support and sign convergence of $\{{\bf x}^k\}$, that is, the supports and signs of the non-zero components will freeze for sufficiently large $k$. When restricted to such nonzero subspace, the gradient of $r_i(u)=|u|^q$ is Lipschitz continuous for any $|u|\ge\tau$ and some positive constant $\tau$, where $\tau$ denotes the lower bound.
Besides $\ell_q$ quasi-norm, there are some other typical cases like SCAD \cite{Fan-SCAD2001} and MCP \cite{Zhang-MCP2010} widely used in statistical learning, satisfying the condition (c) of this theorem .
\end{Remark}

\begin{Remark}
\label{Remark:Proj-SGD}
One tightly related algorithm of Prox-DGD is the projected stochastic gradient descent (Proj SGD) method proposed by \cite{Bianchi2013} for solving the constrained multi-agent optimization problem with a convex constraint set.
When restricted to the deterministic case as studied in this paper,
the convergence results of Proj SGD are very similar to that of Prox-DGD (see, Theorem 4 (c)-(d) in this paper and \cite[Theorem 1]{Bianchi2013}).
However, there are some differences between \cite{Bianchi2013} and this paper.
In short, Proj SGD in \cite{Bianchi2013} uses \textbf{convex constraints}, which correspond to setting $r(x)$ in our paper as indicator functions of those convex sets. Our paper also considers \textbf{nonconvex functions} like $\ell_q$ quasi-norm ($0\leq q<1$), SCAD, and MCP, which are widely used in statistical learning.
Another difference is that Proj SGD of \cite{Bianchi2013} uses \textbf{adaptive-then-combine (ATC)} and Prox-DGD of this paper does \textbf{combine-then-adaptive (CTA)}.
By \cite[Assumption 2]{Bianchi2013},  Proj SGD uses decreasing step sizes like ${\cal O}(k^{-\epsilon})$ for some $\epsilon>1/2$.
We study the step size $\alpha_k = {\cal O}(k^{-\epsilon})$ for any $0<\epsilon \leq 1$ for Prox-DGD, as well as a fixed step size.
\end{Remark}

\section{Related works and discussions}
\label{sc:relatedwork}
We summarize some recent nonconvex decentralized algorithms in Table \ref{Tab:comp_scenario}. Most of them apply to either the smooth optimization problem \eqref{Eq:multi-agentOPT} or the composite optimization problem \eqref{Eq:multi-agentCompOPT} and use diminishing step sizes. Although \eqref{Eq:multi-agentOPT} is a special case of \eqref{Eq:multi-agentCompOPT} via letting $r_i(x)=0$, there are still differences in both algorithm design and theoretical analysis. Therefore, we divide their comparisons. 

We first discuss the algorithms for \eqref{Eq:multi-agentOPT}. In \cite{Tatarenko2016arXiv}, the authors proved the convergence of perturbed push-sum\footnote{The original form of this algorithm, push-sum, was proposed in \cite{Kempe2003} for the average consensus problem. It was modified and analyzed in \cite{Nedic-Subgradientpush2015} for convex consensus optimization problem over time-varying directed graphs.} for nonconvex \eqref{Eq:multi-agentOPT} under some regularity assumptions. They also 
introduced random perturbations to avoid local minima. The  network considered in \cite{Tatarenko2016arXiv} is time-varying and directed, and specific column stochastic matrices and diminishing step sizes are used. The convergence results for the deterministic perturbed push-sum algorithm obtained in \cite{Tatarenko2016arXiv} are similar to those of DGD developed in this paper under similar assumptions (see, Theorem \ref{Thm:Lalphak} above and
\cite[Theorem 3]{Tatarenko2016arXiv}). The detailed comparisons between two algorithms are illustrated in Remark \ref{Remark:Push-sum}.
In \cite{Liu-D-PSGD2017}, the sublinear convergence to a stationary point of D-PSGD algorithm was developed under the nonconvex setting. DGD studied in this paper can be viewed a special D-PSGD with a zero variance.
In \cite{Hong2016}, a primal-dual approximate gradient algorithm called {ZENITH} was developed for \eqref{Eq:multi-agentOPT}. 
The convergence of ZENITH was given in the expectation of constraint violation under the Lipschitz differentiable assumption and other assumptions.
The last one is the proximal primal-dual algorithm (Prox-PDA) recently proposed in \cite{Hong-Prox-PDA2017}. The ${\cal O}(\frac{1}{k})$-rate of convergence to a stationary point was established in \cite{Hong-Prox-PDA2017}. Latter, a perturbed variant of Prox-PDA was proposed in \cite{Hong-PP-PDA2017} for constrained composite (smooth+nonsmooth) optimization problem with a linear equality constraint.

Table \ref{Tab:comp_scenario} includes three algorithms for solving the composite problem \eqref{Eq:multi-agentCompOPT}, which are related to ours. All of them only deal with convex $r_i$ (whereas $r_i$ in this paper can also be nonconvex).
In \cite{Lorenzo2016b}, the authors proposed NEXT based on the previous successive convex approximation (SCA) technique. The iterates of NEXT include two stages, a local SCA stage to update local variables and a consensus update stage to fuse the information between agents. 
While NEXT has results similar to Prox-DGD using diminishing step sizes.
Another interesting algorithm is decentralized Frank-Wolfe (DeFW) proposed in \cite{Wai2016a} for nonconvex, smooth, constrained decentralized optimization, where a bounded convex constraint set is imposed. There are three steps at each iteration of DeFW: 
average gradient computation, local variable evaluation by Frank-Wolfe, 
and information fusion between agents. 
In \cite{Wai2016a}, the authors established convergence results similar to Prox-DGD under diminishing step sizes. 
The stochastic version of DeFW has also been developed in \cite{Wai2016b} for high-dimensional convex sparse optimization.
The next one is 
projected stochastic gradient algorithm (Proj SGD) \cite{Bianchi2013} for constrained, nonconvex, smooth consensus optimization with a convex constrained set. The detailed comparison between Proj SGD and Prox-DGD are shown in Remark \ref{Remark:Proj-SGD}.

Based on the above analysis, the convergence results of DGD and Prox-DGD with decreasing step sizes of this paper are comparable with most of the existing ones.
However, we allow nonconvex nonsmooth $r_i$ and are able to obtain the estimates of asymptotic consensus rates. We also establish global convergence
using a fixed step size 
while it is only found in ZENITH. 

\section{Numerical Experiments}
\label{sc:experiment}

In this section, we describe a set of numerical experiments mainly to verify our theoretical findings for DGD and Prox-DGD.
The comparisons between DGD (or Prox-DGD) and the other existing algorithms can be referred to these literature like \cite{Lorenzo2016a}, \cite{Hong-Prox-PDA2017}, \cite{Hong-PP-PDA2017}, \cite{Hong-zeroth-order2017}.

\subsection{Convergence of DGD}

We verify the performance of DGD using both fixed and diminishing step sizes in the experimental setting identical to \cite{Tatarenko2016arXiv}.
The following one dimensional decentralized optimization is considered in \cite{Tatarenko2016arXiv},
\[
\mathop{\mathrm{minimize}}_{x\in \mathbb{R}} f(x) = f_1(x)+f_2(x)+f_3(x),
\]
where
\begin{align*}
& f_1(x) = \left\{
\begin{array}{l}
(x^3 - 16x)(x+2), \quad \text{if} \ |x|\leq 10,\\
4248x - 32400, \quad  \text{if}\ x>10,\\
-3112x - 25040, \quad \text{if} \ x<-10,
\end{array}%
\right.\\
&f_2(x) = \left\{
\begin{array}{l}
(0.5x^3+x^2)(x-4), \quad \text{if} \ |x|\leq 10,\\
1620x-12600, \quad  \text{if}\ x>10,\\
-2220x-16600, \quad \text{if} \ x<-10,
\end{array}%
\right.\\
&f_3(x) = \left\{
\begin{array}{l}
(x+2)^2(x-4), \quad \text{if} \ |x|\leq 10,\\
288x-2016, \quad  \text{if}\ x>10,\\
288x-1984, \quad \text{if} \ x<-10.
\end{array}%
\right.
\end{align*}
We plot the function $f$ over the intervals $[-15,15]$ and $[-6,6]$ as shown in Fig. \ref{Fig:objfun}. The function achieves its global minimizer at $x=2.62$ and has a local minimizer at $x=-2.49$, as well as a local maximizer at $x = -1.12$.
It is easy to compute the Lipschitz constants of $\nabla f_1$, $\nabla f_2$ and $\nabla f_3$ as $L_1=1288$, $L_2=532$, and $L_3=60$, respectively. Thus, $L_f = \max_i L_i = 1288$, which is used in our theoretical analysis. Moreover, $\nabla f$ is obviously bounded.
We consider one of the following three connected networks:
\[
(1 \leftrightarrow 2, 2 \leftrightarrow 3), \quad (1 \leftrightarrow 3, 3 \leftrightarrow 2), \quad (2 \leftrightarrow 1, 1 \leftrightarrow 3).
\]
In the experiment, the mixing matrix $W$ is taken as
\[
W = \left(
\begin{array}{ccc}
\frac{1}{2} & 0 & \frac{1}{2} \\
0 &\frac{1}{2} & \frac{1}{2} \\
\frac{1}{2}  & \frac{1}{2} & 0
\end{array}
\right),
\]
which has the eigenvalues: 1, 0.5, -0.5. The mixing matrix $W$ is \emph{not} positive definite.
All the assumptions used in our theory are satisfied.

We test the performance of DGD with a theoretically fixed step size and several different kinds of decreasing step sizes, starting iterations from two different initial points:
\[
x_1^0:=(0,0,0), \quad \text{and} \quad x_2^0: = (-1, -1.2, -1.1).
\]
The second initialization is a ``dangerous'' point since it is close to local maximum and between two local minima (one of them is global).
The experiment results are reported in Fig. \ref{Fig:conv-dgd}.
From these figures, DGD successfully converges to the \textbf{global minimum} and achieves the consensus starting from both initial points though the objective function is nonconvex.

\begin{figure}[!t]
\begin{minipage}[b]{0.48\linewidth}
\centering
\includegraphics*[scale=0.3]{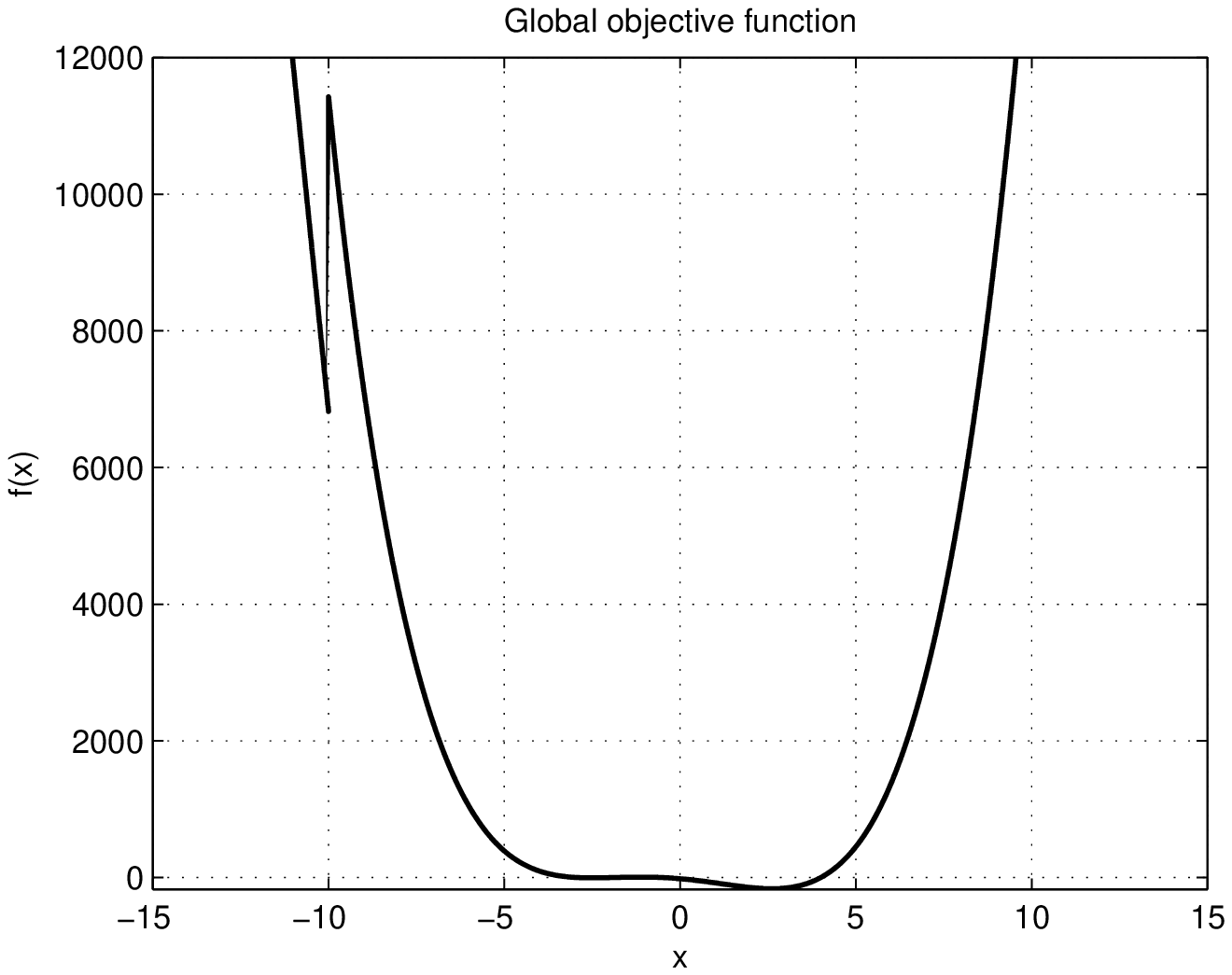}
\centerline{{\small (a) plot on $[-15,15]$}}
\end{minipage}
\hfill
\begin{minipage}[b]{0.48\linewidth}
\centering
\includegraphics*[scale=0.3]{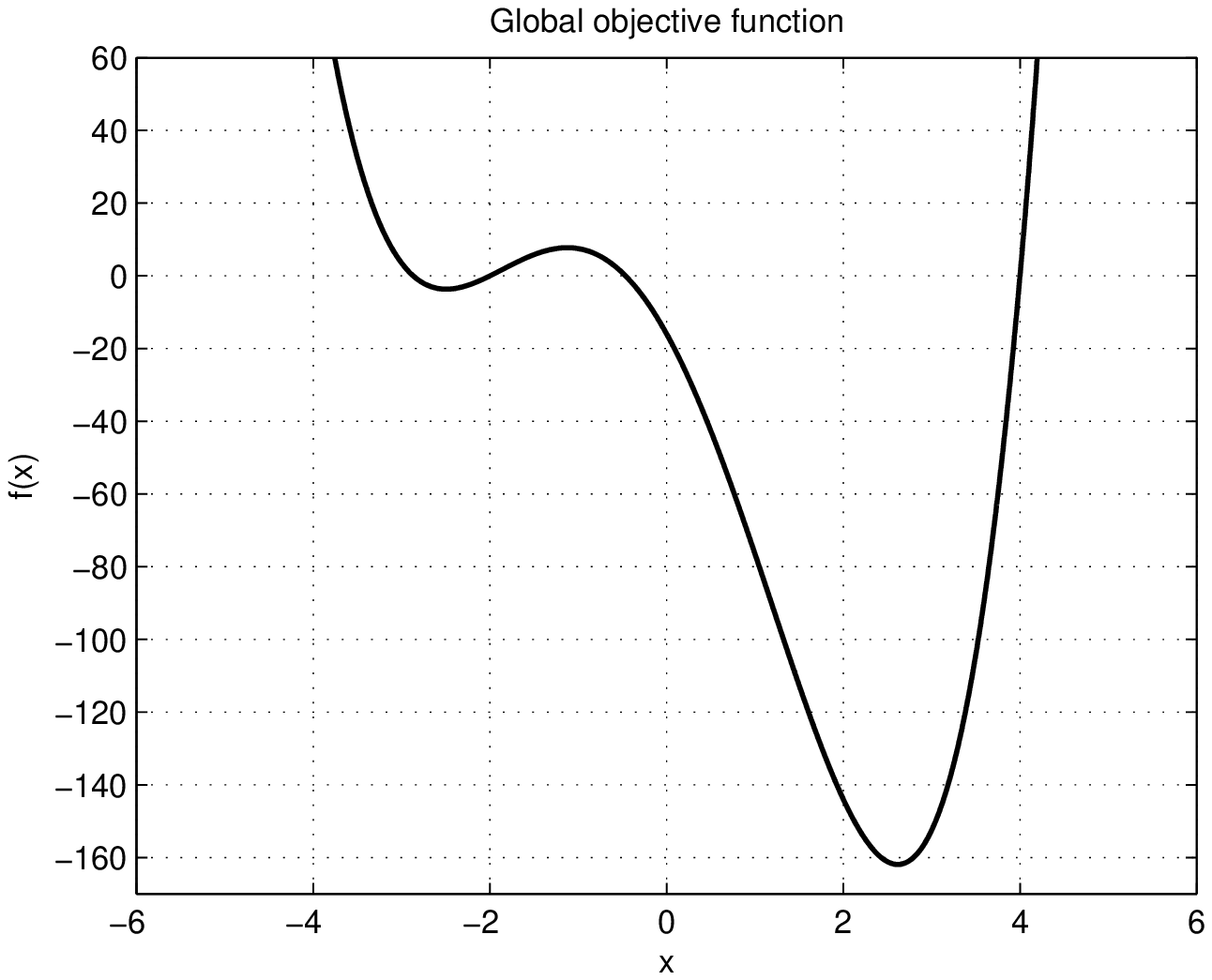}
\centerline{{\small (b) plot on $[-6,6]$}}
\end{minipage}
\hfill
\caption{ The plots of objective function $f$.
}
\label{Fig:objfun}
\end{figure}

\begin{figure}[!t]
\begin{minipage}[b]{0.48\linewidth}
\centering
\includegraphics*[scale=0.3]{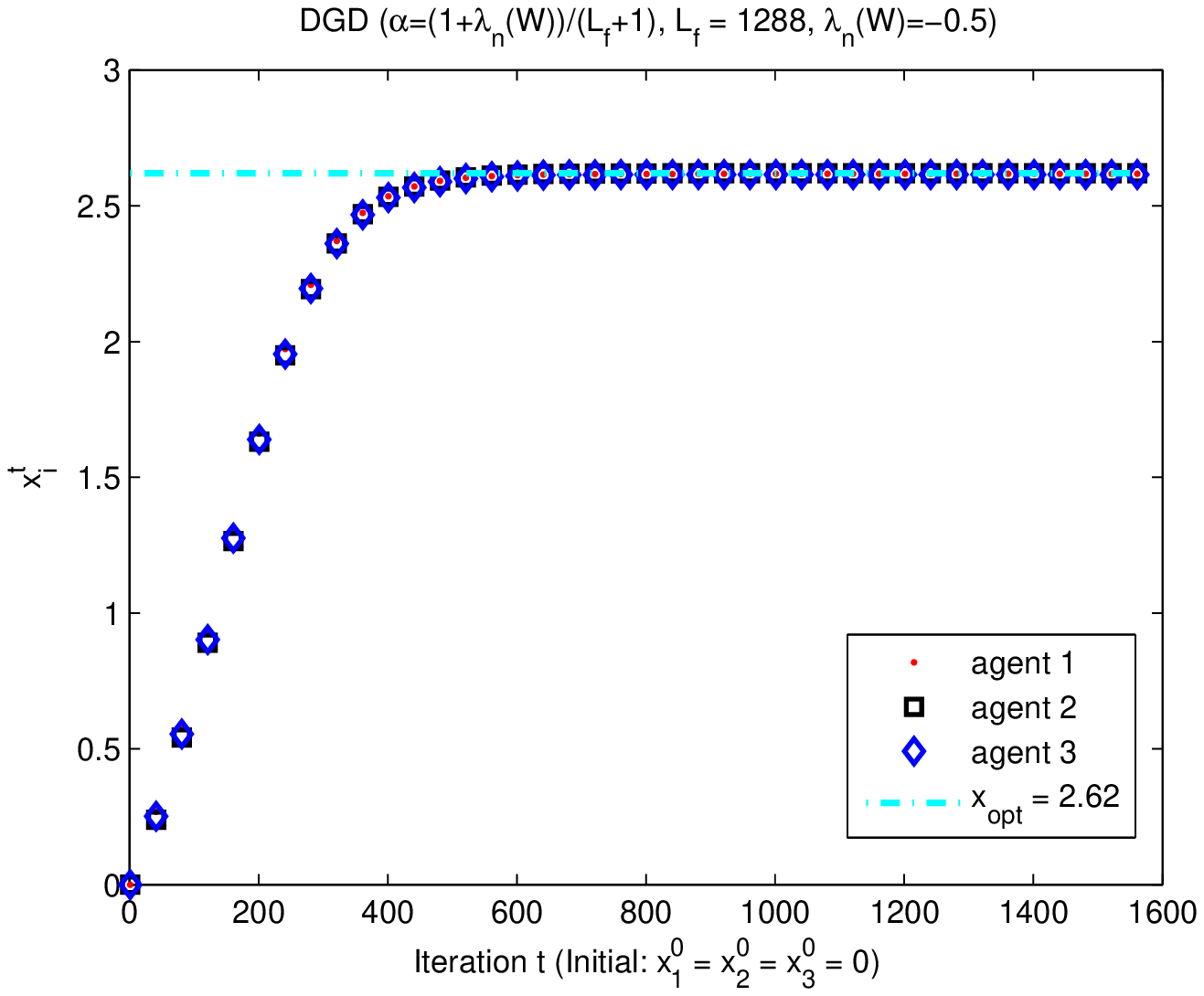}
\centerline{{\small (a) (fixed, $x_1^0$)}}
\end{minipage}
\hfill
\begin{minipage}[b]{0.48\linewidth}
\centering
\includegraphics*[scale=0.3]{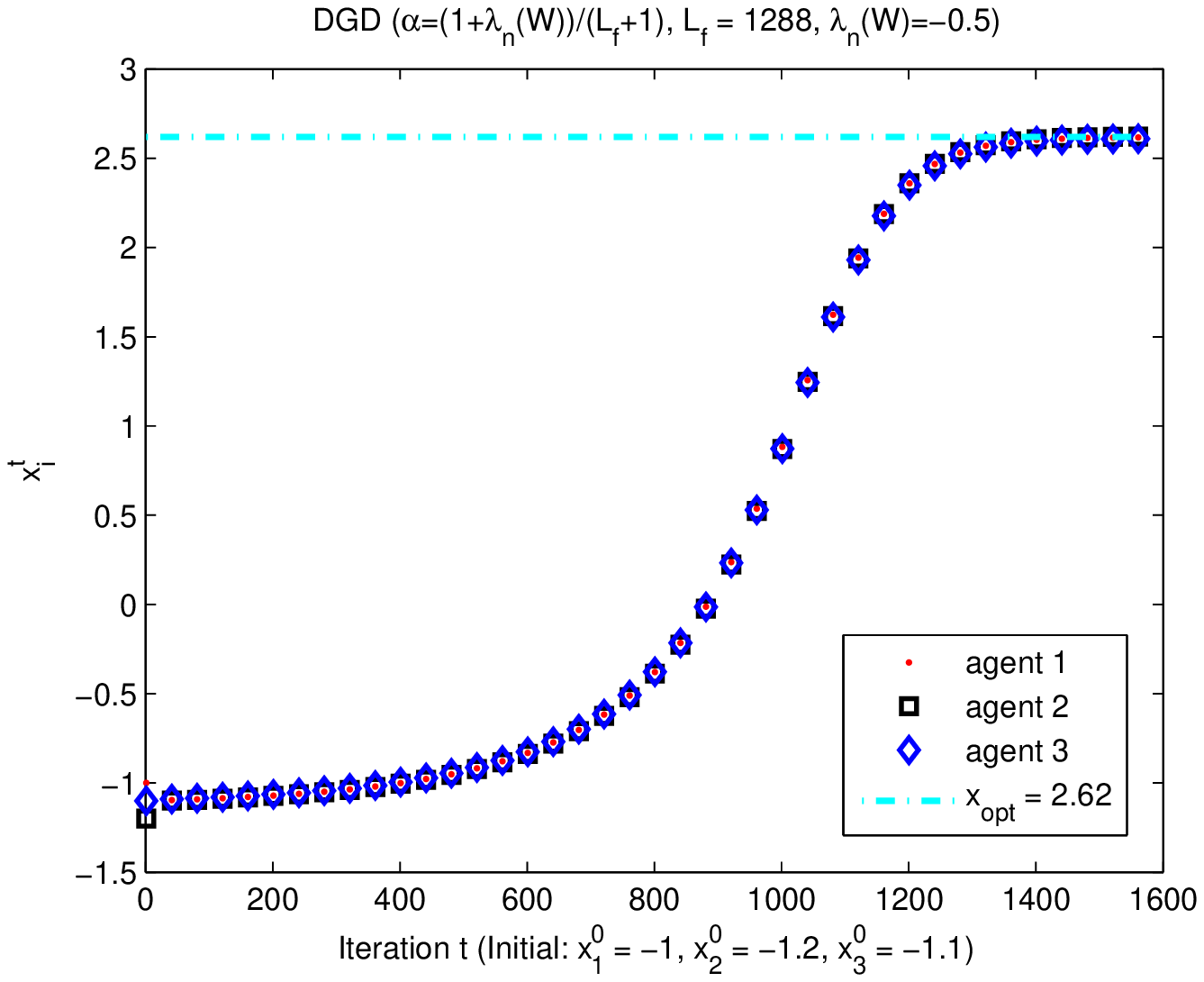}
\centerline{{\small (b) (fixed, $x_2^0$)}}
\end{minipage}
\hfill
\begin{minipage}[b]{0.48\linewidth}
\centering
\vspace{.5cm}
\includegraphics*[scale=0.3]{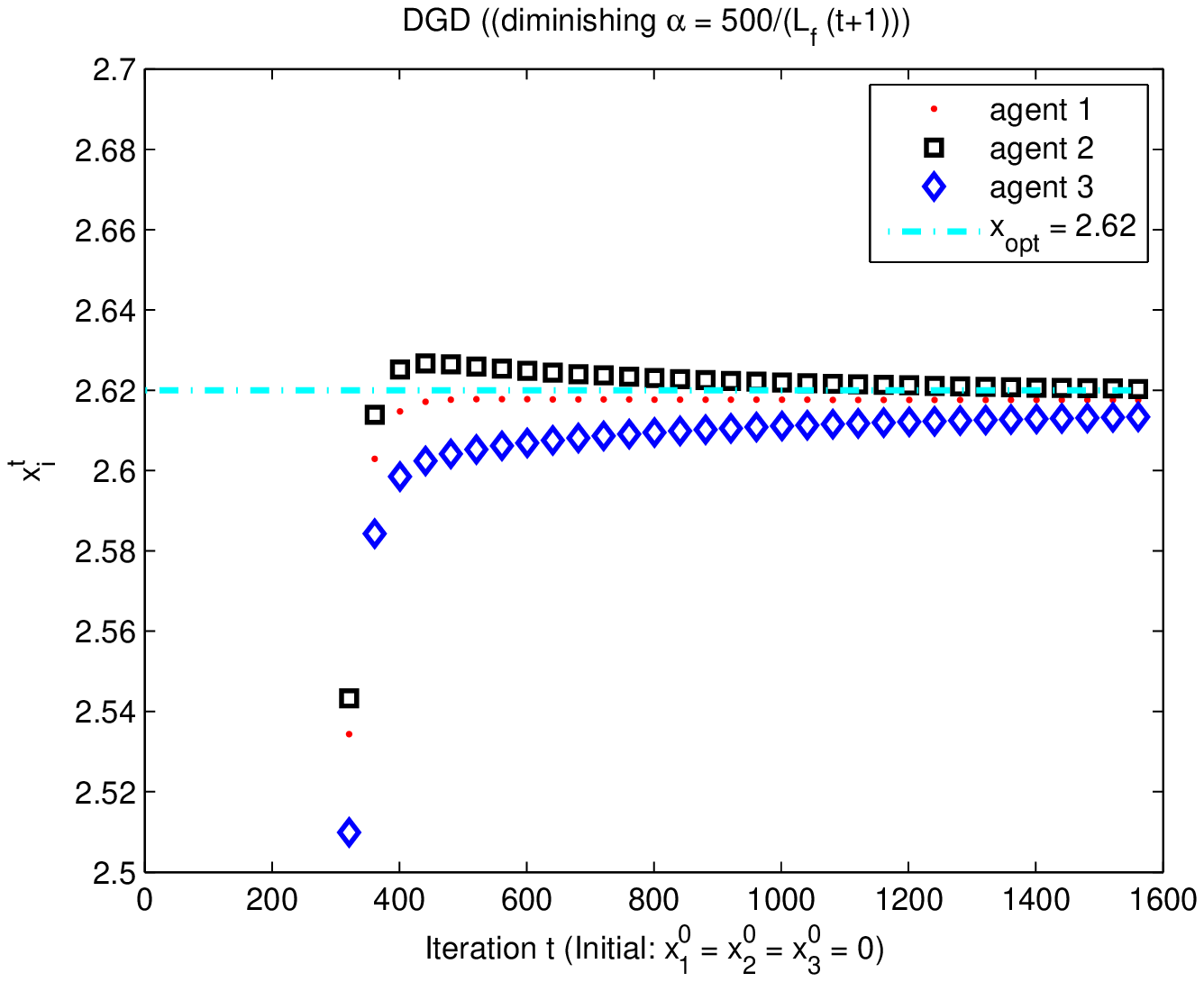}
\centerline{{\small (c) (${\cal O}(\frac{1}{t})$, $x_1^0$)}}
\end{minipage}
\hfill
\begin{minipage}[b]{0.48\linewidth}
\centering
\vspace{.5cm}
\includegraphics*[scale=0.3]{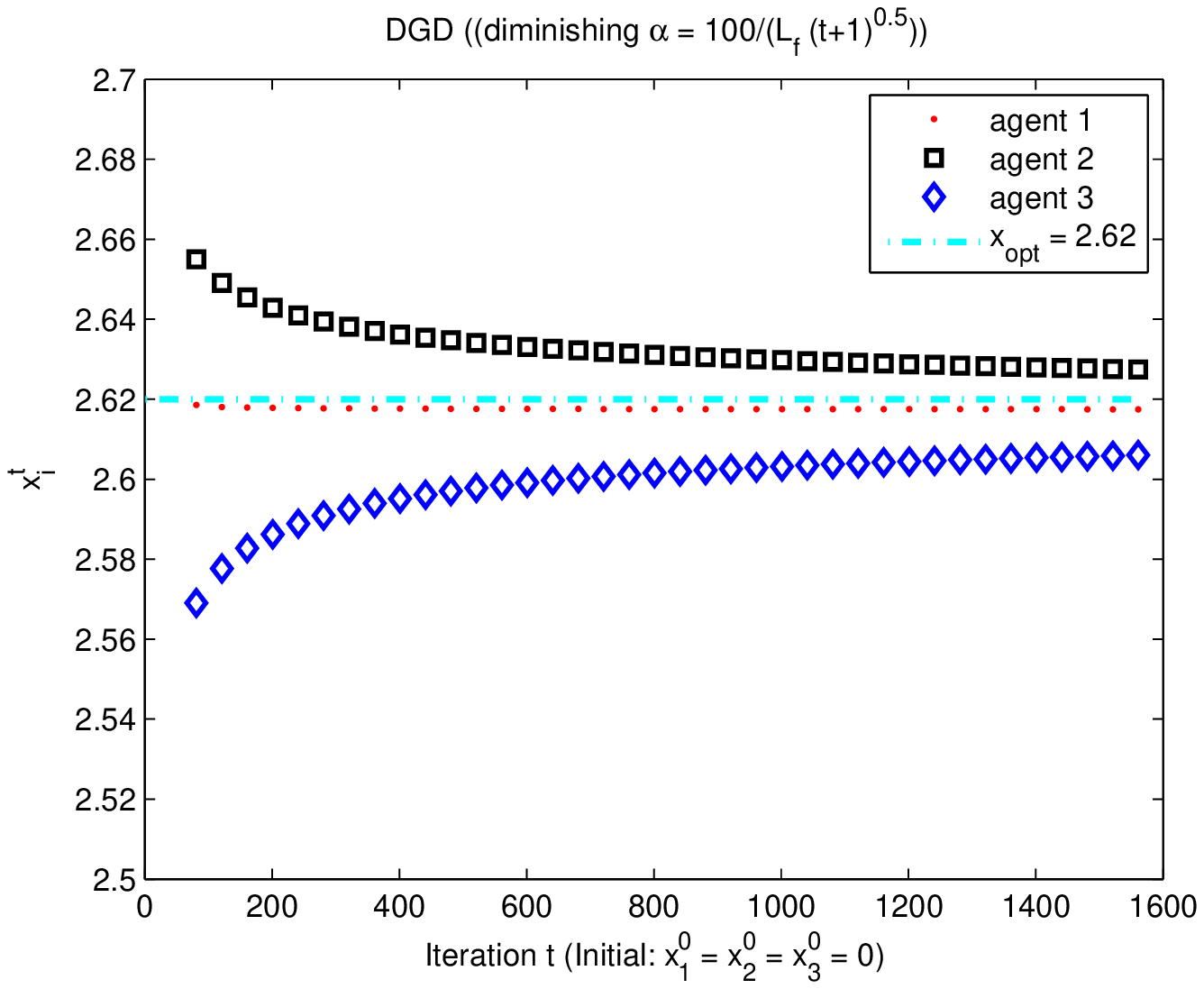}
\centerline{{\small (d) (${\cal O}(\frac{1}{\sqrt{t}})$, $x_1^0$)}}
\end{minipage}
\hfill
\begin{minipage}[b]{0.48\linewidth}
\centering
\vspace{.5cm}
\includegraphics*[scale=0.3]{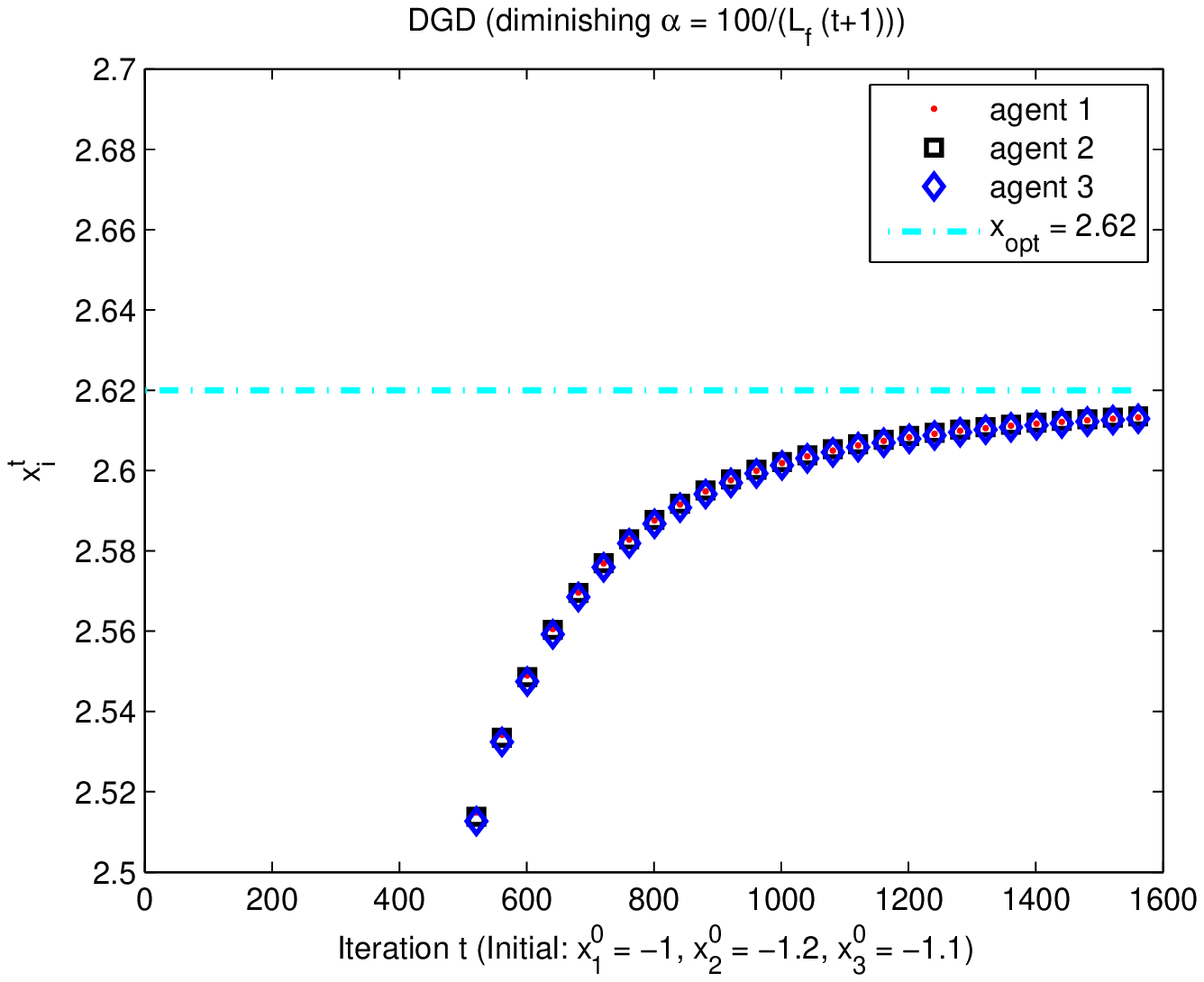}
\centerline{{\small (e) (${\cal O}(\frac{1}{t})$, $x_2^0$)}}
\end{minipage}
\hfill
\begin{minipage}[b]{0.48\linewidth}
\centering
\vspace{.5cm}
\includegraphics*[scale=0.3]{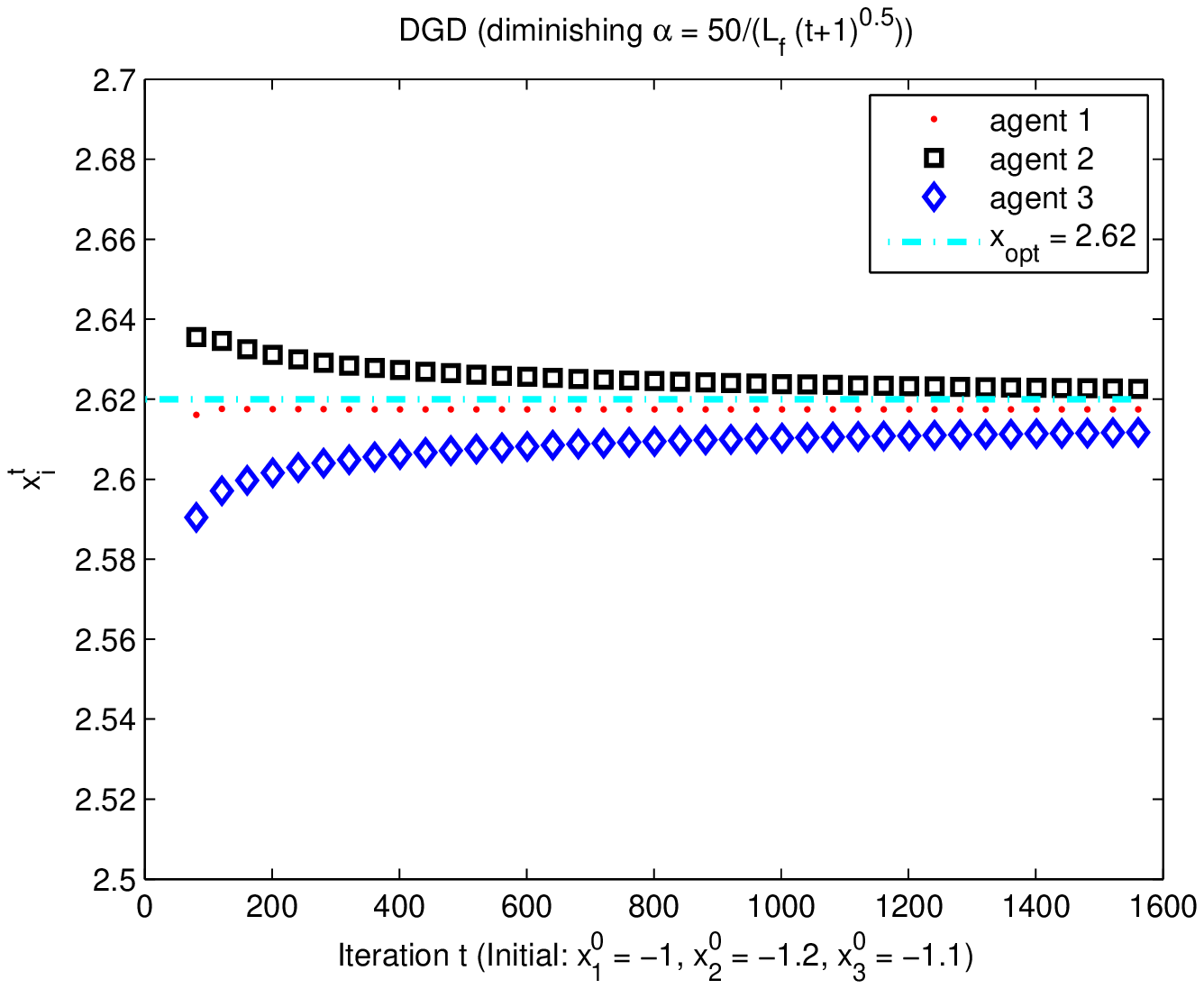}
\centerline{{\small (f) (${\cal O}(\frac{1}{\sqrt{t}})$, $x_2^0$)}}
\end{minipage}
\hfill
\caption{ The convergence behaviors of DGD in different cases.
Two initializations are considered, that is, $x^0_1:= (0,0,0)$ and $x^0_2:=(-1, -1.2, -1.1)$.
}
\label{Fig:conv-dgd}
\end{figure}

\subsection{Prox-DGD for decentralized $L_0$ regularization}

We apply Prox-DGD to solve the following nonconvex decentralized $L_0$ regularization problem:
\begin{equation}
\label{Eq:DistributedLS}
x^* \gets \mathop{\mathrm{argmin}}_{x\in \mathbb{R}^p} f(x) = \frac{1}{n}\sum_{i=1}^n f_i(x),
\end{equation}
where $f_i(x) = \frac{1}{2} \|B_{(i)} x-b_{(i)}\|_2^2+\lambda_i \|x\|_0, B_{(i)} \in \mathbb{R}^{m_i\times p}, b_{(i)} \in \mathbb{R}^{m_i}$ for $i=1, \ldots,n,$ and $\|x\|_0$ is called the $\ell_0$ quasi-norm, which yields the number of the nonzero components of $x$.
In this experiment, we take $n=10$, $p=256$, and $m_i = 150$ for $i=1,\ldots,n$. In this case, the proximal operator of $\ell_0$ quasi-norm is the well-known hard thresholding mapping:
\[
\mathrm{prox}_{\alpha\|\cdot\|_0}(x)
=\left\{
\begin{array}{l}
x, \quad \text{if} \  x<\sqrt{\alpha},\\
0, \quad \text{otherwise}.
\end{array}
\right.
\]

We do not consider the model selection problem, but just test the performance of Prox-DGD applied to such a given deterministic model. Take $\lambda_i = 0.5$ for each agent $i$. The connected network is shown in Fig. \ref{Fig:prox-dgd-network}.
\begin{figure}[!t]
\begin{minipage}[b]{1\linewidth}
\centering
\includegraphics*[scale=0.4]{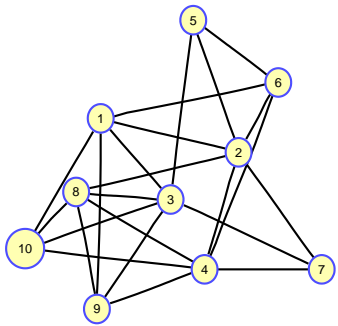}
\end{minipage}
\hfill
\caption{ The network used in the decentralized $L_0$ regularization problem.
}
\label{Fig:prox-dgd-network}
\end{figure}

We use several step-size strategies, including three fixed step sizes and two decreasing step sizes, to test Prox-DGD. The initializations in all cases are 0.
The experiment results are illustrated in Fig. {\ref{Fig:DL0LS}}.
By Fig. {\ref{Fig:DL0LS}}(a) and (b), when a fixed step size is adopted (obviously, we should assume that such step size is sufficiently small to satisfy the theoretical restriction),
a larger fixed step size generally implies a faster convergence rate (just converges to a stationary point of the related Lyapunov function minimization problem but not the original problem) while causes a larger consensus error.
These phenomena are reasonable and further verify the established results in Proposition \ref{propos:consensualbound}.
It can be also observed from Fig. \ref{Fig:DL0LS}(a) that Prox-DGD almost performs at linear rates on all the three fixed step sizes.
It is expected since the $L_0$ regularization function satisfies the Kurdyka-{\L}ojasiewicz (K{\L}) inequality with $\psi(s)=c_1 s^{1-\theta}$ for $\theta \in (0,1/2]$ (see \cite{Attouch2013}).
Thus, by Proposition \ref{Propos:conv-rate-prox-DGD}(ii), Prox-DGD has eventual linear convergence.
Once the initial guess (in this case, 0 may be a good initial guess) is good enough,
Prox-DGD starts decaying linearly starting from early iterations.
Similar phenomenon can be also observed by Fig. \ref{Fig:DL0LS}(c) and (d) under the decreasing step sizes.
Specifically, if $\alpha_t = {\cal O}(\frac{1}{t^{\epsilon}})$, a smaller $\epsilon$ (larger step size) generally implies a faster convergence rate but a slower consensus rate.
This indeed verifies our Proposition \ref{Propos:conv-rate-prox-DGD}, which states that a larger $\epsilon$ generally means a faster consensus rate.
Moreover, from Fig. \ref{Fig:DL0LS}(b) and (d),  under a fixed step size, the consensus error does not vanish but settles to a deterministic value, which can be overcome by using decreasing step sizes. This gives the main motivation of using decreasing step sizes. 


\begin{figure}[!t]
\begin{minipage}[b]{0.48\linewidth}
\centering
\includegraphics*[scale=0.3]{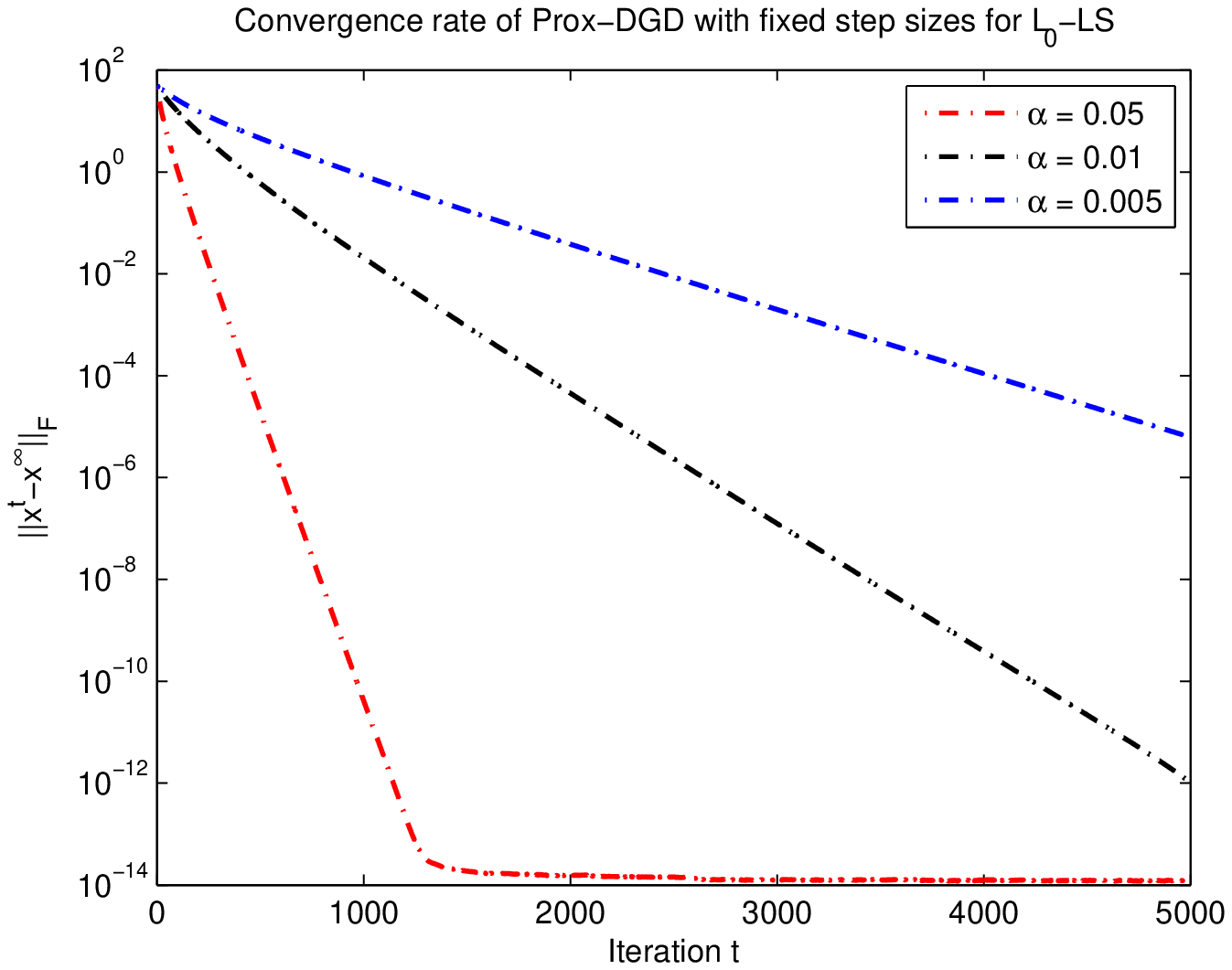}
\centerline{{\small (a) convergence rate (fixed)}}
\end{minipage}
\hfill
\begin{minipage}[b]{0.48\linewidth}
\centering
\includegraphics*[scale=0.3]{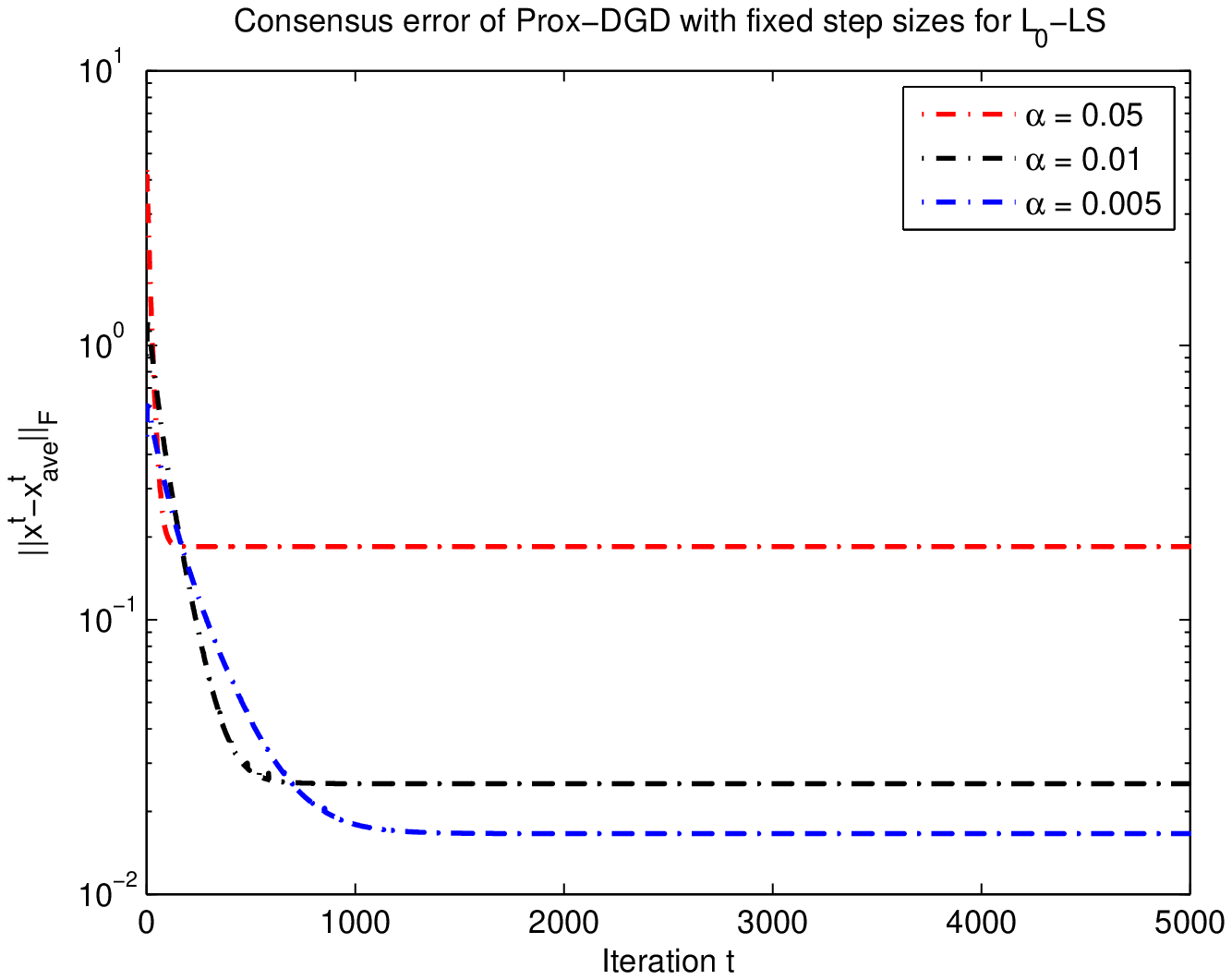}
\centerline{{\small (b) consensus rate (fixed)}}
\end{minipage}
\hfill
\begin{minipage}[b]{0.48\linewidth}
\centering
\vspace{.5cm}
\includegraphics*[scale=0.3]{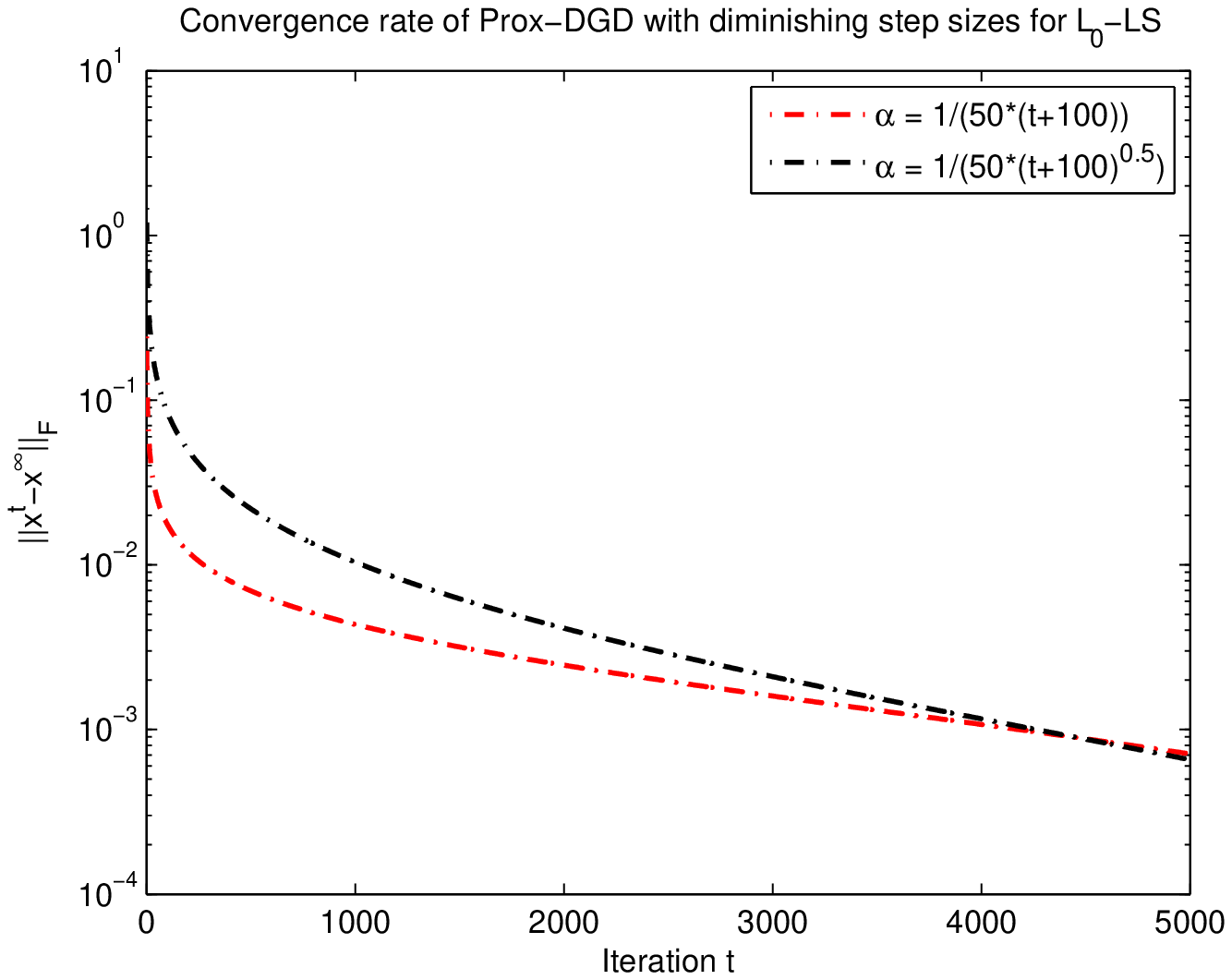}
\centerline{{\small (c) convergence rate (decreasing)}}
\end{minipage}
\hfill
\begin{minipage}[b]{0.48\linewidth}
\centering
\vspace{.5cm}
\includegraphics*[scale=0.3]{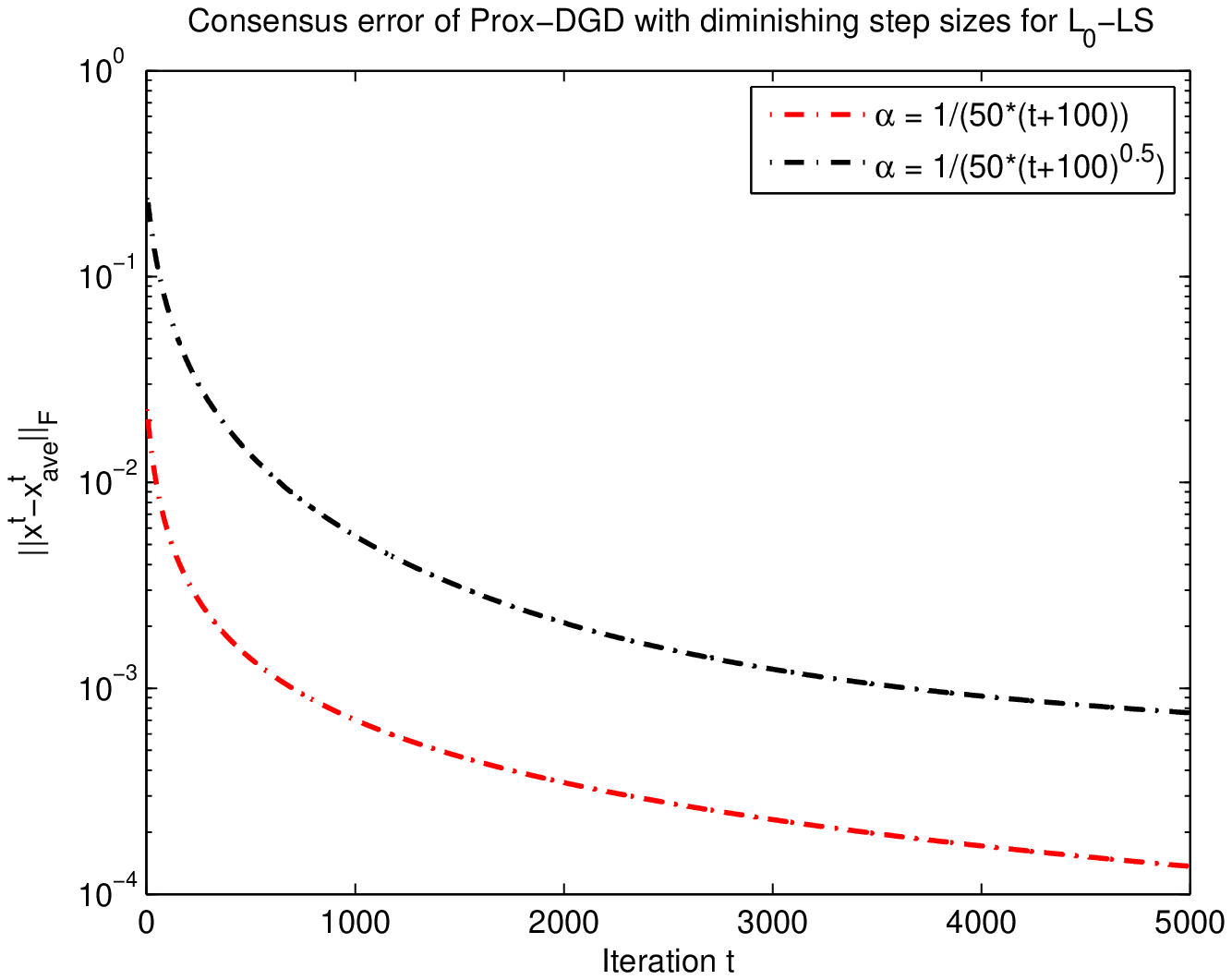}
\centerline{{\small (d) consensus rate (decreasing)}}
\end{minipage}
\hfill
\caption{ The convergence and consensus rates of Prox-DGD under different kinds of step sizes for decentralized $L_0$ regularization. In these experiments, we use the iterate $x^{20000}$ to replace the underlying $x^{\infty}$ used to compute the convergence rates. 
$x_{\mathrm{ave}}^t$ denotes the average of $n$ agents at $t$th iteration, i.e., $x_{\mathrm{ave}}^t = \frac{1}{n}{\bf 1}{\bf 1}^T x^t$.
}
\label{Fig:DL0LS}
\end{figure}

\section{Proofs}
\label{sc:proof}
In this section, we present the proofs of our main theorems and propositions.

\subsection{Proof for Theorem \ref{Thm:Globalconverg}}
The sketch of the proof is as follows: DGD is interpreted as the gradient descent algorithm applied to the  Lyapunov function ${\cal L}_{\alpha}$, following the argument in \cite{Yin-DGD2013}; then,  the properties of sufficient descent, lower boundedness, and bounded gradients are established for the sequence $\{{\cal L}_{\alpha}({\bf x}^k)\}$, giving subsequence convergence of the DGD iterates; finally, whole sequence convergence of the DGD iterates follows from the K{\L} property of ${\cal L}_{\alpha}$.

\begin{Lemm}[Gradient descent interpretation]
\label{Lemm:CentralizedGD}
The sequence $\{{\bf x}^k\}$ generated by the DGD iteration \eqref{Eq:DGD} is the same sequence generated by applying  gradient descent with the fixed step size $\alpha$ to the objective function  ${\cal L}_{\alpha}({\bf x})$. 
\end{Lemm}
A proof of this lemma is given in \cite{Yin-DGD2013}, and it is based on reformulating \eqref{Eq:DGD} as the iteration:
\begin{align}
\label{Eq:DGDequiv1}
{\bf x}^{k+1}
&= {\bf x}^k - \alpha(\nabla {\bf f}({\bf x}^k) + \alpha^{-1}(I-W){\bf x}^k) \nonumber\\
&={\bf x}^k - \alpha\nabla {\cal L}_{\alpha}({\bf x}^k).
\end{align}
Although the sequence $\{{\bf x}^k\}$ generated by the DGD iteration \eqref{Eq:DGD} can be interpreted as a centralized gradient descent sequence of function ${\cal L}_{\alpha}({\bf x})$, it is different to the gradient descent of the original problem \eqref{Eq:consensusProblem}.

\begin{Lemm}[Sufficient descent of $\{{\cal L}_{\alpha}({\bf x}^k)\}$]
\label{Lemm:Suffdescent}
Let Assumptions \ref{Assump:objective} and \ref{Assump:MixMat} hold. Set the step size $0<\alpha<\frac{1+\lambda_n(W)}{L_f}$. It holds that
\begin{align}
&{\cal L}_{\alpha}({\bf x}^{k+1}) \leq {\cal L}_{\alpha}({\bf x}^k) \label{Eq:Suffdescent}\\
&- \frac{1}{2}\big(\alpha^{-1}(1+\lambda_n(W))-L_f\big) \|{\bf x}^{k+1} - {\bf x}^k\|^2, \quad \forall k \in \mathbb{N}. \nonumber
\end{align}
\end{Lemm}

\begin{proof}
From ${\bf x}^{k+1} = {\bf x}^k - \alpha \nabla {\cal L}_{\alpha}({\bf x}^k)$, it follows that
\begin{align}
\label{Eq:Suffdescent1}
 \langle \nabla {\cal L}_{\alpha}({\bf x}^k), {\bf x}^{k+1} - {\bf x}^{k}\rangle
 =-\frac{\|{\bf x}^{k+1} - {\bf x}^{k}\|^2}{\alpha}.
\end{align}
Since   $\sum_{i=1}^n \nabla f_i({\bf x}_{(i)})$ is $L_f$-Lipschitz, $\nabla{\cal L}_{\alpha}$
is Lipschitz with the constant $L^* \triangleq L_f + \alpha^{-1}\lambda_{\max}(I-W) = L_f + \alpha^{-1}(1-\lambda_n(W))$, implying
\begin{align}
\label{Eq:Suffdescent2}
{\cal L}_{\alpha}({\bf x}^{k+1})
&\le {\cal L}_{\alpha}({\bf x}^k) + \langle \nabla {\cal L}_{\alpha}({\bf x}^k), {\bf x}^{k+1} - {\bf x}^{k}\rangle \nonumber\\
&+ \frac{L^*}{2} \|{\bf x}^{k+1} - {\bf x}^{k}\|^2.
\end{align}
Combining \eqref{Eq:Suffdescent1} and \eqref{Eq:Suffdescent2} yields \eqref{Eq:Suffdescent}.
\end{proof}
\begin{Lemm}[Boundedness]
\label{Lemm:Lowerbouded}
 Under Assumptions \ref{Assump:objective} and \ref{Assump:MixMat}, if $0<\alpha<\frac{1+\lambda_n(W)}{L_f}$, then the sequence $\{{\cal L}_{\alpha}({\bf x}^k)\}$ is lower bounded, and the sequence $\{{\bf x}^k\}$ is bounded, i.e., there exists a constant ${\cal B}>0$ such that $\|{\bf x}^k\|<{\cal B}$ for all $k$.
\end{Lemm}

\begin{proof}
The lower boundedness of ${\cal L}_{\alpha}({\bf x}^k)$ is due to  the lower boundedness of each $f_i$ as it is proper and coercive  (Assumption \ref{Assump:objective} Part (2)).

By Lemma \ref{Lemm:Suffdescent} and  the choice of $\alpha$, ${\cal L}_{\alpha}({\bf x}^k)$ is nonincreasing and upper bounded by ${\cal L}_{\alpha}({\bf x}^0)<+\infty$. Hence, ${\bf 1}^T {\bf f}({\bf x}^k) \leq {\cal L}_{\alpha}({\bf x}^0)$  implies that ${\bf x}^k$ is bounded due to the coercivity of ${\bf 1}^T{\bf f}({\bf x})$ (Assumption \ref{Assump:objective} Part (2)).
\end{proof}

From Lemmas \ref{Lemm:Suffdescent} and \ref{Lemm:Lowerbouded}, we immediately obtain the following lemma.

\begin{Lemm}[$\ell_2^2$-summable and asymptotic regularity\footnote{A sequence $\{a_k\}$ is said to be asymptotic regular if $\|a_{k+1}-a_k\| \rightarrow 0$ as $k\rightarrow \infty$.}]
\label{Lemm:Asympregular}
It holds that
$ \sum_{k=0}^{\infty} \|{\bf x}^{k+1}-{\bf x}^k\|^2 < +\infty$ and that $\|{\bf x}^{k+1}-{\bf x}^k\| \rightarrow 0$ as $k\rightarrow \infty.$
\end{Lemm}

From \eqref{Eq:DGDequiv1}, the result below directly follows:
\begin{Lemm}[Gradient bound]
\label{Lemm:boundedgradient}
$\|\nabla {\cal L}_{\alpha}({\bf x}^{k})\|\le \alpha^{-1}\|{\bf x}^{k+1} - {\bf x}^k\|.$
\end{Lemm}

Based on the above lemmas, we get the global convergence of DGD.
\begin{proof}[\bf Proof of Theorem \ref{Thm:Globalconverg}]
By Lemma \ref{Lemm:Lowerbouded}, the sequence $\{{\bf x}^k\}$ is bounded, so there exist a convergent subsequence and a limit point, denoted by  $\{{\bf x}^{k_s}\}_{s\in \mathbb{N}} \rightarrow {\bf x}^*$ as $s\rightarrow +\infty$. By Lemmas \ref{Lemm:Suffdescent} and \ref{Lemm:Lowerbouded}, ${\cal L}_{\alpha}({\bf x}^k)$ is monotonically nonincreasing and lower bounded, and therefore ${\cal L}_{\alpha}({\bf x}^k) \rightarrow {\cal L}^*$ for some ${\cal L}^*$ and $\|{\bf x}^{k+1}-{\bf x}^k\|\rightarrow 0$ as $k\rightarrow \infty$.
Based on Lemma \ref{Lemm:boundedgradient},  $\|\nabla {\cal L}_{\alpha}({\bf x}^{k})\|\rightarrow 0$ as $k\rightarrow \infty$.
In particular, $\|\nabla {\cal L}_{\alpha}({\bf x}^{k_s})\|\rightarrow 0$ as $s\to\infty$. Hence, we have $\nabla {\cal L}_{\alpha}({\bf x}^{*})=0$.

The running best rate of the sequence $\{\|{\bf x}^{k+1}-{\bf x}^k\|^2\}$ follows from \cite[Lemma 1.2]{Deng2013} or \cite[Theorem 3.3.1]{Knopp1956}. By Lemma \ref{Lemm:boundedgradient}, the running best rate of the sequence $\{\|\nabla {\cal L}_{\alpha}({\bf x}^{k})\|^2\}$ is $o(\frac{1}{k})$.

By \eqref{Def:L-alpha}, $\nabla {\cal L}_{\alpha}({\bf x}^{k}) = \nabla {\bf f}({\bf x}^k) + \alpha^{-1}(I-W){\bf x}^k$, which implies
$\frac{1}{n}{\bf 1}^T \nabla {\bf f}({\bf x}^k) = \frac{1}{n}{\bf 1}^T \nabla {\cal L}_{\alpha}({\bf x}^{k})$ due to $\frac{1}{n}{\bf 1}^T (I-W)=0$.
Thus,
\[
\|\frac{1}{n}{\bf 1}^T \nabla {\bf f}({\bf x}^k)\|^2 = \|\frac{1}{n}{\bf 1}^T \nabla {\cal L}_{\alpha}({\bf x}^{k})\|^2 \leq \|\nabla {\cal L}_{\alpha}({\bf x}^{k})\|^2,
\]
which implies the running best rate of $\{\|\frac{1}{n}{\bf 1}^T \nabla {\bf f}({\bf x}^k)\|^2\}$ is also $o(\frac{1}{k})$.

By Lemmas \ref{Lemm:Suffdescent} and \ref{Lemm:boundedgradient}, it holds
\begin{align*}
\|\nabla {\cal L}_{\alpha}({\bf x}^k)\|^2 \leq \frac{2}{\alpha(1+\lambda_n(W)-\alpha L_f)} ({\cal L}_{\alpha}({\bf x}^k) - {\cal L}_{\alpha}({\bf x}^{k+1})),
\end{align*}
which implies
\begin{align*}
\frac{1}{K} \sum_{k=0}^{K-1} \|\nabla {\cal L}_{\alpha}({\bf x}^k)\|^2 \leq \frac{2({\cal L}_{\alpha}({\bf x}^0) - {\cal L}^*)}{\alpha(1+\lambda_n(W)-\alpha L_f)K}.
\end{align*}
Moreover, note that $\|\frac{1}{n}{\bf 1}^T \nabla {\bf f}({\bf x}^k)\|^2 \leq \|\nabla {\cal L}_{\alpha}({\bf x}^{k})\|^2$.
Thus, the convergence rate of $\{\frac{1}{K} \sum_{k=0}^{K-1} \|\frac{1}{n}{\bf 1}^T \nabla {\bf f}({\bf x}^k)\|^2\}$ is ${\cal O}(\frac{1}{K})$.

Similar to \cite[Theorem 2.9]{Attouch2013}, we can claim the global convergence of the considered sequence $\{{\bf x}^k\}_{k\in \mathbb{N}}$ under the K{\L} assumption of ${\cal L}_{\alpha}$.
\end{proof}

Next, we  derive a bound on the gradient sequence $\{\nabla {\bf f}({\bf x}^k)\}$, which is used in Proposition \ref{propos:consensualbound}.

\begin{Lemm}
\label{Lemm:bound_grad_f}
Under Assumption \ref{Assump:objective}, there exists a point ${\bf y}^*$ satisfying $\nabla {\bf f}({\bf y}^*)=0$, and the following bound holds
\begin{align}
\label{Eq:bound_grad_f}
\|\nabla {\bf f}({\bf x}^k)\| \leq D \triangleq L_f({\cal B}+ \|{\bf y}^*\|), \quad \forall k\in \mathbb{N},
\end{align}
where ${\cal B}$ is the bound of $\|{\bf x}^k\|$ given in {Lemma} \ref{Lemm:Lowerbouded}.
\end{Lemm}
\begin{proof}
By the lower boundedness assumption (Assumption \ref{Assump:objective} Part (2)), the minimizer of ${\bf 1}^T{\bf f}({\bf y})$ exists. Let ${\bf y}^*$ be a minimizer. Then by Lipschitz differentiability of each $f_i$ (Assumption \ref{Assump:objective} Part (1)), we have that $\nabla {\bf f}({\bf y}^*)=0$.

Then, for any $k$, we have
\begin{align*}
\|\nabla {\bf f}({\bf x}^k)\| &= \|\nabla {\bf f}({\bf x}^k) - \nabla {\bf f}({\bf y}^*)\| \leq L_f \|{\bf x}^k - {\bf y}^*\|\\
(\text{Lemma}\  \ref{Lemm:Lowerbouded}) \quad &\leq L_f({\cal B}+ \|{\bf y}^*\|).
\end{align*}
Therefore, we have proven this lemma.
\end{proof}

\subsection{Proof for Proposition \ref{Propos:conv-rate}}
\begin{proof}
Note that
 \begin{align*}
 \|\nabla {\cal L}_{\alpha}({\bf x}^{k+1})\|
 &\leq \|\nabla {\cal L}_{\alpha}({\bf x}^{k+1})-\nabla {\cal L}_{\alpha}({\bf x}^{k})\| + \|\nabla {\cal L}_{\alpha}({\bf x}^{k})\|\\
 &\leq L^* \|{\bf x}^{k+1} - {\bf x}^k\|+\alpha^{-1}\|{\bf x}^{k+1} - {\bf x}^k\| \\
 & = (\alpha^{-1}(2-\lambda_n(W))+L_f)\|{\bf x}^{k+1} - {\bf x}^k\|,
 \end{align*}
where the second inequality holds for Lemma \ref{Lemm:boundedgradient} and the Lipschitz continuity of $\nabla {\cal L}_{\alpha}$ with constant $L^*=L_f + \alpha^{-1}(1-\lambda_n(W))$. Thus, it shows that $\{{\bf x}^k\}$ satisfies the so-called relative error condition as list in {\cite{Attouch2013}}.
Moreover, by Lemmas \ref{Lemm:Suffdescent} and \ref{Lemm:Lowerbouded},
$\{{\bf x}^k\}$ also satisfies the so-called sufficient decrease and continuity conditions as listed in {\cite{Attouch2013}}.
Under such three conditions and the K{\L} property of ${\cal L}_{\alpha}$ at ${\bf x}^*$ with $\psi(s) = cs^{1-\theta}$,  following the proof of {\cite[Lemma 2.6]{Attouch2013}}, there exists $k_0>0$ such that for all $k \geq k_0$, we have
\begin{align}
&2\|{\bf x}^{k+1}-{\bf x}^k\| \leq \|{\bf x}^{k}-{\bf x}^{k-1}\|+\frac{cb}{a}\times   \label{Eq:keyineq1}\\
&\big(({\cal L}_{\alpha}({\bf x}^k)-{\cal L}_{\alpha}({\bf x}^*))^{1-\theta}-({\cal L}_{\alpha}({\bf x}^{k+1})-{\cal L}_{\alpha}({\bf x}^*))^{1-\theta}\big),\nonumber
\end{align}
where $a \triangleq \frac{1}{2}(\alpha^{-1}(1+\lambda_n(W))-L_f)$ and $b \triangleq \alpha^{-1}(2-\lambda_n(W))+L_f$.
Then, an easy induction yields
\begin{align*}
&\sum_{t=k_0}^k \|{\bf x}^{t+1} - {\bf x}^t\| \leq \|{\bf x}^{k_0} - {\bf x}^{k_0-1}\|+\frac{cb}{a} \times\\
&\big(({\cal L}_{\alpha}({\bf x}^{k_0})-{\cal L}_{\alpha}({\bf x}^*))^{1-\theta} - ({\cal L}_{\alpha}({\bf x}^{k+1})-{\cal L}_{\alpha}({\bf x}^*))^{1-\theta}\big).
\end{align*}
Following a derivation similar to the proof of {\cite[Theorem 5]{Attouch-Bolte2009}}, we can estimate the rate of convergence of $\{{\bf x}^k\}$ in the different cases of $\theta$.
\end{proof}

\subsection{Proof for Proposition \ref{Propos:asympconsensus}}

In order to prove Proposition \ref{Propos:asympconsensus}, we also need the following lemmas.


\begin{Lemm}(\cite[Proposition 1]{Nedic-Subgradient2009})
\label{Lemm:Wk}
Let $W^k \triangleq \overbrace{W\cdots W}^k$ be the power of $W$ with degree $k$ for any $k\in \mathbb{N}$. Under Assumption \ref{Assump:MixMat},  it holds
\begin{align}
\label{Eq:Wk}
\|W^k - \frac{1}{n}{\bf 1}{\bf 1}^T\| \leq C \zeta^k
\end{align}
for some constant $C>0,$ where $\zeta$ is the second largest magnitude eigenvalue of $W$ as specified in \eqref{Eq:zeta}.
\end{Lemm}

\begin{Lemm}(\cite[Lemma 3.1]{Ram-DSSP2010})
\label{Lemm:convolution}
Let $\{\gamma_k\}$ be a scalar sequence.
If $\lim_{k\rightarrow \infty} \gamma_k = \gamma$ and $0<\beta<1$, then $\lim_{k\rightarrow \infty} \sum_{l=0}^k \beta^{k-l}\gamma_l = \frac{\gamma}{1-\beta}$.
%
\end{Lemm}


\begin{proof}[\bf Proof of Proposition \ref{Propos:asympconsensus}]
By the recursion \eqref{Eq:recursion-xk}, note that
\begin{align}
\label{Eq:asympconsens1}
{\bf x}^{k} - {\bar{\bf x}}^{k} =&~ (W^{k}-\frac{1}{n}{\bf 1}{\bf 1}^T){\bf x}^0\\
& - \sum_{j=0}^{k-1} \alpha_j (W^{k-1-j}-\frac{1}{n}{\bf 1}{\bf 1}^T)\nabla {\bf f}({\bf x}^j).\nonumber
\end{align}
Further by Lemma \ref{Lemm:Wk} and Assumption \ref{Assump:boundedgradient}, we obtain
\begin{align}
\label{Eq:asympconsens2}
&\|{\bf x}^{k} - {\bar{\bf x}}^{k}\| \leq \|(W^{k}-\frac{1}{n}{\bf 1}{\bf 1}^T)\| \|{\bf x}^0\|\nonumber\\
& + \sum_{j=0}^{k-1} \alpha_j \|W^{k-1-j}-\frac{1}{n}{\bf 1}{\bf 1}^T\|\cdot \|\nabla {\bf f}({\bf x}^j)\|\nonumber\\
&\leq C \left(\|{\bf x}^0\|\zeta^{k}+B \sum_{j=0}^{k-1}\alpha_j \zeta^{k-1-j}\right).
\end{align}
Furthermore, by Lemma \ref{Lemm:convolution} and step sizes \eqref{eq:decralpha}, we get $\lim_{k\rightarrow \infty} \|{\bf x}^{k} - {\bar{\bf x}}^{k}\| =0.$

Let $b_k \triangleq (k+1)^{-\epsilon}$. To show the rate of $\|{\bf x}^{k} - {\bar{\bf x}}^{k}\|$, we only need to show that
\[\lim_{k\rightarrow \infty}b_k^{-1}\|{\bf x}^k - {\bar {\bf x}}^k\| \leq C^*\]
for some $0<C^*<\infty$. Let $j_k' \triangleq [k-1+2\log_{\zeta}(b_k^{-1})]$ (where $[x]$ denotes the integer part of $x$ for any $x\in \mathbb{R}$).
Note that
\begin{align}
\label{Eq:alphak-xk}
&b_k^{-1}\|{\bf x}^k - {\bar {\bf x}}^k\|\nonumber\\
& \leq C b_k^{-1}\left(\|{\bf x}^0\|\zeta^k+B\sum_{j=0}^{k-1}\alpha_j\zeta^{k-1-j}\right) \nonumber\\
&= C\|{\bf x}^0\| b_k^{-1}\zeta^{k}
+CB b_k^{-1} \sum_{j=0}^{j_k'}\alpha_j\zeta^{k-1-j} \nonumber\\
&+CB b_k^{-1} \sum_{j=j_k'+1}^{k-1}\alpha_j\zeta^{k-1-j} \nonumber\\
&\triangleq T_1 +T_2 +T_3,
\end{align}
where the first inequality holds because of \eqref{Eq:asympconsens2}.

In the following, we will estimate the above three terms in the right-hand side of \eqref{Eq:alphak-xk}, respectively.
First, by the definition of $j_k'$, for any $j\leq j_k'$, we have
\[
b_k^{-1} \zeta^{\frac{k-1-j}{2}} \leq b_k^{-1} \zeta^{\frac{k-1-j_k'}{2}} \leq 1.
\]
Thus,
\begin{align}
\label{Eq:alphak-xk-part2}
T_2 \leq CB\sum_{j=0}^{j_k'}\alpha_j\zeta^{(k-1-j)/2}.
\end{align}
Second, for $j_k'<j\leq k-1$,
\begin{align*}
&b_k^{-1}\alpha_j \leq \frac{(k+1)^{\epsilon}}{L_f(j_k'+1)^{\epsilon}} \leq \frac{(k+1)^{\epsilon}}{L_f(k-1+2\epsilon \log_{\zeta}(k+1))^{\epsilon}},
\end{align*}
and also
\begin{align*}
&b_k^{-1}\alpha_j \geq  \frac{(k+1)^{\epsilon}}{L_f(k+1)^{\epsilon}} = \frac{1}{L_f},
\end{align*}
Thus, for any $j_k'<j\leq k-1$
\begin{align}
\label{Eq:bk4-part1}
\lim_{k\rightarrow \infty} b_k^{-1}\alpha_j = \frac{1}{L_f }.
\end{align}
Furthermore, note that
\begin{align}
\label{Eq:bk4-part2}
& \lim_{k\rightarrow \infty} b_k^{-1}\zeta^{k/2}=0.
\end{align}
Therefore, there exists a $k^{*}$ such that for $k\geq k^{*}$
\begin{align}
& b_k^{-1}\alpha_j \leq \frac{2}{L_f}, \label{Eq:cond1-k**}\\
& b_k^{-1}\zeta^{k/2}\leq 1 \label{Eq:cond2-k**}.
\end{align}
The above two inequalities imply that for sufficiently large $k$,
\begin{align}
& T_1 \leq C \|{\bf x}^0\|\zeta^{k/2}, \label{Eq:alphak-xk-part1}\\
& T_3 \leq \frac{2CB}{L_f}\sum_{j=j_k'+1}^{k-1} \zeta^{k-1-j}. \label{Eq:alphak-xk-part3}
\end{align}
From \eqref{Eq:alphak-xk-part2}, \eqref{Eq:alphak-xk-part1} and \eqref{Eq:alphak-xk-part3}, we get
\begin{align}
\label{Eq:alphak-xk1}
&b_k^{-1}\|{\bf x}^k - {\bar {\bf x}}^k\| \leq C \|{\bf x}^0\|\zeta^{k/2}\\
& + CB \left(\sum_{j=0}^{j_k'}\alpha_j\zeta^{(k-1-j)/2}+ \frac{2}{L_f}\sum_{j=j_k'+1}^{k-1} \zeta^{k-1-j}\right). \nonumber
\end{align}
By Lemma \ref{Lemm:convolution} and \eqref{Eq:alphak-xk1}, there exists a $C^*>0$ such that
\begin{align}
\label{Eq:alphak-xk2}
\lim_{k\rightarrow \infty} b_k^{-1}\|{\bf x}^k - {\bar {\bf x}}^k\| \leq C^*.
\end{align}
We have completed the proof of this proposition.
\end{proof}

%

%

\subsection{Proof for Theorem \ref{Thm:Lalphak}}

To prove Theorem \ref{Thm:Lalphak}, we first note that
similar to \eqref{Eq:DGDequiv1}, the DGD iterates under decreasing step sizes can be rewritten as
\begin{align}
\label{Eq:DGDequiv-decreas}
{\bf x}^{k+1}
={\bf x}^k - \alpha_k\nabla {\cal L}_{\alpha_k}({\bf x}^k),
\end{align}
where ${\cal L}_{\alpha_k}({\bf x}) = {\bf 1}^T{\bf f}({\bf x}) + \frac{1}{2\alpha_k} \|{\bf x}\|_{I-W}^2$,
and we also need the following lemmas.

\begin{Lemm}[\cite{Robbins-supermartingale1971}]
\label{Lemm:supermartingale}
Let $\{v_t\}$ be a nonnegative scalar sequence such that
\[
v_{t+1} \leq (1+b_t)v_t - u_t +c_t
\]
for all $t\in \mathbb{N}$, where $b_t \geq 0$, $u_t\geq 0$ and $c_t\geq 0$ with $\sum_{t=0}^{\infty} b_t <\infty$ and $\sum_{t=0}^{\infty} c_t <\infty$. Then the sequence $\{v_t\}$ converges to some $v\geq 0$ and $\sum_{t=0}^{\infty} u_t <\infty.$
\end{Lemm}

\begin{Lemm}
\label{Lemm:alphak}
Let $\alpha_k$ satisfy \eqref{eq:decralpha}. Then it holds
\begin{align*}
\alpha_{k+1}^{-1} - \alpha_k^{-1} \leq 2\epsilon L_f(k+1)^{\epsilon-1}.
\end{align*}
\end{Lemm}
\begin{proof}
We first prove that
\begin{align}
\label{Eq:taylor}
(1+x)^{\epsilon}-1 \leq 2\epsilon x, \quad \forall x\in [0,1].
\end{align}
Let $g(x) = (1+x)^{\epsilon}-1 - 2\epsilon x$. Then its derivative
\[
g'(x) = \epsilon(1+x)^{\epsilon-1}-2\epsilon<0, \quad \forall x\in [0,1].
\]
It implies $g(x) \leq g(0)=0$ for any $x\in [0,1]$, that is, the inequality \eqref{Eq:taylor} holds.

Note that
\begin{align}
\label{Eq:diff-alphak}
\alpha_{k+1}^{-1} - \alpha_k^{-1}
&= L_f\big((k+2)^{\epsilon} - (k+1)^{\epsilon}\big) \nonumber\\
&= L_f(k+1)^{\epsilon}\big((1+\frac{1}{k+1})^{\epsilon}-1\big) \nonumber\\
&\leq 2\epsilon L_f(k+1)^{\epsilon-1},
\end{align}
where the last inequality holds for \eqref{Eq:taylor}.
\end{proof}

The following lemma shows that $\{(\alpha_{k+1}^{-1}-\alpha_k^{-1})\|{\bf x}^{k+1}\|_{I-W}^2\}$ is summable.
\begin{Lemm}
\label{Lemm:bk}
 Let Assumptions \ref{Assump:objective}, \ref{Assump:MixMat}, and \ref{Assump:boundedgradient} 
hold. In DGD, use step sizes $\alpha_k$ in \eqref{eq:decralpha}. Then $\{(\alpha_{k+1}^{-1}-\alpha_k^{-1})\|{\bf x}^{k+1}\|_{I-W}^2\}$ is summable, i.e., $\sum_{k=0}^{\infty} (\alpha_{k+1}^{-1}-\alpha_k^{-1})\|{\bf x}^{k+1}\|_{I-W}^2 < \infty.$
\end{Lemm}

\begin{proof}
Note that
\begin{align}
\label{Eq:xk+1-I-W}
\|{\bf x}^{k+1}\|_{I-W}^2
&= \|{\bf x}^{k+1} - {\bar {\bf x}}^{k+1}\|_{I-W}^2 \nonumber\\
&\leq (1-\lambda_n(W))\|{\bf x}^{k+1} - {\bar {\bf x}}^{k+1}\|^2.
\end{align}
By Lemma \ref{Lemm:alphak},
\begin{align}
\label{Eq:bk1}
&(\alpha_{k+1}^{-1}-\alpha_k^{-1})\|{\bf x}^{k+1}\|_{I-W}^2 \nonumber\\
&\leq 2\epsilon L_f(k+1)^{\epsilon-1}\|{\bf x}^{k+1}\|_{I-W}^2 \nonumber\\
&\leq 2\epsilon L_f(k+1)^{\epsilon-1}(1-\lambda_n(W))\|{\bf x}^{k+1} - {\bar {\bf x}}^{k+1}\|^2.
\end{align}
Furthermore, by \eqref{Eq:bk1} and Proposition \ref{Propos:asympconsensus}, the sequence $\{(\alpha_{k+1}^{-1}-\alpha_k^{-1})\|{\bf x}^{k+1}\|_{I-W}^2\}$ converges to 0 at the rate of ${\cal O}(1/(k+1)^{1+\epsilon})$, which implies that the sequence $\{(\alpha_{k+1}^{-1}-\alpha_k^{-1})\|{\bf x}^{k+1}\|_{I-W}^2\}$ is $\ell_1$-summable, i.e.,  $\sum_{k=0}^{\infty} (\alpha_{k+1}^{-1}-\alpha_k^{-1})\|{\bf x}^{k+1}\|_{I-W}^2 <\infty$.
\end{proof}

\begin{Lemm}[convergence of weakly summable sequence]
\label{Lemm:betak}
Let $\{\beta_k\}$ and $\{\gamma_k\}$ be two nonnegative scalar sequences such that
\begin{enumerate}
\item[(a)]
$\gamma_k = \frac{1}{(k+1)^{\epsilon}}$, for some $\epsilon \in (0,1]$, $k\in \mathbb{N}$;

\item[(b)]
$\sum_{k=0}^{\infty}\gamma_k\beta_k <\infty$;

\item[(c)]
$|\beta_{k+1}-\beta_k| \lesssim \gamma_k,$
\end{enumerate}
where ``$\lesssim$''means that $|\beta_{k+1}-\beta_k| \le M \gamma_k$ for some constant $M>0$, then $\lim_{k\rightarrow \infty}\beta_k \rightarrow 0$.
\end{Lemm}
We call a sequence $\{\beta_k\}$ satisfying Lemma \ref{Lemm:betak} (a) and (b) a \textit{weakly summable} sequence since itself is not necessarily summable but becomes summable via multiplying another non-summable, diminishing sequence $\{\gamma_k\}$. It is generally impossible to claim that $\beta_k$ converges to 0. However, if the distance of two successive steps of $\{\beta_k\}$ with the same order of the multiplied sequence $\gamma_k$, then we can claim the convergence of $\beta_k$. A special case with $\epsilon=1/2$ has been observed in \cite{Chow2016}.

\begin{proof}
By condition (b),
we have
\begin{align}
\label{Eq:cauchyseq}
\sum_{i=k}^{k+k'} \gamma_i \beta_i \rightarrow 0,
\end{align}
as $k\rightarrow \infty$ and for any $k' \in \mathbb{N}$.

In the following, we will show $\lim_{k\rightarrow \infty} \beta_k =0$ by contradiction. Assume this is not the case, i.e., $\beta_k \nrightarrow 0$ as $k\rightarrow \infty$, then $\mathop{\limsup}_{k\rightarrow \infty} \beta_k \triangleq C^*>0.$ Thus, for every $N>k_0$, there exists a $k>N$ such that $\beta_k > \frac{C^*}{2}$. Let
\[k' \triangleq \left[\frac{C^*}{4M}(k+1)^{\epsilon}\right],\]
where $[x]$ denotes the integer part of $x$ for any $x\in \mathbb{R}$. By condition (c), i.e., $|\beta_{j+1}-\beta_j| \le M \gamma_j$ for any $j\in \mathbb{N}$, then
\begin{align}
\label{Eq:betak+i}
\beta_{k+i} \ge \frac{C^*}{4}, \quad \forall i\in \{0,1,\ldots,k'\}.
\end{align}
Hence,
\begin{align}
&\sum_{j=k}^{k+k'} \gamma_j \beta_j \geq \frac{C^*}{4}\sum_{j=k}^{k+k'} \gamma_j \geq \frac{C^*}{4}\int_{k}^{k+k'}(x+1)^{-\epsilon}dx \label{Eq:cauchyseq1}\\
&=
\left\{
\begin{array}{ll}
\frac{C^*}{4(1-\epsilon)}\left((k+k'+1)^{1-\epsilon} - (k+1)^{1-\epsilon}\right), & \epsilon \in (0,1),\\
\frac{C^*}{4}\left(\ln(k+k'+1)-\ln(k+1)\right), & \epsilon=1.
\end{array}%
\right.
\nonumber
\end{align}
Note that when $\epsilon \in (0,1)$, the term $(k+k'+1)^{1-\epsilon} - (k+1)^{1-\epsilon}$ is monotonically increasing with respect to $k$, which implies that $\sum_{j=k}^{k+k'} \gamma_j \beta_j$ is lower bounded by a positive constant when $\epsilon \in (0,1)$. While when $\epsilon=1$, noting that the specific form of $k'$, we have
\[\ln(k+k'+1)-\ln(k+1) = \ln\left(1+\frac{k'}{k+1}\right)=\ln\left(1+\frac{C^*}{4M}\right),\]
 which is a positive constant. As a consequence, $\sum_{j=k}^{k+k'} \gamma_j \beta_j$ will not go to $0$ as $k\rightarrow 0$, which contradicts with \eqref{Eq:cauchyseq}. Therefore, $\lim_{k\rightarrow \infty} \beta_k =0.$
\end{proof}

\begin{proof}[\bf Proof of Theorem \ref{Thm:Lalphak}]
We first develop the following inequality
\begin{align}
\label{Eq:Suffdescent-decrease}
&{\cal L}_{\alpha_{k+1}}({\bf x}^{k+1}) \leq {\cal L}_{\alpha_k}({\bf x}^k) + \frac{1}{2}(\alpha_{k+1}^{-1}-\alpha_k^{-1})\|{\bf x}^{k+1}\|^2_{I-W}\nonumber\\
&- \frac{1}{2}\big(\alpha_k^{-1}(1+\lambda_n(W))-L_f\big)\|{\bf x}^{k+1}-{\bf x}^k\|^2,
\end{align}
and then claim the convergence of the sequences $\{{\cal L}_{\alpha_k}({\bf x}^k)\}$, $\{{\bf 1}^T {\bf f}({\bf x}^k)\}$ and $\{{\bf x}^k\}$ based on this inequality.

\textbf{(a) Development of \eqref{Eq:Suffdescent-decrease}:}
From ${\bf x}^{k+1} = {\bf x}^k - \alpha_{k} \nabla {\cal L}_{\alpha_k}({\bf x}^k)$, it follows that
\begin{align}
\label{Eq:Suffdescent-decrease1}
 \langle \nabla {\cal L}_{\alpha_{k}}({\bf x}^k), {\bf x}^{k+1} - {\bf x}^{k}\rangle
 =-\frac{\|{\bf x}^{k+1} - {\bf x}^{k}\|^2}{\alpha_k}.
\end{align}
Since  $\sum_{i=1}^n \nabla f_i({\bf x}_{(i)})$ is $L_f$-Lipschitz, $\nabla{\cal L}_{\alpha_k}$
is Lipschitz with the constant $L_k \triangleq L_f + \alpha_k^{-1}\lambda_{\max}(I-W) = L_f + \alpha_k^{-1}(1-\lambda_n(W))$, implying
\begin{align}
\label{Eq:Suffdescent-decrease2}
&{\cal L}_{\alpha_k}({\bf x}^{k+1})\\
&\le {\cal L}_{\alpha_k}({\bf x}^k) + \langle \nabla {\cal L}_{\alpha_k}({\bf x}^k), {\bf x}^{k+1} - {\bf x}^{k}\rangle
+ \frac{L_k}{2} \|{\bf x}^{k+1} - {\bf x}^{k}\|^2 \nonumber\\
&= {\cal L}_{\alpha_k}({\bf x}^k) - \frac{1}{2}\big(\alpha_k^{-1}(1+\lambda_n(W))-L_f\big)\|{\bf x}^{k+1}-{\bf x}^k\|^2. \nonumber
\end{align}
Moreover,
\begin{align}
\label{Eq:Suffdescent-decrease3}
&{\cal L}_{\alpha_{k+1}}({\bf x}^{k+1}) \nonumber\\
&= {\cal L}_{\alpha_k}({\bf x}^{k+1}) + \frac{1}{2}(\alpha_{k+1}^{-1}-\alpha_k^{-1})\|{\bf x}^{k+1}\|^2_{I-W}.
\end{align}
Combining \eqref{Eq:Suffdescent-decrease2} and \eqref{Eq:Suffdescent-decrease3} yields \eqref{Eq:Suffdescent-decrease}.

\textbf{(b) Convergence of objective sequence:}
By Lemma \ref{Lemm:bk} and Lemma \ref{Lemm:supermartingale}, \eqref{Eq:Suffdescent-decrease} yields the convergence of $\{{\cal L}_{\alpha_k}({\bf x}^k)\}$ and
\begin{align}
\label{Eq:l2-sum}
\sum_{k=0}^{\infty} \big(\alpha_k^{-1}(1+\lambda_n(W))-L_f\big)\|{\bf x}^{k+1}-{\bf x}^k\|^2 <\infty
\end{align}
which implies that $\|{\bf x}^{k+1}-{\bf x}^k\|^2$ converges to 0 at the rate of $o(k^{-\epsilon})$ and $\{{\bf x}^k\}$ is asymptotic regular¡£
Moreover, notice that
\begin{align*}
\alpha_k^{-1}\|{\bf x}^{k}\|^2_{I-W}
&= \alpha_k^{-1}\|{\bf x}^{k} - {\bar {\bf x}}^{k}\|^2_{I-W} \\
&\leq(1-\lambda_n(W))L_f(k+1)^{\epsilon}\|{\bf x}^{k} - {\bar {\bf x}}^{k}\|^2.
\end{align*}
By Proposition \ref{Propos:asympconsensus}, the term $\alpha_k^{-1}\|{\bf x}^{k}\|^2_{I-W}$ converges to 0 as $k\rightarrow \infty$.
As a consequence,
\begin{align*}
\lim_{k\rightarrow \infty} {\bf 1}^T{\bf f}({\bf x}^k)
&= \lim_{k\rightarrow \infty} \left({\cal L}_{\alpha_k}({\bf x}^k) - \frac{\|{\bf x}^{k}\|^2_{I-W}}{2\alpha_k}\right) \\
&= \lim_{k\rightarrow \infty} {\cal L}_{\alpha_k}({\bf x}^k).
\end{align*}

\textbf{(c) Convergence to a stationary point:}
Let $\bar{\nabla}{\bf f}({\bf x}^k) \triangleq \frac{1}{n}{\bf 1}{\bf 1}^T\nabla {\bf f}({\bf x}^k)$.
By the specific form \eqref{eq:decralpha} of $\alpha_k$, we have
\begin{align}
\label{Eq:alpha-1}
&\alpha_k^{-1}(1+\lambda_n(W))-L_f \nonumber\\
&= \alpha_k^{-1}(1+\lambda_n(W)-L_f \alpha_k) \nonumber\\
&\geq \alpha_k^{-1}\left(1+\lambda_n(W) - \frac{1}{(k_0+1)^{\epsilon}}\right)
\end{align}
for all $k>k_0$, where $k_0 = \left[(1+\lambda_n(W))^{-\frac{1}{\epsilon}}\right]$, i.e., the integer part of $(1+\lambda_n(W))^{-\frac{1}{\epsilon}}$.
Note that
\begin{align}
\label{Eq:avgx-x}
\|\bar{\bf x}^{k+1} - \bar{\bf x}^{k}\|
&= \|\frac{1}{n}{\bf 1}{\bf 1}^T ({\bf x}^{k+1} - {\bf x}^{k})\| \nonumber\\
&\leq \|{\bf x}^{k+1} - {\bf x}^{k}\|.
\end{align}
Thus, \eqref{Eq:l2-sum}, \eqref{Eq:alpha-1} and \eqref{Eq:avgx-x} yield
\begin{align}
\label{Eq:l2-sum-avg}
\sum_{k=0}^{\infty} \alpha_k^{-1}\|\bar{{\bf x}}^{k+1}-\bar{{\bf x}}^k\|^2 <\infty.
\end{align}
By the iterate \eqref{Eq:DGD} of DGD, we have
\begin{align}
\label{Eq:avg-iter}
\bar{{\bf x}}^{k+1}-\bar{{\bf x}}^k = -\alpha_k \bar{\nabla}{\bf f}({\bf x}^k).
\end{align}
Plugging \eqref{Eq:avg-iter} into \eqref{Eq:l2-sum-avg} yields
\begin{align}
\label{Eq:l2-sum-grad}
\sum_{k=0}^{\infty} \alpha_k \|\bar{\nabla}{\bf f}({\bf x}^k)\|^2 < \infty.
\end{align}
Moreover,
\begin{align}
&|\|{\bar \nabla}{\bf f}({\bf x}^{k+1})\|^2-\|{\bar \nabla}{\bf f}({\bf x}^k)\|^2| \nonumber\\
&\leq \|{\bar \nabla}{\bf f}({\bf x}^{k+1})-{\bar \nabla}{\bf f}({\bf x}^{k})\| \cdot (\|{\bar \nabla}{\bf f}({\bf x}^{k+1})\|+\|{\bar \nabla}{\bf f}({\bf x}^{k})\|) \nonumber\\
&\leq 2B\|{\bar \nabla}{\bf f}({\bf x}^{k+1})-{\bar \nabla}{\bf f}({\bf x}^{k})\| \nonumber\\
&\leq 2B\|{\nabla}{\bf f}({\bf x}^{k+1})-{\nabla}{\bf f}({\bf x}^{k})\| \nonumber\\
&\leq 2BL_f \|{\bf x}^{k+1}-{\bf x}^{k}\|, \label{Eq:diff-betak}
\end{align}
where the second inequality holds by the bounded gradient assumption (Assumption \ref{Assump:boundedgradient}), the third inequality holds by the specific form of ${\bar \nabla}{\bf f}({\bf x}^{k})$, and the last inequality holds by the Lipschitz continuity of ${\nabla}{\bf f}$.
Note that
\begin{align}
\label{Eq:xk+1-xk}
&\|{\bf x}^{k+1}-{\bf x}^{k}\| \nonumber\\
&= \|{\bf x}^{k+1}-\bar{\bf x}^{k+1}+\bar{\bf x}^{k+1}-\bar{\bf x}^{k}+\bar{\bf x}^{k}-{\bf x}^{k}\| \nonumber\\
&\leq \|{\bf x}^{k+1}-\bar{\bf x}^{k+1}\| + \|\bar{\bf x}^{k}-{\bf x}^{k}\| + \alpha_k \|\bar{\nabla}{\bf f}({\bf x}^k)\| \nonumber\\
&\lesssim \alpha_k,
\end{align}
where the first inequality holds for the triangle inequality and \eqref{Eq:avg-iter}, and the last inequality holds for Proposition \ref{Propos:asympconsensus} and the bounded assumption of $\nabla {\bf f}$.
Thus, \eqref{Eq:diff-betak} and \eqref{Eq:xk+1-xk} imply
\begin{align}
\label{Eq:diff-betak1}
|\|{\bar \nabla}{\bf f}({\bf x}^{k+1})\|^2-\|{\bar \nabla}{\bf f}({\bf x}^k)\|^2|\lesssim \alpha_k.
\end{align}
By the specific form \eqref{eq:decralpha} of $\alpha_k$, \eqref{Eq:l2-sum-grad}, \eqref{Eq:diff-betak1} and Lemma \ref{Lemm:betak}, it holds
\begin{align}
\label{Eq:conv-sumofgrad}
\lim_{k\rightarrow \infty} \|{\bar \nabla}{\bf f}({\bf x}^k)\|^2 =0.
\end{align}
As a consequence,
\begin{align}
\label{Eq:conv-sumofgrad1}
\lim_{k\rightarrow \infty} {\bf 1}^T{\nabla}{\bf f}({\bf x}^k) =0.
\end{align}


Furthermore, by the coercivity of $f_i$ for each $i$ and the convergence of $\{{\bf 1}^T{\bf f}({\bf x}^k)\}$, $\{{\bf x}^k\}$ is bounded. Therefore, there exists a convergent subsequence of $\{{\bf x}^k\}$.
Let ${\bf x}^*$ be any limit point of $\{{\bf x}^k\}$. By \eqref{Eq:conv-sumofgrad} and the continuity of $\nabla {\bf f}$, it holds
\[
{\bf 1}^T\nabla {\bf f}({\bf x}^*)=0.
\]
Moreover, by Proposition \ref{Propos:asympconsensus}, ${\bf x}^*$ is consensual. As a consequence, ${\bf x}^*$ is a stationary point of problem \eqref{Eq:consensusProblem}.

In addition, if ${\bf x}^*$ is isolated, then by the asymptotic regularity of $\{{\bf x}^k\}$ (Lemma \ref{Lemm:Asympregular}), $\{{\bf x}^k\}$ converges to ${\bf x}^*$ \cite{Ostrowski1973}.

\end{proof}

\subsection{Proof for Proposition \ref{Propos:convergrate-dgd}}

To prove Proposition \ref{Propos:convergrate-dgd}, we still need the following lemmas.

\begin{Lemm}[Accumulated consensus of iterates]
\label{Lemm:accum-cons}
Under conditions of Proposition \ref{Propos:asympconsensus}, we have
\begin{align}
\label{Eq:accum-cons}
\sum_{k=0}^K \alpha_k \|{\bf x}^{k+1}-\bar{\bf x}^{k+1}\| \leq D_1 + D_2 \sum_{k=0}^K \alpha_k^2,
\end{align}
where $D_1 = \frac{C\|{\bf x}^0\|\zeta}{2(1-\zeta)}$, $D_2 = C\left(\frac{\|{\bf x}^0\|\zeta}{2}+\frac{B}{1-\zeta}\right)$, and $B$ is specified in Assumption \ref{Assump:boundedgradient}.
\end{Lemm}

\begin{proof}
By \eqref{Eq:asympconsens2},
\begin{align}
\label{Eq:accum-cons1}
&\sum_{k=0}^K \alpha_k \|{\bf x}^{k+1}-\bar{\bf x}^{k+1}\| \leq C\|{\bf x}^0\|\zeta\sum_{k=0}^K \alpha_k \zeta^k \nonumber\\
& + CB\sum_{k=0}^K \sum_{j=0}^k \zeta^{k-j}\alpha_k \alpha_j.
\end{align}
In the following, we estimate these two terms in the right-hand side of \eqref{Eq:accum-cons1}, respectively.
Note that
\begin{align}
\sum_{k=0}^K \alpha_k \zeta^k
&\leq \frac{1}{2}\sum_{k=0}^K \alpha_k^2 + \frac{1}{2}\sum_{k=0}^K \zeta^{2k} \nonumber\\
&\leq \frac{1}{2(1-\zeta)} + \frac{1}{2}\sum_{k=0}^K \alpha_k^2, \label{Eq:accum-cons-term1}
\end{align}
and
\begin{align}
&\sum_{k=0}^K \sum_{j=0}^k \zeta^{k-j} \alpha_k \alpha_j
\leq \frac{1}{2}\sum_{k=0}^K \sum_{j=0}^k \zeta^{k-j} (\alpha_k^2 + \alpha_j^2) \nonumber\\
&=\frac{1}{2}\sum_{k=0}^K \alpha_k^2 \sum_{j=0}^k \zeta^{k-j} + \frac{1}{2}\sum_{j=0}^K \alpha_j^2 \sum_{k=j}^K \zeta^{k-j}\nonumber\\
&\leq \frac{1}{1-\zeta}\sum_{k=0}^K \alpha_k^2. \label{Eq:accum-cons-term2}
\end{align}
Plugging \eqref{Eq:accum-cons-term1} and \eqref{Eq:accum-cons-term2} into \eqref{Eq:accum-cons1} yields \eqref{Eq:accum-cons}.
\end{proof}

Besides Lemma \ref{Lemm:accum-cons}, we also need the following two lemmas, which have appeared in the literature (cf. \cite{Chen2012b}).

\begin{Lemm}[\cite{Chen2012b}]
\label{Lemm:rate-sum-gammak}
Let $\gamma_k = \frac{1}{k^{\epsilon}}$ for some $0<\epsilon\leq 1$.
Then the following hold
\begin{enumerate}
\item[(a)] if $0<\epsilon<1/2$,
$\frac{1}{\sum_{k=1}^K \gamma_k} \leq \frac{1-\epsilon}{K^{1-\epsilon}-1} = {\cal O}(\frac{1}{K^{1-\epsilon}}),$
\[\frac{\sum_{k=1}^K \gamma_k^2}{\sum_{k=1}^K \gamma_k} \leq \frac{1-\epsilon}{1-2\epsilon}\cdot \frac{K^{1-2\epsilon-2\epsilon}}{K^{1-\epsilon}-1} = {\cal O}(\frac{1}{K^{\epsilon}}).\]

\item[(b)] if $\epsilon=1/2,$
$\frac{1}{\sum_{k=1}^K \gamma_k} \leq \frac{1-\epsilon}{K^{1-\epsilon}-1} = {\cal O}(\frac{1}{\sqrt{K}}),$
\[\frac{\sum_{k=1}^K \gamma_k^2}{\sum_{k=1}^K \gamma_k} \leq \frac{1+\ln K}{2(K^{1/2}-1)}={\cal O}(\frac{\ln K}{\sqrt{K}}).\]

\item[(c)] if $1/2 < \epsilon<1$,
$\frac{1}{\sum_{k=1}^K \gamma_k} \leq \frac{1-\epsilon}{K^{1-\epsilon}-1} = {\cal O}(\frac{1}{K^{1-\epsilon}}),$
\[\frac{\sum_{k=1}^K \gamma_k^2}{\sum_{k=1}^K \gamma_k} \leq \frac{1-\epsilon}{2\epsilon-1}\cdot \frac{2\epsilon - 1/K^{2\epsilon-1}}{K^{1-\epsilon}-1} = {\cal O}(\frac{1}{K^{1-\epsilon}}).\]

\item[(d)] if $\epsilon=1$,
$\frac{1}{\sum_{k=1}^K \gamma_k} \leq \frac{1}{\ln K} = {\cal O}(\frac{1}{\ln K}),$
\[\frac{\sum_{k=1}^K \gamma_k^2}{\sum_{k=1}^K \gamma_k} \leq \frac{1-1/K}{\ln(K+1)-\ln 2} = {\cal O}(\frac{1}{\ln K}).\]
\end{enumerate}
\end{Lemm}

\begin{Lemm}(\cite[Proposition 3]{Chen2012b})
\label{Lemm:cvx-lip}
Let $h: \mathbb{R}^d \rightarrow \mathbb{R}$ be a continuously differentiable function whose gradient is Lipschitz continuous with  constant $L_h$. Then for any $x,y,u \in\mathbb{R}^p,$
\[ h(u)\geq h(x)+\langle \nabla h(y),u-x\rangle -\frac{L_h}{2}\|x-y\|^2.\]
\end{Lemm}

\begin{proof}[\bf Proof of Proposition \ref{Propos:convergrate-dgd}]
To prove this proposition, we first develop the following inequality,
\begin{align}
\label{Eq:rate-dgd1}
{\cal L}_{\alpha_k}({\bf x}^{k+1}) - {\cal L}_{\alpha_k}({\bf u}) \leq \frac{1}{2\alpha_k}(\|{\bf x}^k - {\bf u}\|^2 - \|{\bf x}^{k+1} - {\bf u}\|^2)
\end{align}
for any ${\bf u}\in \mathbb{R}^{n\times p}.$
By Lemma \ref{Lemm:cvx-lip}, we have
\begin{align}
&{\cal L}_{\alpha_k}({\bf u}) \geq {\cal L}_{\alpha_k}({\bf x}^{k+1}) \label{Eq:rate-dgd2}\\
&+ \langle \nabla {\cal L}_{\alpha_k}({\bf x}^k), {\bf u}-{\bf x}^{k+1}\rangle -\frac{L^*}{2} \|{\bf x}^{k+1}-{\bf x}^k\|^2, \nonumber
\end{align}
where $L^* = L_f + \alpha_k^{-1}(1-\lambda_n(W))$, and by \eqref{Eq:DGDequiv1}, we have $\nabla {\cal L}_{\alpha_k}({\bf x}^k) = \alpha_k^{-1}({\bf x}^k - {\bf x}^{k+1})$.
Then \eqref{Eq:rate-dgd2} implies
\begin{align}
&{\cal L}_{\alpha_k}({\bf u}) \geq {\cal L}_{\alpha_k}({\bf x}^{k+1}) \label{Eq:rate-dgd3}\\
&+ \alpha_k^{-1} \langle {\bf x}^k - {\bf x}^{k+1}, {\bf u}-{\bf x}^{k+1}\rangle -\frac{L^*}{2} \|{\bf x}^{k+1}-{\bf x}^k\|^2.\nonumber
\end{align}
Note that the specific form of $\alpha_k (=\frac{1}{L_f(k+1)^\epsilon})$, there exists an integer $k_0>0$ such that $L^* \leq \alpha_k^{-1}$ for all $k>k_0$. Actually, for the simplicity of the proof, we can take $\alpha_k < \frac{\lambda_n(W)}{L_f}$ starting from the first step so that $L^* \leq \alpha_k^{-1}$ holds from the initial step. Thus, \eqref{Eq:rate-dgd3} implies
\begin{align}
&{\cal L}_{\alpha_k}({\bf u}) \geq {\cal L}_{\alpha_k}({\bf x}^{k+1}) \label{Eq:rate-dgd4}\\
&+ \alpha_k^{-1} \langle {\bf x}^k - {\bf x}^{k+1}, {\bf u}-{\bf x}^{k+1}\rangle -\frac{1}{2\alpha_k} \|{\bf x}^{k+1}-{\bf x}^k\|^2.\nonumber
\end{align}
Recall that for any two vectors $a$ and $b$, it holds $2\langle a, b\rangle -\|a\|^2 = \|b\|^2-\|a-b\|^2$. Therefore,
\[{\cal L}_{\alpha_k}({\bf u}) \geq {\cal L}_{\alpha_k}({\bf x}^{k+1})+\frac{1}{2\alpha_k} (\|{\bf u}-{\bf x}^{k+1}\|^2-\|{\bf u}-{\bf x}^{k}\|^2).\]
As a consequence, we get the basic inequality \eqref{Eq:rate-dgd1}.

Note that the optimal solution ${\bf x}_{\mathrm{opt}}$ is consensual and thus, $\|{\bf x}_{\mathrm{opt}}\|^2_{I-W}=0$. Therefore, ${\cal L}_{\alpha_k}({\bf x}_{\mathrm{opt}}) = \bar{f}({\bf x}_{\mathrm{opt}}) = f_{\mathrm{opt}}$. By \eqref{Eq:rate-dgd1}, we have
\begin{align*}
&\alpha_k\left({\cal L}_{\alpha_k}({\bf x}^{k+1})-f_{\mathrm{opt}}\right) \\
&\leq (\|{\bf x}^k - {\bf x}_{\mathrm{opt}}\|^2 - \|{\bf x}^{k+1}-{\bf x}_{\mathrm{opt}}\|^2)/2.
\end{align*}
Summing the above inequality over $k=0,1,\ldots,K$ yields
\begin{align}
\label{Eq:rate-dgd4}
\sum_{k=0}^K \alpha_k({\cal L}_{\alpha_k}({\bf x}^{k+1})-f_{\mathrm{opt}}) \leq \|{\bf x}^0 - {\bf x}_{\mathrm{opt}}\|^2/2.
\end{align}
Moreover, noting that ${\cal L}_{\alpha_k}(\bar{\bf x}^{k+1}) = {\bar f}(\bar{\bf x}^{k+1})$ and by the convexity of ${\cal L}_{\alpha_k}$,
\begin{align}
{\cal L}_{\alpha_k}({\bf x}^{k+1})
&\geq {\bar f}(\bar{\bf x}^{k+1}) + \langle \nabla {\cal L}_{\alpha_k}({\bf x}^{k+1}), {\bf x}^{k+1}-\bar{\bf x}^{k+1}\rangle \nonumber\\
&\geq {\bar f}(\bar{\bf x}^{k+1})- B\|{\bf x}^{k+1}-\bar{\bf x}^{k+1}\|, \label{Eq:rate-dgd5}
\end{align}
where the second inequality holds by the bounded assumption of gradient (cf. Assumption \ref{Assump:boundedgradient}).
Plugging \eqref{Eq:rate-dgd5} into \eqref{Eq:rate-dgd4} yields
\begin{align}
&\sum_{k=0}^K \alpha_k ({\bar f}(\bar{\bf x}^{k+1})-f_{\mathrm{opt}}) \label{Eq:rate-dgd6}\\
&\leq \frac{1}{2}\|{\bf x}^0 - {\bf x}_{\mathrm{opt}}\|^2 + B\sum_{k=0}^K \alpha_k \|{\bf x}^{k+1} - \bar{\bf x}^{k+1}\|. \nonumber
\end{align}
By the definition of $\bar{f}^K$ \eqref{Eq:def-bar-fk}, then \eqref{Eq:rate-dgd6} implies
\begin{align}
&(\bar{f}^K -f_{\mathrm{opt}})\sum_{k=0}^K \alpha_k \label{Eq:rate-dgd7}\\
&\leq \frac{1}{2}\|{\bf x}^0 - {\bf x}_{\mathrm{opt}}\|^2 + B\sum_{k=0}^K \alpha_k \|{\bf x}^{k+1} - \bar{\bf x}^{k+1}\| \nonumber\\
&\leq D_3 + D_4 \sum_{k=0}^K \alpha_k^2,
\end{align}
where $D_3 = \frac{1}{2}\|{\bf x}^0 - {\bf x}_{\mathrm{opt}}\|^2 + BD_1$, $D_4 = B D_2$, $D_1$ and $D_2$ are specified in Lemma \ref{Lemm:accum-cons}, and the second inequality holds for Lemma \ref{Lemm:accum-cons}.
As a consequence,
\begin{align}
\label{Eq:rate-dgd8}
\bar{f}^K -f_{\mathrm{opt}} \leq \frac{D_3 + D_4 \sum_{k=0}^K \alpha_k^2}{\sum_{k=0}^K \alpha_k}.
\end{align}
Furthermore, by Lemma \ref{Lemm:rate-sum-gammak}, we get the claims of this proposition.
\end{proof}

\subsection{Proofs for Theorem \ref{Thm:Globalconverg-PGDGD} and Proposition \ref{Propos:conv-rate-prox-DGD}}
In order to prove Theorem \ref{Thm:Globalconverg-PGDGD}, we need the following lemmas.

\begin{Lemm}[Sufficient descent of $\{\hat{\cal L}_{\alpha}({\bf x}^k)\}$]
\label{Lemm:Suffdescent-PGDGD}
Let Assumptions \ref{Assump:MixMat} and \ref{Assump:objective-composite} hold. Results are given in two cases below:
\begin{enumerate}
\item[C1:] $r_i$'s are convex. Set $0<\alpha<\frac{1+\lambda_n(W)}{L_f}$.
\begin{align}
&\hat{\cal L}_{\alpha}({\bf x}^{k+1}) \leq \hat{\cal L}_{\alpha}({\bf x}^k) \label{Eq:Suffdescent-PGDGD0}\\
&- \frac{1}{2}\big(\alpha^{-1}(1+\lambda_n(W))-L_f\big) \|{\bf x}^{k+1} - {\bf x}^k\|^2, \forall k \in \mathbb{N}.\nonumber
\end{align}
\item[C2:] $r_i$'s are not necessarily convex (in this case, we assume $\lambda_n(W)>0$). Set $0<\alpha<\frac{\lambda_n(W)}{L_f}$.
\begin{align}
&\hat{\cal L}_{\alpha}({\bf x}^{k+1}) \leq \hat{\cal L}_{\alpha}({\bf x}^k) \label{Eq:Suffdescent-PGDGD1}\\
&- \frac{1}{2}\big(\alpha^{-1}\lambda_n(W)-L_f\big) \|{\bf x}^{k+1} - {\bf x}^k\|^2, \forall k \in \mathbb{N}.\nonumber
\end{align}
\end{enumerate}
\end{Lemm}
\begin{proof}
Recall from Lemma \ref{Lemm:Suffdescent} that $\nabla{\cal L}_{\alpha}({\bf x})$ is $L^* $-Lipschitz continuous for $L^* = L_f + \alpha^{-1}(1-\lambda_n(W))$, and thus
\begin{align}
&\hat{\cal L}_{\alpha}({\bf x}^{k+1})-\hat{\cal L}_{\alpha}({\bf x}^k)\nonumber\\
&= {\cal L}_{\alpha}({\bf x}^{k+1}) - {\cal L}_{\alpha}({\bf x}^k) + r({\bf x}^{k+1})- r({\bf x}^k)\nonumber\\
&\le \langle\nabla{\cal L}_{\alpha}({\bf x^k}), {\bf x}^{k+1}-{\bf x}^{k}\rangle+\frac{L^*}{2}\|{\bf x}^{k+1}-{\bf x}^{k}\|^2 \nonumber\\
&+ r({\bf x}^{k+1})- r({\bf x}^k).  \label{eq:proxdesc}
\end{align}

{\bf C1}: From the convexity of $r$, \eqref{eq:prox_def}, and \eqref{Eq:proximalDGDequiv1}, it follows that
$$ 0=  \xi^{k+1} +\frac{1}{\alpha}\big({\bf x}^{k+1}-{\bf x}^{k} + \alpha\nabla{\cal L}_{\alpha}({\bf x^k})\big),\  \xi^{k+1}\in\partial r({\bf x}^{k+1}).$$
This and the convexity of $r$ further give us
\begin{align*}
&r({\bf x}^{k+1})- r({\bf x}^k)\le \langle \xi^{k+1}, {\bf x}^{k+1}-{\bf x}^{k}\rangle \nonumber\\
&=-\frac{1}{\alpha}\|{\bf x}^{k+1}-{\bf x}^{k}\|^2-\langle\nabla{\cal L}_{\alpha}({\bf x^k}), {\bf x}^{k+1}-{\bf x}^{k}\rangle.
\end{align*}
Substituting this inequality into the inequality \eqref{eq:proxdesc} and then expanding $L^*= L_f + \alpha^{-1}(1-\lambda_n(W))$ yield
\begin{align*}
&\hat{\cal L}_{\alpha}({\bf x}^{k+1})-\hat{\cal L}_{\alpha}({\bf x}^k) \le -\big(\frac{1}{\alpha}-\frac{L^*}{2}\big)\|{\bf x}^{k+1}-{\bf x}^{k}\|^2\\
&=- \frac{1}{2}\big(\alpha^{-1}(1+\lambda_n(W))-L_f\big) \|{\bf x}^{k+1}-{\bf x}^{k}\|^2.
\end{align*}
Sufficient descent requires the last term to be negative, thus $0<\alpha< \frac{1+\lambda_n(W)}{L_f}$.

{\bf C2}: From \eqref{eq:prox_def} and \eqref{Eq:proximalDGDequiv1}, it follows that the function
$r({\bf u})+\frac{\|{\bf u}-({\bf x}^k - \alpha\nabla{\cal L}_{\alpha}({\bf x^k}))\|^2}{2\alpha}$ reaches its minimum at ${\bf u}={\bf x}^{k+1}$. Comparing the values of this function at ${\bf x}^{k+1}$ and ${\bf x}^{k}$ yields
\begin{align*}
 &r({\bf x}^{k+1})- r({\bf x}^k)\le \frac{1}{2\alpha}\|{\bf x}^k -({\bf x}^k - \alpha\nabla{\cal L}_{\alpha}({\bf x^k}))\|^2\\
 &-\frac{1}{2\alpha}\|{\bf x}^{k+1}-({\bf x}^k - \alpha\nabla{\cal L}_{\alpha}({\bf x^k}))\|^2\\
&=-\frac{1}{2\alpha}\|{\bf x}^{k+1}-{\bf x}^k\|^2 -\langle \nabla{\cal L}_{\alpha}({\bf x^k}), {\bf x}^{k+1}-{\bf x}^k\rangle.
\end{align*}
Substituting this inequality into \eqref{eq:proxdesc} and expanding $L^*$ yield
\begin{align*}
\hat{\cal L}_{\alpha}({\bf x}^{k+1})-\hat{\cal L}_{\alpha}({\bf x}^k)&\le -\big(\frac{1}{2\alpha}-\frac{L^*}{2}\big)\|{\bf x}^{k+1}-{\bf x}^{k}\|^2\\
&=- \frac{1}{2}\big(\alpha^{-1}\lambda_n(W)-L_f\big) \|{\bf x}^{k+1}-{\bf x}^{k}\|^2.
\end{align*}
Hence, sufficient descent requires $0<\alpha< \frac{\lambda_n(W)}{L_f}$.
\end{proof}

\begin{Lemm}[Boundedness]
\label{Lemm:Lowerbouded-composite}
Under the conditions of Lemma \ref{Lemm:Suffdescent-PGDGD}, the sequence $\{\hat{\cal L}_{\alpha}({\bf x}^k)\}$ is lower bounded, and the sequence $\{{\bf x}^k\}$ is bounded.
\end{Lemm}
\begin{proof}
The lower boundedness of $\{\hat{\cal L}_{\alpha}({\bf x}^k)\}$ is due to Assumption \ref{Assump:objective-composite} Part (2).

By Lemma \ref{Lemm:Suffdescent-PGDGD} and under a proper step size, $\hat{\cal L}_{\alpha}({\bf x}^k)$ is nonincreasing and upper bounded by $\hat{\cal L}_{\alpha}({\bf x}^0)$.
Hence, $\sum_{i=1}^n (f_i({\bf x}_{(i)}^k) + r_i({\bf x}_{(i)}^k))$ is upper bounded by $\hat{\cal L}_{\alpha}({\bf x}^0)$. Consequently, $\{{\bf x}^k\}$ is bounded due to the coercivity of each $f_i +r_i$ (see Assumption \ref{Assump:objective-composite} Part (2)).
\end{proof}

\begin{Lemm}[Bounded subgradient]
\label{Lemm:boundedsubgradient-PGDGD}
 Let $\partial \hat{\cal L}_{\alpha}({\bf x}^{k+1})$ denote the (limiting) subdifferential of $\hat{\cal L}_{\alpha}$, which is assumed to exist for all $k\in\mathbb{N}$. Then, there exists ${\bf g}^{k+1} \in \partial \hat{\cal L}_{\alpha}({\bf x}^{k+1})$ such that
\[\|{\bf g}^{k+1}\| \leq (\alpha^{-1}(2-\lambda_n(W))+L_f)\|{\bf x}^{k+1} - {\bf x}^k\|.\]
\end{Lemm}
\begin{proof}
By the iterate \eqref{Eq:proximalDGDequiv1}, the following optimality condition holds
\begin{align}
\label{Eq:optcond-PGDGD}
0\in \alpha^{-1}({\bf x}^{k+1} - {\bf x}^k + \alpha \nabla {\cal L}_{\alpha}({\bf x}^k)) + \partial r({\bf x}^{k+1}),
\end{align}
where $\partial r({\bf x}^{k+1})$ denotes the (limiting) subdifferential of $r$ at ${\bf x}^{k+1}$. For any $\xi^{k+1} \in \partial r({\bf x}^{k+1})$,
it follows from \eqref{Eq:optcond-PGDGD} that
\begin{align*}
&\nabla {\cal L}_{\alpha}({\bf x}^{k+1}) + \xi^{k+1} \\
&= \alpha^{-1}({\bf x}^k-{\bf x}^{k+1}) + (\nabla {\cal L}_{\alpha}({\bf x}^{k+1}) - \nabla {\cal L}_{\alpha}({\bf x}^{k})),
\end{align*}
which immediate yields
\begin{align*}
&\|\nabla {\cal L}_{\alpha}({\bf x}^{k+1}) + \xi^{k+1}\|\\
&\leq \alpha^{-1} \|{\bf x}^{k+1}-{\bf x}^k\| + \|\nabla {\cal L}_{\alpha}({\bf x}^{k+1}) - \nabla {\cal L}_{\alpha}({\bf x}^{k})\| \\
&\leq (\alpha^{-1}+L^*) \|{\bf x}^{k+1}-{\bf x}^k\|\\
\quad & \leq (\alpha^{-1}(2-\lambda_n(W))+L_f)\|{\bf x}^{k+1}-{\bf x}^k\|.
\end{align*}
Thus, then the claim of Lemma \ref{Lemm:boundedsubgradient-PGDGD} holds.
\end{proof}

Based on Lemmas \ref{Lemm:Suffdescent-PGDGD}--\ref{Lemm:boundedsubgradient-PGDGD}, we can easily prove Theorem \ref{Thm:Globalconverg-PGDGD} and Proposition \ref{Propos:conv-rate-prox-DGD}.

%
\begin{proof}[\bf Proof of Theorem \ref{Thm:Globalconverg-PGDGD}]
The proof of this theorem is similar to that of Theorem \ref{Thm:Globalconverg} and thus is omitted.
\end{proof}

%
%

\begin{proof}[\bf Proof of Proposition \ref{Propos:conv-rate-prox-DGD}]
The proof is similar to that of Proposition \ref{Propos:conv-rate}. We shall however note that in \eqref{Eq:keyineq1}, $a = \frac{1}{2}\big(\alpha^{-1}(1+\lambda_n(W))-L_f\big)$ if $r_i$'s are convex, while $a=\frac{1}{2}\big(\alpha^{-1}\lambda_n(W)-L_f\big)$ if $r_i$'s are not necessarily convex and $\lambda_n(W)>0$.
\end{proof}

\subsection{Proofs for Theorem \ref{Thm:hatLalphak} and Proposition \ref{Propos:asympconsensus-prox-dgd}}

Based on the iterate \eqref{Eq:proximalDGD} of Prox-DGD, we derive the following recursion of the iterates of Prox-DGD, which is similar to \eqref{Eq:recursion-xk}.

\begin{Lemm}[Recursion of $\{{\bf x}^k\}$]
\label{Lemm:recursion-xk-prox-dgd}
For any $k \in \mathbb{N}$,
\begin{align}
\label{Eq:recursion-xk-prox-dgd}
{\bf x}^k = W^k {\bf x}^0 - \sum_{j=0}^{k-1} \alpha_j W^{k-1-j}(\nabla {\bf f}({\bf x}^j)+{\xi}^{j+1}),
\end{align}
where
${\xi}^{j+1} \in \partial r({\bf x}^{j+1})$
is the one determined by the proximal operator \eqref{eq:prox_def}, for any $j=0,\ldots, k-1.$
\end{Lemm}

\begin{proof}
By the definition of the proximal operator \eqref{eq:prox_def}, the iterate \eqref{Eq:proximalDGD} implies
\begin{align}
\label{Eq:prox-dgd-equiv1}
{\bf x}^{k+1} + \alpha_k{\xi}^{k+1} = W{\bf x}^k - \alpha_k \nabla {\bf f}({\bf x}^k),
\end{align}
where ${\xi}^{k+1}\in \partial r({\bf x}^{k+1})$, and thus
\begin{align}
\label{Eq:prox-dgd-equiv2}
{\bf x}^{k+1} = W{\bf x}^k - \alpha_k (\nabla {\bf f}({\bf x}^k)+{\xi}^{k+1}).
\end{align}
By \eqref{Eq:prox-dgd-equiv2}, we can easily derive the recursion \eqref{Eq:recursion-xk-prox-dgd}.
\end{proof}

\begin{proof}[\bf Proof of Proposition \ref{Propos:asympconsensus-prox-dgd}]
The proof of this proposition is similar to that of Proposition \ref{Propos:asympconsensus}. It only needs to note that the subgradient term $\nabla {\bf f}({\bf x}^j)+{\xi}^{j+1}$ is uniformly bounded by the constant $\bar{B}$ for any $j$. Thus, we omit it here.
\end{proof}

To prove Theorem \ref{Thm:hatLalphak}, we still need the following lemmas.
\begin{Lemm}
\label{Lemm:Suffdescent-PGDGD-decreasing}
Let Assumptions \ref{Assump:MixMat} and \ref{Assump:objective-composite} hold. In Prox-DGD, use the step sizes \eqref{eq:decralpha}. Results are given in two cases below:
\begin{enumerate}
\item[C1:] $r_i$'s are convex. For any $k \in \mathbb{N}$,
\begin{align}
&\hat{\cal L}_{\alpha_{k+1}}({\bf x}^{k+1}) \leq \hat{\cal L}_{\alpha_k}({\bf x}^k) + \frac{1}{2}(\alpha_{k+1}^{-1}-\alpha_k^{-1})\|{\bf x}^{k+1}\|_{I-W}^2 \nonumber\\
&- \frac{1}{2}\big(\alpha_k^{-1}(1+\lambda_n(W))-L_f\big) \|{\bf x}^{k+1} - {\bf x}^k\|^2.\label{Eq:Suffdescent-PGDGD-decreasing0}
\end{align}
\item[C2:] $r_i$'s are not necessarily convex. For any $k \in \mathbb{N}$,
\begin{align}
&\hat{\cal L}_{\alpha_{k+1}}({\bf x}^{k+1}) \leq \hat{\cal L}_{\alpha_k}({\bf x}^k)  + \frac{1}{2}(\alpha_{k+1}^{-1}-\alpha_k^{-1})\|{\bf x}^{k+1}\|_{I-W}^2 \nonumber\\
&- \frac{1}{2}\big(\alpha_k^{-1}\lambda_n(W)-L_f\big) \|{\bf x}^{k+1} - {\bf x}^k\|^2.\label{Eq:Suffdescent-PGDGD-decreasing1}
\end{align}
\end{enumerate}
\end{Lemm}
\begin{proof}
The proof of this lemma is similar to that of Lemma \ref{Lemm:Suffdescent-PGDGD} via noting that
\begin{align*}
\hat{\cal L}_{\alpha_{k+1}}({\bf x}^{k+1})
&= \hat{\cal L}_{\alpha_{k}}({\bf x}^{k})+(\hat{\cal L}_{\alpha_{k+1}}({\bf x}^{k+1})-\hat{\cal L}_{\alpha_{k}}({\bf x}^{k+1}))\\
&+(\hat{\cal L}_{\alpha_{k}}({\bf x}^{k+1})-\hat{\cal L}_{\alpha_{k}}({\bf x}^{k})),
\end{align*}
and
\begin{align*}
\hat{\cal L}_{\alpha_{k+1}}({\bf x}^{k+1})-\hat{\cal L}_{\alpha_{k}}({\bf x}^{k+1}) = \frac{1}{2}(\alpha_{k+1}^{-1}-\alpha_k^{-1})\|{\bf x}^{k+1}\|_{I-W}^2.
\end{align*}
While the term $\hat{\cal L}_{\alpha_{k}}({\bf x}^{k+1})-\hat{\cal L}_{\alpha_{k}}({\bf x}^{k})$ can be estimated similarly by the proof of Lemma \ref{Lemm:Suffdescent-PGDGD}.
\end{proof}

\begin{Lemm}
\label{Lemm:distant-hatLalpha}
Let Assumptions \ref{Assump:MixMat}, \ref{Assump:objective-composite} and \ref{Assump:boundedcompositesubgradient} hold. In Prox-DGD, use the step sizes \eqref{eq:decralpha}. If further each $f_i$ and $r_i$ are convex, then for any ${\bf u}\in \mathbb{R}^{n\times p},$ we have
\begin{align*}
\hat{\cal L}_{\alpha_k}({\bf x}^{k+1}) - \hat{\cal L}_{\alpha_k}({\bf u}) \leq \frac{1}{2\alpha_k}(\|{\bf x}^k - {\bf u}\|^2 - \|{\bf x}^{k+1} - {\bf u}\|^2).
\end{align*}
\end{Lemm}

\begin{proof}
By Lemma \ref{Lemm:cvx-lip}, we have
\begin{align}
&{\cal L}_{\alpha_k}({\bf u}) \geq {\cal L}_{\alpha_k}({\bf x}^{k+1}) \label{Eq:rate-prox-dgd2}\\
&+ \langle \nabla {\cal L}_{\alpha_k}({\bf x}^k), {\bf u}-{\bf x}^{k+1}\rangle -\frac{L^*}{2} \|{\bf x}^{k+1}-{\bf x}^k\|^2, \nonumber
\end{align}
where $L^* = L_f + \alpha_k^{-1}(1-\lambda_n(W))$, and by the convexity of $r$, we have
\begin{align}
\label{Eq:rate-prox-dgd3}
r({\bf u}) \geq r({\bf x}^{k+1}) + \langle \xi^{k+1}, {\bf u}-{\bf x}^{k+1} \rangle,
\end{align}
where $\xi^{k+1} \in \partial r({\bf x}^{k+1})$ is the one determined by the proximal operator \eqref{eq:prox_def}. By \eqref{Eq:prox-dgd-equiv2}, it follows
\begin{align}
\label{Eq:rate-prox-dgd4}
\xi^{k+1} = \alpha_k^{-1}({\bf x}^k - {\bf x}^{k+1}) -\nabla {\cal L}_{\alpha_k}({\bf x}^k).
\end{align}
Plugging \eqref{Eq:rate-prox-dgd4} into \eqref{Eq:rate-prox-dgd3}, and then summing up \eqref{Eq:rate-prox-dgd2} and \eqref{Eq:rate-prox-dgd3} yield
\begin{align}
&\hat{\cal L}_{\alpha_k}({\bf u}) \geq \hat{\cal L}_{\alpha_k}({\bf x}^{k+1}) \label{Eq:rate-prox-dgd5}\\
&+ \alpha_k^{-1} \langle {\bf x}^k - {\bf x}^{k+1}, {\bf u}-{\bf x}^{k+1}\rangle -\frac{L^*}{2} \|{\bf x}^{k+1}-{\bf x}^k\|^2.\nonumber
\end{align}
Similar to the rest proof of the inequality \eqref{Eq:rate-dgd1}, we can prove this lemma based on \eqref{Eq:rate-prox-dgd5}.
\end{proof}
\begin{proof}[\bf Proof of Theorem \ref{Thm:hatLalphak}]
Based on Lemma \ref{Lemm:Suffdescent-PGDGD-decreasing} and Lemma \ref{Lemm:distant-hatLalpha}, we can proof Theorem \ref{Thm:hatLalphak}.
The proof of Theorem \ref{Thm:hatLalphak}(a)-(d) is similar to that of Theorem \ref{Thm:Lalphak}, where one minor difference is that \eqref{Eq:diff-betak} in the proof of Theorem \ref{Thm:Lalphak} should be
\begin{align}
&|\|{\bar \nabla}{\bf f}({\bf x}^{k+1}) + \bar{\xi}^{k+1}\|^2-\|{\bar \nabla}{\bf f}({\bf x}^k) + \bar{\xi}^{k}\|^2| \nonumber\\
&\leq \|({\bar \nabla}{\bf f}({\bf x}^{k+1}) + \bar{\xi}^{k+1}) -({\bar \nabla}{\bf f}({\bf x}^{k}) + \bar{\xi}^{k})\| \times \nonumber\\
&(\|{\bar \nabla}{\bf f}({\bf x}^{k+1}\|+\|{\bar \nabla}{\bf f}({\bf x}^{k}) + \bar{\xi}^{k}\|) \nonumber\\
&\leq 2\bar{B}\|({\bar \nabla}{\bf f}({\bf x}^{k+1}) + \bar{\xi}^{k+1}) -({\bar \nabla}{\bf f}({\bf x}^{k}) + \bar{\xi}^{k})\| \nonumber\\
&\leq 2\bar{B}\|({\nabla}{\bf f}({\bf x}^{k+1}) + {\xi}^{k+1}) -({\nabla}{\bf f}({\bf x}^{k}) + {\xi}^{k})\| \nonumber\\
&\leq 2\bar{B}(L_f+L_r) \|{\bf x}^{k+1}-{\bf x}^{k}\|,
\end{align}
where $\bar{\xi}^k \triangleq \frac{1}{n}{\bf 1}{\bf 1}^T \xi^k$, and the final inequality holds for the Lipschitz assumption on $\{\xi^k\}$ for large $k$ in Theorem \ref{Thm:hatLalphak}(c).

The proof of Theorem \ref{Thm:hatLalphak}(e) is very similar to that of Proposition \ref{Propos:convergrate-dgd}.
\end{proof}

\section{Conclusion}
\label{sc:conclusion}

In this paper, we study the convergence behavior of the algorithm DGD for smooth, possibly nonconvex consensus optimization. We consider both fixed and decreasing step sizes. When using a fixed step size, we show that the iterates of DGD converge to a stationary point of a Lyapunov function, which approximates to one of the original problem. Moreover, we estimate the bound between each local point and its global average, which is proportional to the step size and inversely proportional to the gap between the largest and the second largest magnitude eigenvalues of the mixing matrix. This motivate us to study the algorithm DGD with decreasing step sizes. When using decreasing step sizes, we show that the iterates of DGD reach consensus asymptotically at a sublinear rate and converge to a stationary point of the original problem. We also estimate the convergence rates of objective sequence in the convex setting using different diminishing step size strategies. Furthermore, we extend these convergence results to Prox-DGD designed for minimizing the sum of a differentiable function and a proximal function. Both functions can be nonconvex. If the proximal function is convex, a larger fixed step size is allowed. These results are obtained by applying both existing and new proof techniques. 

\section*{Acknowledgments}

The work of J. Zeng has been supported in part by the NSF grants (61603162, 11501440) and the Doctoral start-up foundation of Jiangxi Normal University.
The work of W. Yin has been supported in part by the NSF grant ECCS-1462398 and ONR grants N000141410683 and N000141210838.

\ifCLASSOPTIONcaptionsoff
  \newpage
\fi


\begin{thebibliography}{1}


\bibitem{Attouch-Bolte2009}
H.~Attouch, and J.~Bolte, \emph{On the convergence of the proximal algorithm for nonsmooth functions involving analytic features}, Math. Program., 116: 5-16, 2009.


\bibitem{Attouch2013}
H.~Attouch, J. Bolte and B. Svaiter,
\emph{Convergence of descent methods for semi-algebraic and tame problems: proximal algorithms, forward-backward splitting, and regularized Gauss-Seidel methods},
Math. Program., Ser. A, 137: 91-129, 2013.


\bibitem{Hinton-DL2015}
Y. Bengio, Y. LeCun, and G. Hinton, \emph{Deep Learning}, Nature. 521: 436-444, 2015.

\bibitem{Bianchi2013}
P. Bianchi and J. Jakubowicz, \emph{Convergence of a multi-agent projected stochastic gradient algorithm for nonconvex optimization}, IEEE Trans. Automatic Control, 58(2): 391-405, 2013.

\bibitem {Bianchi2013b}
P. Bianchi, G. Fort and W. Hachem, \emph{Performance of a distributed stochastic approximation algorithm}, IEEE Trans. Information Theory, 59(11): 7405-7418, 2013.


\bibitem {Bjornson-RS2012}
E. Bjornson, and E. Jorswieck, \emph{Optimal resource allocation in coordinated multi-cell systems},
Foundations and Trends in Communications and Information Theory, 9(2-3): 113-381, 2012.


\bibitem{bolte2007lojasiewicz}
J. Bolte, A. Daniilidis and A. Lewis, \emph{The {\L}ojasiewicz inequality for  nonsmooth subanalytic functions with applications to subgradient dynamical  systems}, SIAM Journal on Optimization, 17(4): 1205-1223, 2007.

\bibitem{Chen2012a}
A. Chen and A. Ozdaglar, \emph{A fast distributed proximal gradient method}, in Proc. 50th Allerton Conf. Commun., Control Comput., Moticello, IL, pp. 601-608, Oct. 2012.

\bibitem{Chang-InexactADMM2015}
T.~Chang, M.~Hong and X.~Wang, Multi-agent distributed optimization via inexact consensus ADMM, \emph{IEEE Trans. Signal Process.}, 63(2): 482-497, 2015.



\bibitem{Chen2012b}
A.~Chen,
\emph{Fast Distributed First-Order Methods}, Master's thesis, Department of Electrical Engineering and Computer Science, Massachusetts Institute of Technology, Cambridge, MA, 2012.


\bibitem{Chen-Lowerbound-Lq}
X.~Chen, F. Xu, and Y. Ye,
\emph{Lower bound theory of nonzero entries in solutions of $\ell_2-\ell_p$ minimization},
SIAM Journal of Scientific Computing, 32(5): 2832-2852, 2010.


\bibitem{Chow2016}
Y.T. Chow, T. Wu and W. Yin,
\emph{Cyclic Coordinate Update Algorithms for Fixed-Point Problems: Analysis and Applications.}
SIAM J. Sci. Comput.,
39(4): A1280-A1300, 2017.


\bibitem{Deng2013}
W. Deng, M. Lai, Z. Peng and W. Yin,
\emph{Parallel multi-block admm with o(1/k) convergence},
Journal of Scientific Computing, 71(2): 712-736, 2017.


\bibitem{Duarte-DCS2005}
M.F. Duarte, S. Sarvotham, D. Baron, M.B. Wakin, and R.G. Baraniuk,
\emph{Distributed Compressed Sensing of Jointly Sparse Signals},
Conference Record of the Thirty-Ninth Asilomar Conference on Signals, Systems and Computers, 1058-6393, 2005.


\bibitem{Scutari-Bigdata2015}
F. Facchinei, G. Scutari, and S. Sagratella,
\emph{Parallel selective algorithms for nonconvex big data optimization},
IEEE Transactions on Signal Processing, 63(7): 1874-1889, 2015.

\bibitem{Fan-SCAD2001}
J. Fan, and R. Li,
\emph{Variable selection via nonconcave penalized likelihood and its oracle properties},
Journal of the American Statistical Association: Theory and Method, 96(456): 1348-1360, 2001.


\bibitem{Hazan2016}
E. Hazan, K.Y. Levy and S. Shalev-Shwarz,
\emph{On graduated optimization for stochastic nonconvex problems},
In Proceedings of the 33rd International Conference on Machine Learning, New York, NY, USA, 2016.




\bibitem{Hong2016}
D. Hajinezhad, M. Hong and A. Garcia,
\emph{ZENTH: a zeroth-order distributed algorithm for multi-agent nonconvex optimization},
(Technical report).

\bibitem{Hong-zeroth-order2017}
D. Hajinezhad, M. Hong, and A. Garcia,
\emph{Zeroth order nonconvex multi-Agent optimization over networks},
arXiv:1710.09997, 2017.


\bibitem{Hong-RS2013}
M. Hong and Z.-Q. Luo, \emph{Signal processing and optimal resource allocation
for the interference channel,} in Library in Signal Processing.
New York: Academic Press, 2013, vol. 2, Communications and Radar
Signal Processing, ch. 8, pp. 409-462.



\bibitem{Hong-Prox-PDA2017}
M. Hong, D. Hajinezhad, and M.-M. Zhao,
\emph{Prox-PDA: the proximal primal-dual algorithm for fast distributed nonconvex optimization and learning over networks},
Proceedings of the 34th International Conference on Machine Learning, 70: 1529-1538, 2017.


\bibitem{Hong-PP-PDA2017}
D. Hajinezhad, and M. Hong,
\emph{Perturbed proximal primal dual algorithm for nonconvex nonsmooth optimization},
(Technique Report) \url{http://people.ece.umn.edu/~mhong/PProx_PDA.pdf},
 2017.




\bibitem{Hosseini16}
S.~Hosseini, A.~Chapman and M.~Mesbahi,
\emph{Online distributed optimization on dynamic networks},
IEEE T. Auto. Control, 61(11): 3545-3550, 2016.

\bibitem{Jakovetic-FastGradient2014}
D.~Jakovetic, J.~Xavier and J.~Moura,
\emph{Fast distributed gradient methods}, IEEE Trans. Automatic Control, 59: 1131-1146, 2014.



\bibitem{Kempe2003}
D. Kempe, A. Dobra and J. Gehrke, \emph{Gossip-based computation of aggregate information}, In Foundations of Computer Science, 2003. Proceedings 44th Annual IEEE Symposium on, 482-491, IEEE Computer Society, 2003.


\bibitem{Knopp1956}
K. Knopp, \emph{Infinite sequences and series}, Courier Corporation, 1956.


\bibitem{Wai2016b}
J. Lafond, H. Wai and E. Moulines,
\emph{D-FW: communication efficient distributed algorithms for high-dimensional sparse optimization}, ICASSP 2016.

\bibitem{Lan-AccSGD2016}
S. Ghadimi, and G. Lan,
\emph{Accelerated gradient methods for nonconvex nonlinear and stochastic programming},
Mathematical Programming, 156(1-2): 59-99, 2016.


\bibitem{Lee-DistRandProj2013}
S.~Lee and A.~Nedic, Distributed random projection algorithm for convex optimization, \emph{IEEE J. Sel. Topics Signal Process.}, 7(2): 221-229, 2013.

\bibitem{Ling-DCS2010}
Q.~Ling and Z.~Tian, Decentralized sparse signal recovery for compressive sleeping wireless sensor networks, \emph{IEEE Trans. Signal Process.}, 58(7): 3816-3827, 2010.

\bibitem{Ling-DLM2015}
Q. Ling, W. Shi, G. Wu, and A. Ribeiro,
\emph{DLM: Decentralized linearized alternating direction method of multipliers.}
IEEE Transactions on Signal Processing, 63(15): 4051-4064, Aug. 2015.


\bibitem{Liu-APSGD2015}
X. Lian, Y. Huang, Y. Li, and J. Liu,
\emph{Asynchronous parallel stochastic gradient for nonconvex optimization},
In Proceedings of the 28th International Conference on Neural Information Processing Systems (NIPS),
2: 2737-2745, 2015.

\bibitem{Liu-D-PSGD2017}
X. Lian, C. Zhang, H. Zhang, C.-J. Hsieh, W. Zhang, and J. Liu,
\emph{Can decentralized algorithms outperform centralized algorithms? A case study for decentralized parallel
stochastic gradient descent},
In Proceedings of the 30th International Conference on Neural Information Processing Systems (NIPS), 2017.



\bibitem{lojasiewicz1993geometrie}
S. {\L}ojasiewicz, \emph{Sur la g{\'e}om{\'e}trie semi-et sous-analytique},
Ann. Inst. Fourier (Grenoble) 43(5): 1575-1595, 1993.

\bibitem{Lorenzo2016a}
P.D. Lorenzo and G. Scutari, \emph{NEXT: in-network nonconvex optimization}, IEEE Trans. Signal and Information Processing over Network, 2(2): 120-136, 2016.

\bibitem{Lorenzo2016b}
P.D. Lorenzo and G. Scutari, \emph{Distributed nonconvex optimization over time-varying networks}, ICASSP 2016.

\bibitem{Matei-Subgradient2011}
I.~Matei and J.~Baras,
\emph{Performance evaluation of the consensus-based distributed subgradient method under random communication topologies}, IEEE J. Sel. Top. Signal Process., 5: 754-771, 2011.

\bibitem{McMahan-14}
H.~McMahan and M.~Streeter,
\emph{Delay-Tolerant algorithms for asynchronous distributed online learning}, In: Advances in Neural Information Processing Systems (NIPS), 2014.


\bibitem{Mateos-DistSparLinearReg2010}
G.~Mateos, J.~Bazerque and G.~Giannakis, Distributed sparse linear regression, \emph{IEEE Trans. Signal Process.}, 58(10): 5262-5276, 2010.


\bibitem{Nedic-Subgradient2009}
A.~Nedic and A.~Ozdaglar,
\emph{Distributed subgradient methods for multi-agent optimization}, IEEE Trans. Automatic Control, 54(1): 48-61, 2009.

\bibitem{Nedic-Subgradientpush2015}
A.~Nedic and A.~Olshevsky,
\emph{Distributed optimization over time-varying directed graphs}, IEEE Trans. Automatic Control, 60(3): 601-615, 2015.

\bibitem{Nevelson1973}
M.~Nevelson and R.Z.~Khasminskii,
\emph{Stochastic approximation and recursive estimation}, [translated from the Russian by Israel Program for Scientific Translations; translation edited by B. Silver]. Americal Mathematical Society, 1973.

\bibitem{Omidshafiei-RL2017}
S. Omidshafiei, J. Pazis, C. Amato,  J. P. How, and J. Vian,
\emph{Deep decentralized multi-task multi-agent reinforcement learning under partial observability},
arXiv:1703.06182, 2017.

\bibitem{Ostrowski1973}
A.M. Ostrowski,
\emph{Solution of equations in Euclidean and Banach spaces}, Academic Press, 1973.


\bibitem{Eldar2014}
Stacy Patterson, Yonina C. Eldar, and Idit Keidar,
\emph{Distributed compressed sensing for static and time-varying networks},
IEEE Transactions on Signal Processing, 62(19): 4931-4946, 2014.



\bibitem{Qu-Li2016}
G.~Qu and N.~Li,
\emph{Harnessing smoothness to accelerate distributed optimization},
IEEE Transactions on Control of Network Systems, 2017, Volume: PP, Issue: 99.


\bibitem{Raginsky11}
M.~Raginsky, N.~Kiarashi and R.~Willett,
\emph{Decentralized online convex programming with local information}, In: 2011 American Control Conference, San Francisco, CA, USA, 2011.

\bibitem{Ram-DSSP2010}
S. Ram, A. Nedic and V. Veeravalli,
\emph{Distributed stochastic subgradient projection algorithms for convex optimization},
J. Optim. Theory Appl., 147: 516-545, 2010.

\bibitem{Ravazzi-DIHT2015}
C. Ravazzi, S. M. Fosson, and E. Magli,
\emph{Distributed iterative thresholding for $\ell_0$/$\ell_1$-regularized linear inverse problems},
IEEE Transactions on Signal Processing, 61(4): 2081-2100, 2015.

\bibitem{Robbins-supermartingale1971}
H. Robbins and D. Siegmund,
\emph{A convergence theorem for nonnegative almost supermartingales and some applications},
in Proc. Optim, Methods Stat., 233-257, 1971.


\bibitem{Schizas-DADMM2008}
I.~Schizas, A.~Ribeiro and G.~Giannakis,
\emph{Consensus in ad hoc WSNs with noisy links-part I: Distributed estimation of deterministic signals},
IEEE Trans. Signal Process., 56(1): 350-364, 2008.

\bibitem{Shi-DADMM2014}
W.~Shi, Q.~Ling, K.~Yuan, G.~Wu and W.~Yin,
\emph{On the linear convergence of the ADMM in decentralized consensus optimization},
IEEE Trans. Signal Process., 62(7): 1750-1761, 2014.

\bibitem{Yin-EXTRA2015}
W.~Shi, Q.~Ling, G.~Wu and W.~Yin,
\emph{EXTRA: an exact first-order algorithm for decentralized consensus optimization}, SIAM Journal on Optimization, 25(2): 944-966, 2015.

\bibitem{Shi-PGEXTRA2015}
W.~Shi, Q.~Ling, G.~Wu and W.~Yin,
\emph{A Proximal Gradient Algorithm for Decentralized Composite Optimization}, IEEE Trans. Signal Processing, 63(22): 6013-6023, 2015.



\bibitem{Scutari2014}
G. Scutari, F. Facchinei, P. Song, D. P. Palomar, and J.-S. Pang
Decomposition by partial linearization: parallel optimization of multi-agent systems,
IEEE Transactions on Signal Processing, 62(3): 641-656, 2014.


\bibitem{Tatarenko2016conf}
T. Tatarenko and B. Touri,
\emph{On local analysis of distributed optimization}, ACC 2016.


\bibitem{Tatarenko2016arXiv}
T. Tatarenko and B. Touri,
\emph{Non-convex distributed optimization}, IEEE Trans. Automat. Contr., 62(8): 3744-3757, 2017.

\bibitem{Leung2013}
G. Tychogiorgos, A. Gkelias£¬and K. K. Leung
A non-convex distributed optimization framework and its application to wireless ad-hoc networks,
12(9): 4286-4296, 2013.

\bibitem{Tsitsiklis1986}
J. Tsitsiklis, D. Bertsekas and M. Athans,
\emph{Distributed asynchronous deterministic and stochastic gradient optimization algorithms}, IEEE Trans. Automatic Control, AC-32(9): 803-812, 1986.

\bibitem{Wai2015}
H. Wai, T. Chang and A. Scaglione,
\emph{A consensus-based decentralized algorithm for nonconvex optimization with application to dictionary learning}, ICASSP 2015.

\bibitem{Wai-Regression2015}
H. Wai, and A. Scaglione,
\emph{Consensus on state and time: decentralized regression with asynchronous sampling},
IEEE Transactions on Signal Processing, 63(11): 2972-2985, 2015.

\bibitem{Wai2016a}
H. Wai, A. Scaglione, J. Lafond and E. Moulines,
\emph{Decentralized Frank-Wolfe algorithm for convex and nonconvex problems}, IEEE Trans. Automatic Control, 62(11): 5522-5537, 2017.

\bibitem {Xu-Yin2013}
Y. Xu and W. Yin,
\emph{A block coordinate descent method for regularized multiconvex optimization with applications to nonnegative tensor
factorization and completion}, SIAM Journal on Imaging Sciences, 6: 1758-1789, 2013.

\bibitem{Yan-13}
F.~Yan, S.~Sundaram, S.~Vishwanathan and Y.~Qi,
\emph{Distributed autonomous online learning: regrets and intrinsic privacy-preserving properties},
IEEE T Knowledge and Data Engineering, 25(11): 2483--2493, 2013.

\bibitem{Yin-DGD2013}
K. Yuan, Q. Ling and W. Yin,
\emph{On the Convergence of Decentralized Gradient Descent}, SIAM Journal Optimization, 26(3): 1835-1854, 2016.

\bibitem{Zeng-IJT2016}
J. Zeng, S. Lin and Z. Xu,
\emph{Sparse regularization: convergence of iterative jumping thresholding algorithm}, IEEE Transactions on Signal Processing, 64(19): 5160-5118, 2016.

\bibitem{Zhang-MCP2010}
C.H. Zhang,
\emph{Nearly unbiased variable selection under minimax concave penalty}, Annals of Statistics, 38(2): 894-942, 2010.

\bibitem{Zhu2013}
M. Zhu and S. Martinez,
\emph{An approximate dual subgradient algorithm for multi-agent non-convex optimization}, IEEE Trans. Automatic Control, 58(6): 1534-1539, 2013.
\end{thebibliography}
\end{document}